\documentclass[11pt]{article}
\usepackage{amsmath,enumerate,amsfonts,amssymb,color,graphicx,amsthm}
\renewcommand{\raggedright}{\leftskip=0pt \rightskip=0pt plus 0cm}
\usepackage{multirow,caption,subcaption}
\usepackage{mathrsfs}
\usepackage{textcomp}
\usepackage{extarrows}
\usepackage{authblk}
\usepackage{titlesec}
\allowdisplaybreaks[4]
\usepackage{textcomp}
\usepackage{fancyhdr}
\usepackage{verbatim}
\usepackage{hyperref}
\usepackage{framed}
\usepackage[normalem]{ulem}
\usepackage{tikz}
\usetikzlibrary{patterns}
\usepackage[margin=1.25in]{geometry}
\usepackage{cite}
\titleformat{\section}{\sc\centering\Large}{\thesection}{1em}{}

\usepackage{soul}
\soulregister\cite7 
\soulregister\citep7 
\soulregister\citet7 
\soulregister\ref7 
\soulregister\pageref7 
\soulregister\eqref7
\hypersetup{
	colorlinks,
	citecolor=blue,
	filecolor=blue,
	linkcolor=blue,
	urlcolor=black
}

\numberwithin{equation}{section}
\newtheorem{thm}{Theorem}[section]
\newtheorem{lem}{Lemma}[section]
\newtheorem{prop}{Proposition}[section]
\newtheorem{defn}{Definition}[section]
\newtheorem{exa}{Example}[section]
\newtheorem{cor}{Corollary}[section]
\newtheorem{rem}{Remark}[section]
\newtheorem{claim}{Claim}[]

\renewcommand{\d}{\mathrm{d}}
\renewcommand{\H}{\mathcal{H}}
\renewcommand{\over}{\overline}
\newcommand{\R}{\mathbb{R}} 
\newcommand{\Rn}{\R^n}
\newcommand{\Rp}{\R^{n+1}_+}
\newcommand{\si}{\sigma}              \newcommand{\Sn}{\mathbb{S}^n}
\newcommand{\al}{\alpha}                \newcommand{\lam}{\lambda}
\newcommand{\om}{\Omega}                \newcommand{\pa}{\partial}
\newcommand{\var}{\varepsilon}           
\newcommand{\be}{\begin{equation}}      \newcommand{\ee}{\end{equation}}               
\newcommand{\Lam}{\Lambda}              \newcommand{\B}{\mathcal{B}}

\begin{document}
	
\title{
\bf	On the density   and multiplicity of solutions to the fractional Nirenberg problem  \bigskip
}

\author[a]{{\sc Zhongwei Tang}\thanks{The research was supported by National Science Foundation of China(12071036,12126306)}}
\author[a]{\sc Heming Wang}
\author[a]{\sc Ning Zhou}
\affil[a]{\rm School of Mathematical Sciences, Laboratory of Mathematics and Complex Systems, MOE,   \par Beijing Normal University,  Beijing, 100875, P.R. of China}

\date{}	
\maketitle

\begin{abstract}
	This paper is devoted to establishing some results on the density  and multiplicity of solutions to the  fractional Nirenberg problem which is equivalent to studying the  conformally invariant equation  $P_\si(v)=K v^{\frac{n+2\si}{n-2\si}}$  on the standard unit sphere $(\Sn,g_0)$ with   $\si\in (0,1)$ and $n\geq 2$, where $P_\si$ is the intertwining operator of order $2\si$ and $K$ is the prescribed curvature function.  More specifically, by  using the variational gluing  method,  refined  analysis of bubbling behavior,  extension formula,  as well as the blow up analysis arguments, we obtain the existence  of infinitely many multi-bump solutions. In particular, we show  the smooth curvature  functions of metrics conformal to $g_0$ are dense in the $C^{0}$ topology. Moreover, the related fractional Laplacian  equations  $(-\Delta)^{\si} u=K(x) u^{\frac{n+2\si}{n-2\si}}$ in $\Rn$, with $K(x)$ being asymptotically periodic in one of the variables, are also studied and infinitely many  solutions are obtained under natural flatness assumptions.	
\end{abstract}


\section{\sc Introduction}\label{sec:0}
\subsection{History}
The Nirenberg problem, raised by Nirenberg in the years 1969-1970, 
asks on the $n$-dimensional standard sphere $(\Sn,g_0)$ $(n\geq 2)$, if one
can find a conformally invariant metric $g$ such that the scalar curvature (Gauss curvature for $n=2$) of $g$ is equal to the given positive  function $K$. So the Nirenberg problem is also called the prescribed curvature problem on $\Sn$. If we denote $g=e^{2v}g_0$ in the two dimensional case
and $g=v^{\frac{4}{n-2}}g_0$ in the $n$ $(n\geq 3)$ dimensional case, This problem is equivalent to solving
$$
-\Delta_{g_{0}}v+1=Ke^{2v} \quad\mbox{ on }\, \mathbb{S}^2,
$$
and
$$
-\Delta_{g_{0}}v+c(n)R_0v=c(n)Kv^{\frac{n+2}{n-2}} \quad\mbox{ on }\, \Sn \quad  \mbox{ for } \,n\geq 3,
$$
where  $\Delta_{g_{0}}$ is the Laplace-Beltrami operator on $(\Sn, g_{0})$, $c(n)=(n-2)/(4(n-1))$,  $R_0=n(n-1)$ is the scalar curvature associated to $g_0$. The Nirenberg  problem has been studied extensively and it would be impossible to mention here all works in this area. Two significant aspects most  related to this paper are the fine analysis of blow up (approximate) solutions and the gluing methods in construction of solutions. These were studied in \cite{CL2,BC1,BC2,Li93,Li93b,Li95,Li96,Han,Yan1,DLY,LWX,WY,SZ,CY91,ES,BE,GMPY,PWW,Ahmedou,ZD,Ahmedou1} and references therein.  

In this paper, we are concerned with  the \emph{fractional Nirenberg problem} with the Nirenberg’s equation in the fractional setting which constitutes in itself a branch in geometric analysis.  This problem was naturally raised on $\si$-curvature:  finding a new metric $g$  on the standard sphere $\Sn$, $n\geq 2$, conformally equivalent to the standard one $g_{0}$, such that its $\si$-curvature is equal to a   positive function $K$ on $\Sn$.  More precisely, we investigate the existence of solutions to the following nonlinear equation:
\be\label{maineq}
 P_{\si}(v)=K v^{\frac{n+2 \si}{n-2 \si}} ,\quad v>0\quad  \text{ on }\,\Sn,
\ee
where $n\geq 2$, $\si\in (0,1)$,    $P_{\si}$ is an intertwing operator (see, e.g., Branson \cite{Branson}) of order $2\si$ given by
$$P_{\si}=\frac{\Gamma(B+\frac{1}{2}+\si)}{\Gamma(B+\frac{1}{2}-\si)}, \quad B=\sqrt{-\Delta_{g_{0}}+\Big(\frac{n-1}{2}\Big)^{2}}$$ with $\Gamma$ being  the Gamma function,   and $K$ is the  prescribed  $\si$-curvature function. The operator  $P_{\si}$ can be seen more concretely on $\Rn$ 
using stereographic projection. Indeed,  let  $\mathcal{N}$ be the north pole of $\Sn$ and let 
$$
F: \Rn \to  \Sn \setminus \{\mathcal{N}\}, \quad x \mapsto\Big(\frac{2 x}{|x|^{2}+1}, \frac{|x|^{2}-1}{|x|^{2}+1}\Big)
$$
be the inverse of stereographic projection. Then it holds that
$$
P_{\si}(\phi) \circ F=|J_{F}|^{-\frac{n+2 \si}{2 n}}(-\Delta)^{\si}(|J_{F}|^{\frac{n-2 \si}{2 n}}(\phi \circ F))\quad \text{ for }\, \phi\in C^{\infty}(\Sn),
$$
where $|J_{F}|=(\frac{2}{1+|x|^{2}})^{n}$ is the deteminant of the Jacobian of $F$ and $(-\Delta)^{\si}$ is the fractional Laplacian operator (see, e.g., \cite{Stein,Hitchhiker}).   When $\si \in(0,1)$, Pavlov and Samko \cite{Pavlov} showed that
$$
P_\si(v)(\xi)=P_\si(1) v(\xi)+c_{n,-\si} \int_{\Sn} \frac{v(\xi)-v(\zeta)}{|\xi-\zeta|^{n+2 \si}}\, \d vol_{g_{0}}(\zeta)\quad \text{ for }\,v \in C^2(\Sn),
$$
 where $c_{n,-\si}=\frac{2^{2 \si} \si \Gamma(\frac{n+2 \si}{2})}{\pi^{n / 2} \Gamma(1-\si)}$ and $\int_{\Sn}$ is understood as $\lim _{\var \to  0} \int_{|x-y|>\var}$.
Therefore, if we write  $u=|J_{F}|^{\frac{n-2 \si}{2 n}}(v \circ F)$,  one can transfer Eq. \eqref{maineq} into the following  equation 
\be\label{maineq1}
(-\Delta)^{\si} u=K(x) u^{\frac{n+2 \si}{n-2 \si}},\quad u>0 \quad  \text{ in }\,\Rn,
\ee 
which contains fractional Laplacian operator and Sobolev critical exponent. We piont out that if $K(x)$ behaves well at infinity and a solution $u$ of \eqref{maineq1} decays to zero at infinity (hence $u$ must decay to zero at the rate of $|x|^{2\si-n}$), then $u$ is actually a solution of \eqref{maineq}.

Actually, the intertwing operator  $P_{\si}$ can be well-defined for all real number $\si\in (0,\frac{n}{2})$,  see, e.g., Branson  \cite{Branson}. For $\si=1$, $P_{1}=-4 \frac{n-1}{n-2} \Delta_{g_{0}}+n(n-1)$ is the well-known conformal Laplacian associated with the classical Nirenberg problem. For $\si=2$, $P_{2}=\Delta_{g_{0}}^{2}-\frac{1}{2}(n^{2}-2n-4)\Delta_{g_{0}}+\frac{n-4}{16} n(n^{2}-4)$ is the well known \emph{Paneitz} operator. Up to positive constants $P_{1}(1)$ is the scalar curvature associated to $g_{0}$ and $P_{2}(1)$ is the so-called $Q$-curvature. In fact, $P_{1}$ and $P_{2}$ are the first two terms of a sequence of conformally covariant elliptic operators $P_{k}$ which exists for all positive integers $k$ if $n$ is odd and for $k=\{1, \ldots, n / 2\}$ if $n$ is even. These operators have been first introduced by Graham, Jenne, Mason and Sparling in \cite{GJMS}. In \cite{GZ}, Graham and Zworski showed that $P_{k}$ can be realized as the residues at $\si=k$ of a meromorphic family of scattering operators.  Unlike the Laplacian, the fractional Laplacian is a non-local operator. In a seminal paper \cite{CS2}, Caffarelli and Silvester express the non-local operator $(-\Delta)^{\si}$ when $\si \in(0,1)$ on $\Rn$ as a generalized Dirichlet-Neumann map for an elliptic boundary value problem with local differential operators. Later on, Chang and Gonz\'alez \cite{CG} showed that for any $\si\in (0,\frac{n}{2})$, the operator $P_{\si}$ can also be defined as a Dirichlet-to-Neumann operator of a conformally compact Einstein manifold.  
The fractional operators $P_{\si}$ and their associated fractional order curvatures $P_{\si}(1)$ which will be called $\si$-curvatures have been the subject of many studies. On general manifolds, the prescribing $\si$-curvature problem was considered in \cite{GZ,CG,GMS,MQ,QR} and references therein. 
  Throughout the paper, we assume $\si \in (0,1)$ and  $n\geq 2$ without otherwise stated.

Problem \eqref{maineq} (or \eqref{maineq1})  is a focus of reserach in the recent decades, and it continues to inspire new thoughts,  see for example \cite{JLX,ACH,JLX2,LR1,LR2,CZ1,AACHM,CY,CA2,LY,Niu18,Liuzy,GNT,GN,CA1}. 
 Fundamental progress  was made by Jin, Li and Xiong in \cite{JLX,JLX2}, from which they obtained compactness and existence results by applying the blow up analysis and the degree counting argument. Later on, the authors in \cite{AACHM,ACH,CY,CA1} obtained some existence criterions by establishing Euler-Hopf type index formula.   Recently, there have been some works devoted to the multiplicity results, and those  mainly use the Lyapunov-Schmidt reduction method  (see, e.g., \cite{CA2,LY,CZ1,Niu18,Liuzy,LR2,LR1,GN}).

In general, Eq. \eqref{maineq} may have no positive solution, since if $v$ is a positive solution of \eqref{maineq} with $K\in C^1(\Sn)$, then it has to satisfy the Kazdan–Warner type condition (see \cite[Proposition A.1]{JLX})
\be\label{KW}
\int_{\Sn}\langle\nabla_{g_{0}} K, \nabla_{g_{0}} \xi\rangle v^{2 n /(n-2 \si)} \,\d \xi=0.
\ee
Hence, if $K(\xi)=\xi_{n+1}+2$ with $\xi=(\xi_1,\ldots,\xi_n
,\xi_{n+1})\in \Sn$, Eq. \eqref{maineq} has no positive solutions.
 The aim of this paper is to investigate the number of positive solutions to Eq. \eqref{maineq}  (or \eqref{maineq1}) under various local assumptions on the prescribed function  $K$. Basically speaking, we obtain a $C^0$ density result for the fractional Nirenberg problem \eqref{maineq} by constructing infinitely many multi-bump solutions to some  perturbed equations.  As a variation of this idea, the related problem \eqref{maineq1}  with $K(x)$ being periodic in one of the variables are also studied and infinitely many multi-bump solutions (modulo translations by its periods) are obtained under some  flatness conditions. 
\subsection{Main theorems}
We now list the main theorems of this paper and some remarks on them.
The first one deals with the existence of multi-bump solutions to the perturbed \emph{fractional Nirenberg problem}.
\begin{thm}\label{thm:0.1}
	Let $K(x)$ be a positive function on $\Sn$. Assume that $\widetilde{x}\in \Sn$ and  $K \in C^{0}(B_{\widetilde{\var}}(\widetilde{x}))$ ($B_{\widetilde{\var}}(\widetilde{x})$ denotes the geodestic ball in $\Sn$ of radius $\widetilde{\var}$ and centered at $\widetilde{x}$). Then for any $\var\in(0,\widetilde{\var})$, any integers $k\geq 1$ and $m\geq 2$,  there exists
	$K_{\var, k, m} \in L^{\infty}(\Sn)$ with $K_{\var, k, m}-K \in C^{0}(\Sn)$, $\|K_{\var, k, m}-K\|_{C^{0}(\Sn)}<\var$ and $K_{\var, k, m}\equiv K$ in $\Sn \setminus  B_{\var}(\widetilde{x})$, such that, for each $2 \leq s \leq m$, the equation
\be\label{perturbed}
		P_\si(v)=K_{\var, k, m}  v^{\frac{n+2\si}{n-2\si}} \quad \mbox{ on } \,\Sn
\ee
	has at least $k$ positive solutions with $s$ bumps.
\end{thm}
	For the precise  meaning of $s$ bumps, refer to the  proof of Theorem \ref{thm:0.1} in Section \ref{sec:6}.	Roughly speaking, we say a solution has $s$ bumps if most of its mass is concentrated in $s$ disjoint regions. Since the number of bumps and the number of solutions can be chosen arbitrarily, we obtain the existence of infinitely many multi-bump solutions to Eq. \eqref{perturbed}.

If $K_{\var}=1+\var K(x)$ and $K(x)$ has at least two critical points satisfying some local conditions, Chen-Zheng \cite{CZ1} showed that Eq. \eqref{maineq1} with $K=K_{\var}$ has two multi-bump solutions when $\var$ is small. Here we give a more general existence result  since we can perturb  any given positive continous function in any  neighborhood of any given point on $\Sn$ such that for the perturbed equations there exist many solutions. The perturbation $K_{\var, k, m}$ can be constructed explicitly
in Section \ref{sec:5} and it is a way of gluing approximate solutions into genuine solutions. The method is variational, rather than through Lyapunov-Schmidt reduction method as  in \cite{CA2,LY,CZ1,Niu18,Liuzy,LR2,LR1,GN}, etc. The solutions we obtained  have most of their mass in disjoint small balls centered at the maximum points of $K_{\var, k, m}$, which are far away from each other. This fact accounts for why there are infinitely many solutions solve Eq. \eqref{perturbed}. Moreover, the solutions we constructed are nearly \emph{bubble} functions, we refer to Section \ref{sec:5} for more details.  However, one cannot expect to perturb any $K(x)$ near any point $\widetilde{x} \in \Sn$ in the $C^{1}$ sense to obtain the existence of solutions, which is evident from Kazdan-Warner type condition \eqref{KW}.

\begin{rem}
	The main feature of Theorem \ref{thm:0.1}  is that, even if a given positive function $K$ in Theorem \ref{thm:0.1}  cannot be realized as the $\si$-curvature of a metric $g$ conformal to $g_{0}$, nevertheless we can find a function $K'$ arbitraly close to $K$ in $C^{0}(\Sn)$ which is the $\si$-curvature  as many conformal metrics to $g_{0}$ as we want. 	
\end{rem}

Regarding the statement in  Theorem \ref{thm:0.1}, we have the following result.
\begin{cor}\label{cor:1}
	The 	smooth $\si$-curvature  functions of metrics conformal to $g_0$ are dense in $C^{0}(\Sn)$ among positive functions.
\end{cor}

  Next we consider the  problem \eqref{maineq1}.   Before  stating the  results, we introduce some notations. 
  
   Let $E$ be the completion of the space $C_{c}^{\infty}(\Rn)$  with respect to  the norm (see \cite{Hitchhiker})
$$
\| u\|_E:=\Big(\int_{\Rn}|(-\Delta)^{\frac{\si}{2}} u|^2\,\d  x\Big)^{1/2} =\Big(\int_{\Rn\times\Rn}\frac{|u(x)-u(y)|^2}{|x-y|^{n+2\si}}\,\d x\, \d y\Big)^{1/2}. 
$$
  Denote the critical Sobolev exponent $2_{\si}^*:=\frac{2n}{n-2\si}$,  it is well-known that  $E$  can be embedded into  $L^{2_{\si}^*}(\Rn)$ and the sharp  Sobolev  inequality  is
\be\label{sharp1}
S_{n,\si}\Big(\int_{\Rn}|u|^{2_{\si}^{*}} \Big)^{1/2_{\si}^{*}} \leq \Big(\int_{\Rn}|(-\Delta)^{\frac{\si}{2}} u|^{2} \Big)^{1/2}.
\ee
For any $z\in\Rn$ and $\lam >0$,  set
\be\label{bubble}
\delta(z,\lam )(x):=2^{\frac{n-2\si}{2}}\Big(\frac{\Gamma(\frac{n+2\si}{2})}{\Gamma(\frac{n-2\si}{2})}\Big)^{\frac{n-2\si}{4\si}}\Big(\frac{\lam }{\lam ^{2}+|x-z|^{2}}\Big)^{\frac{n-2\si}{2}}.
\ee Then $\delta( z,\lam  )$ is the  solution to the problem 
$$
(-\Delta)^{\si}u=u^{\frac{n+2\si}{n-2\si}},\quad u>0 \quad \text{ in }\, \Rn
$$ for every $(z,\lam )\in \Rn\times (0,+\infty)$ (see, e.g., \cite{MQ,JLX}). Moreover, \eqref{bubble} and its non-zero constant  multiples attain the sharp   Sobolev  inequality \eqref{sharp1},  see Lieb  \cite{Lie83}.

Let $D$ be the completion of $C_{c}^{\infty}(\over{\R}^{n+1}_+)$ with respect to the weighted Sobolev norm
$$
\| U\|_{D}:=\Big(\int_{\Rp}t^{1-2\si}|\nabla U (x,t)|^2 \,\d x\,\d t\Big)^{1/2},\quad (x,t)\in  \Rp:=\Rn \times(0, \infty).
$$
It is clear that $D$ is a Hilbert space with  the inner product
\be\label{inner}
\langle U,V \rangle:=\int_{\Rp}t^{1-2\si}\nabla U \nabla V \,\d x\,\d t.
\ee
We recall that there exists a well-defined continuous trace map $\operatorname{Tr}:D\rightarrow E$. 
Throughout the paper, we  write $\|\cdot\|$ to denote the norm of $D$ and $D^+$ to denote the set consisiting of all positive functions of $D$.

We analyze Eq. \eqref{maineq1} via the extension formulation for fractional Laplacians established by Caffarelli and Silvestre \cite{CS2}. This is a commonly used tool nowadays, through which instead of Eq. \eqref{maineq1} we can study a degenerate elliptic equation with a Neumann boundary condition in one dimension higher:
\be\label{maineq2} 
\left\{\begin{aligned}
&\operatorname{div}(t^{1-2 \si} \nabla U)=0 && \text { in }\, \Rp,\\ &\pa _{\nu}^{\si} U= K(x)U(x,0)^{\frac{n+2 \si}{n-2 \si}}&& \text { on }\, \Rn,
\end{aligned}\right.
\ee
	where  $u(x)=U(x,0)$ and $\pa _{\nu}^{\si}U:=-N_{\si}^{-1}\lim_{t \to  0^{+}}t^{1-2\si}\pa_t U(x,t)$ with $N_{\si}=2^{1-2 \si}\Gamma(1-\si) / \Gamma(\si) $.
	Associated with \eqref{maineq2} is the following  energy functional $I_K: D\to  \R$ by
\be\label{functional}
I_{K}(U):=\frac{1}{2} \int_{\Rp}t^{1-2\si}|\nabla U|^2\,\d x\,\d t-\frac{N_{\si}}{2^{*}_{\si}} \int_{\Rn}K(x)|U(x,0)|^{2^{*}_{\si}}\, \d x.
\ee Obviously a positive critical point give rise to a positive solution to Eq. \eqref{maineq2} and thus a positive solution to  Eq. \eqref{maineq1}.	
	 The extension in \eqref{maineq2} will always refer to the cononical one:
\be\label{extension}
	U(x,t)=\mathcal{P}_{\si}[u]:=\beta(n,\si)\int_{\Rn}\frac{t^{2\si}}{(|x-\xi|^2+t^2)^{\frac{n+2\si}{2}}} u(\xi)\, \d \xi,
\ee
	where $\beta(n,\si)$ is a normalization constant. We   refer to $U=\mathcal{P}_{\si}[u]$ in  \eqref{extension} to be the \emph{extension}  of $u$.

Let $H^{1}(t^{1-2\si},\Rp)$ be the closure of $C_{c}^{\infty}(\over{\R}^{n+1}_{+})$ under the norm
$$
\|U\|_{H^{1}(t^{1-2\si},\Rp)}:=\Big(\int_{\Rp}t^{1-2\si}(|U|^2+|\nabla U|^2)\,\d x\,\d t\Big)^{1/2}. 
$$
Then the extremal functions of Sobolev trace inequality on  $H^{1}(t^{1-2\si},\Rp)$ 
\be\label{trace}
\mathcal{S}_{n,\si} \Big(\int_{\Rn}|U(x,0)|^{2^{*}_{\si}}\,\d x\Big)^{1/2^{*}_{\si}}\leq \Big(\int_{\Rp}t^{1-2\si}|\nabla U|^2\,\d x\,\d t\Big)^{1/2}
\ee
have the  form $U(x,t)=\al \widetilde{\delta}(z,\lam )$ for any $\al\in 
\R\setminus \{0\}$, $z\in \Rn$ and $\lam>0$,  where $\mathcal{S}_{n,\si}=N_{\si}^{-1/2}S_{n,\si}$ is the optimal constant and $\widetilde{\delta}(z,\lam):=\mathcal{P}_{\si}[\delta(z,\lam )]$,  
see, e.g., \cite{JX2,GMS}. We call $\delta(z,\lam )$ and $\widetilde{\delta}(z,\lam )$ \emph{bubbles}.

Let $K(x)\in L^{\infty}(\Rn)$,  $O^{(1)}, \ldots, O^{(m)} \subset \Rn$ are some open sets with $\operatorname{dist}(O^{(i)}, O^{(j)}) \geq 1$  for any $ i \neq j$. If $K \in C^{0}(\bigcup_{i=1}^{m} O^{(i)})$, 
we define $V(m, \var):=V(m, \var, O^{(1)}, \ldots, O^{(m)}, K)$  as the following open set in  $D$ for $\var>0$:
\begin{align}\label{bumps}\begin{aligned}
V(m, \var):=\Big\{
U\in D:&\,\exists \,\al=(\al_{1}, \ldots, \al_{m}) \in \R^m,\, \exists\, z=(z_{1}, \ldots, z_{m}) \in O^{(1)}\times \ldots \times O^{(m)}, \\&\,\exists\,\lam =(\lam _{1}, \ldots, \lam _{m}) ,\,\lam _{i}>\var^{-1},\,\forall \,i\leq m, \text{ such that}\\ &\, |\al_{i}-K(z_{i})^{(2\si-n)/4\si}|<\var,\, \forall \,i\leq m,\text{ and}
\\&\, \|U-\varphi(\al,z , \lam )\|<\var
\Big\},\end{aligned}
\end{align}
where $\varphi(\al,z, \lam ):=\sum_{i=1}^{m} \al_{i}  \widetilde{\delta}( z_{i},\lam _{i} )$.The open set  $V(m, \var)$ recodes the information of the concentration rate and the locations of points of concentration.  

In this paper, following the ideas in \cite{Li93,Li93b,Li95}, we will construct multi-bump solutions  near critical points of $K(x)$ and the bumps can be chosen  arbitrarily many.  For this purpose, we assume that $K(x)$ satisfies the following conditions:
 \begin{itemize}
 	\item[$(H_1)$] $0< \inf_{\Rn}K(x)<\sup_{\Rn}K(x)<\infty$.
 	\item[$(H_2)$] $K(x)$ is periodic in at least one variable,  that is, there is a positive constant $T$, such that $K(x_1+lT, x_2 , \ldots, x_n)=K(x_1 , x_2 ,\ldots, x_n)$	for any integer $l$ and $x\in\Rn$.
 	\item[$(H_3)$] Let $\Sigma$ denote the set containing  all  critical points $z$ of $K(x)$, satisfying (after a suitable rotation of the coordinate system depending on $z$),
$$
 	K(x)=K(z)+\sum_{i=1}^{n} a_{i}|x_{i}-z_{i}|^{\beta}+R(|x-z|) \quad \text{ for $x$ close to $z$,}
$$ where $a_{i}$ and $\beta$  are some constants depending on $z$, $a_{i} \neq 0$ for $i=1, \ldots, n$, $\sum_{i=1}^{n} a_{i}<0$, $\beta \in(n-2\si, n)$, and $R(y)$ is  $C^{[\beta]-1,1}$   and  satisfies $\sum_{s=0}^{[\beta]}|\nabla^{s} R(y)||y|^{-\beta+s}=o(1)$ as $y\to  0$,  where $C^{[\beta]-1,1}$ means that up to $[\beta]-1$ derivatives are Lipschitz functions, $[\beta]$ denotes the integer part $\beta$ and $\nabla^{s}$ denotes all possible partial derivatives of order $s$.
 \end{itemize}

We remark that	the  $(H_3)$ type condition was originally introduced by Li in \cite{Li95} and also has been applied to  the fractional setting,  see, e.g., \cite{Niu18,JLX}.

\begin{thm}	\label{thm:0.2}Assume that  $K(x) \in C^{1}(\Rn) \cap L^{\infty}(\Rn)$ satisfies  $(H_1)$-$(H_3)$ and 	
	\begin{itemize}
		\item[$(H_4)$] $K_{\max } := \max _{x \in \Rn} K(x)>0$ is achieved and the set $K^{-1}(K_{\max }) := \{x \in \Rn: K(x)=K_{\max }\}$ has at least one bounded connected component,
		denoted as $\mathscr{C}$. 
	\end{itemize}
Then for any integers $m\geq 2$, Eq. \eqref{maineq2} has infinitely $m$-bump solutions  modulo translations by $T$ in the $x_{1}$ variable.	More precisely, for any $\var>0$, $x^*\in\mathscr{C}$, there exists a  constant $l^{*}>0$ such that for any	integers  $l^{(1)},\ldots, l^{(k)}$ satisfying $2\leq k \leq m$, $\min _{1 \leq i \leq k}|l^{(i)}|\geq l^{*}$, $\min _{i \neq j}|l^{(i)}-l^{(j)}| \geq l^{*}$,   there is at least one solution $U$ of Eq.
	\eqref{maineq2} in $V(k, \var, B_{\var}(x^{(1)}), \ldots, B_{\var}(x^{(k)}))$ with $$
	k c-\var \leq I_{K}(U) \leq k c+\var,
	$$ where $$x^{(i)}=x^{*}+(l^{(i)} T, 0,\ldots,0),\quad c=\frac{\si N_{\si}}{n}(K(x^{*}))^{(2 \si-n) / 2\si}(S_{n,\si})^{n/\si},$$
	and $V(k, \var, B_{\var}(x^{(1)}), \ldots, B_{\var}(x^{(k)}))$ are some subsets of $D$ defined in \eqref{bumps}.	
\end{thm}

One can see from the description of  $(H_4)$ that there exists some bounded open neighborhood $O$ of $\mathscr{C}$ such that $K_{max}\geq \max_{x\in\pa O}K+\delta$,
		where $\delta$ is some small positive number. This fact together with  assumption $(H_2)$ implies that $K(x)$ has a sequence of local maximum points $z_j$	with $|z_j|\to  \infty$ as $j\to  \infty$.  Furthermore, 	 $(H_4)$ is sharp in the sense that one can construct examples easily  to show that if $(H_4)$ is not satiesfied, Eq. \eqref{maineq1}  may have no nontrivial solutions, which shows that  $(H_4)$ is not merely a technical hypothesis, see Example \ref{exa} below.

From the definition in  \eqref{bumps} we know that if	$U\in V(k, \var, B_{\var}(x^{(1)}), \ldots, B_{\var}(x^{(k)}))$, then $U$ has most of
		its mass concentrated in $B_{\var}(x^{(1)}), \ldots, B_{\var}(x^{(k)})$. In particular, if $(l^{(1)}, \ldots, l^{(k)}) \neq (\widetilde{l}^{(1)}, \ldots, \widetilde{l}^{(k)})$, the  solutions $U$ and $\widetilde{U}$ are different.  We also remark that 	$c$	is the mountain pass value  to Eq. \eqref{maineq2} and $S_{n,\si}$ is the sharp constant in    \eqref{sharp1}.

	Note that 	the authors  in \cite{Niu18,LR1} apply the Lyapunov-Schmidt reduction method to  obtain infinitely	many multi-bump solutions clustered on some lattice points in $\Rn$ under   similar assumptions.  Liu \cite{Liuzy} has used the same method to construct infinitely many concentration solutions to Eq. \eqref{maineq1} under the assumption that $K(x)$ has a sequence of strictly local maximum points moving to infinity. 	The solutions we constructed in Theorem \ref{thm:0.2}, roughly speaking, concentrate at $k$ different points and the distance between different concentrate points is very large. 
\begin{exa}[Nonexistence]\label{exa} Suppose that $K(x) \in C^{1}(\Rn) \cap L^{\infty}(\Rn)$,  $K(x)$ and $\nabla K(x)$ are bounded in $\Rn$, $\frac{\pa K}{\pa x_2}$ is nonnegative but not identically zero. Then the only nonnegative solution of Eq. \eqref{maineq1} in $E$ is the trivial solution $u\equiv 0$.
\end{exa}
\begin{proof}
	Let $u\geq 0$ be any  solution in  $E$. Multiplying  \eqref{maineq1} by $\frac{\pa u}{\pa x_2}$ and using \eqref{KW}, we obtain $\int_{\Rn}\frac{\pa K}{\pa x_2}u^{2^{*}_{\si}}\,\d x=0$.
	The hypotheses on $K(x)$ imply that $u$ is identically zero in an open set, hence $u\equiv 0$ by the  unique continuation results (see, e.g., \cite{L.H,FV}).
\end{proof}
In fact we can obtain more information on the solutions obtained in Theorem \ref{thm:0.2}.
\begin{thm}\label{thm:0.3}Assume that  $K(x)$ satisfies  $(H_1)$-$(H_3)$  and 
\begin{itemize}
			\item[$(H_4)'$] There exist some positive constant $A_{1}>1$ and  a bounded open set $O \subset \Rn$ such that
			\begin{gather*}
			K \in C^{1}(\over{O}),\\	1 / A_{1} \leq K(x) \leq A_{1}, \quad \forall \, x \in \over{O},\\\max_{x\in\over{O}}K(x)=\sup _{x \in \Rn} K(x)>\max _{x \in \pa  O} K(x).
			\end{gather*}
	\end{itemize}	Then for any $\var>0$, Eq. \eqref{maineq2}  has infinitely many solutions satisfying
\be\label{0.13}
	c \leq I_{K}(U) \leq c+\var  \text {~ or ~} 2 c-\var \leq I_{K}(U) \leq 2 c+\var
\ee and
	$$
	\sup \big\{\|U\|_{L^{\infty}(\Rp)} : I_{K}'(U)=0,\,\text {$U$\,satisfies \eqref{0.13}} \big\}=\infty,
$$ where $$
	c=\frac{\si N_{\si}}{n}(\max _{\over{O}} K)^{(2\si-n) / 2\si}(S_{n,\si})^{n/\si}.$$
		More precisely, for any $\var>0$, there exists $l^{*}>0$ such that for any
	integers  $l^{(1)}$, $l^{(2)}$ satisfying $|l^{(1)}-l^{(2)}| \geq l^{*}$, there is at least one  solution $U$ of 
Eq.	\eqref{maineq2} in $V(1, \var, O,K) \cup V(2, \var, O_{l}^{(1)}, O_{l}^{(2)}, K)$ satisfying  \eqref{0.13}, where $$
	O_{l}^{(1)}=O+(l^{(1)} T, 0, \ldots, 0), \quad O_{l}^{(2)}=O+(l^{(2)} T, 0, \ldots, 0),
	$$
	and $V(1, \var, O,K)$, $V(2, \var, O_{l}^{(1)}, O_{l}^{(2)}, K)$ are some subsets of $D$ defined in \eqref{bumps}.
\end{thm}
\begin{rem}
By the stereographic projection from $\Sn \setminus \{\mathcal{N}\}$ to $\Rn$, the solutions obtained in Theorem \ref{thm:0.2} and  \ref{thm:0.3} can be lifted to a solution of Eq. \eqref{maineq} on $\Sn$ which is positive except at the north pole $\mathcal{N}$. In this sense, Eq. \eqref{maineq} is solvable under the assumptions of Theorem \ref{thm:0.2} and  \ref{thm:0.3}.
\end{rem}

Analogous conclusions in the setting  $\si=1$ are deduced   in a series of papers \cite{Li95,Li93,Li93b}. Using the Dirichlet-Neumann map, we consider the prescribing mean curvature problem where the equations are without weights and thus elliptic, and we obtain the density and multiplicity results  in  \cite{TWZ}. The main objective of this paper is to extend the above results  to the nonlocal setting $\si \in (0,1)$. Although certain parts of the proof can be obtained by minor modifications of the classical arguments in \cite{Li93,Li93b,Li95}, there are plenty of technical difficulties which demand new ideas to handle the non-local terms. 

We study Eq. \eqref{maineq} (or  \eqref{maineq1}) by subcritical approach, for example, we refer to the reader \cite{JLX,JLX2}.   It is worth noting that our methods  continue in the direction pioneered in the earlier work \cite{Se,CR1,CR2}. These techniques provide, roughly speaking, ways of gluing approximate solutions together to obtain a genuine solution. There have been some works on gluing approximate solutions by using the Lyapunov-Schmidt reduction method (see, e.g., \cite{CA2,LY,CZ1,Niu18,Liuzy,LR2,LR1,GN}) where more precise information on the linearized problem is needed. However, it seems that the methods in S\'ere \cite{Se}, Coti Zelati-Rabinowitz \cite{CR1,CR2}   have provided an elegant way to glue approximate solutions for certain periodic differential equations where it is difficult to obtain as precise information as needed for applying the  Lyapunov-Schmidt reduction method. The basic idea is as follows: Given finitely many solutions (at low energy), to translate their supports far apart and patch the pieces together create many multi-bump solutions.  The  original and powerful ideas in  \cite{CR1,CR2,CES,Se} permit the construction of such solutions via variational methods. In particular, they are able to find many homoclinic-type solutions to periodic Hamiltonian systems (see \cite{Se,CR1}) and to certain elliptic equations of nonlinear Schr\"odinger type on $\Rn$ with periodic coefficients (see \cite{CR2}). Li  has given a slight modification to the minimax procedure in \cite{CR1,CR2} and has applied it to certain problems where periodicity is not present, for example, the problem of prescribing scalar curvature on $\Sn$ (see \cite{Li93b,Li93,Li95}).  Inspired by the above  works, we   attempt to modify the above mentioned gluing method towards equations with the fractional Laplacians in the Euclidean setting or the conformal Laplacian operators under a particular choice of the metric
in constructing multi-bump solutions. This paper also overcomes the difficulty appearing in using  Lyapunov-Schmidt reduction method to locate the concentrating points of the solutions. We also believe that this gluing method can be applied to the construction of \emph{bubbling} solutions in various problems in conformal geometry, for instance, prescribing (fractional) $Q$-curvature problems.

\subsection{Structure of the paper and comment on the proof}

 Theorems \ref{thm:0.1}-\ref{thm:0.3} and  Corollary \ref{cor:1} are  derived in Section \ref{sec:6} from Proposition \ref{prop:5.1}, a more general result on Eq. \eqref{maineq1}. To  derive Proposition \ref{prop:5.1}, we first study a compactified problem (Theorem \ref{thm:3.1}) in Section \ref{sec:3}. Then we derive Proposition \ref{prop:5.1} by using Theorem \ref{thm:3.1} and some blow up analysis  in \cite{JLX}. Theorem \ref{thm:3.1} is a  technical result in our paper, which is essential to make the variational gluing methods applicable. Our presentation is largely influenced by the papers \cite{Li93,Li95,Li93b} which studied existence and compactness of solutions to the classical Nirenberg problem.

The present paper is organized as the following.  In Section \ref{sec:3}, existence and  multiplicity results for the subcritical equations will be stated, and its proof will be  sketched. 
In Section \ref{sec:4}, we follow and refine the analysis of Bahri and Coron \cite{BC1,BC2} to study the subcritical interaction of two well-spaced \emph{bubbles}.  The technical result Theorem \ref{thm:3.1} will be completed  in Section \ref{sec:5} by applying minimax procedure as in Coti Zelati and Rabinowitz \cite{CR1,CR2}.  Finally, the main theorems are proved in Section \ref{sec:6} with the aid of  blow up analysis  established by Jin, Li and Xiong \cite{JLX}. In Appendix \ref{sec:1}, we establish some a priori estimates for solutions to degenerate elliptic  equations.  In Appendix \ref{sec:2}, we consider a minimization problem on exterior domain.  The results in appendix  were used in proving  Theorem \ref{thm:3.1}.

\subsection{Notation}
  We collect below a list of the main notation used throughout the paper.
\begin{itemize}\raggedright
	\item We use capital letters, such as $X=(x,t)$ to denote an element of the upper half space $\Rp$,  where $x \in \Rn$  and $t>0$.
	\item For a domain $D \subset \R^{n+1}$ with boundary $\pa  D$, we denote $\pa ' D$ as the interior of $\over{D} \cap \pa  \Rp$ in $\Rn=\pa  \Rp$ and $\pa ''D=\pa  D \setminus  \pa ' D$.	
	\item  For $\over{X} \in \R^{n+1}$, denote $\B_{r}(\over{X}):=\{X \in \R^{n+1}:|X-\over{X}|<r\}$ and $\B_{r}^{+}(\over{X}):=\B_{r}(\over{X}) \cap \Rp$. If $\over{X}=(\over{x},0) \in \pa  \Rp$, denote $B_{r}(\over{X}):=\{x\in\Rn:|x-\over{x}|<r\}$.
	Hence $\pa ' \B_{r}^{+}(\over{X})=B_{r}(\over{X})$ if $\over{X} \in \pa  \Rp$. Moreover, when $\over{X}=(\over{x},0)$, we simply use $\B_r(\over{x})$ (resp. $\B^{+}_r(x)$ and $B_r(\over{x})$)  for $\B_r(\over{X})$ (resp. $\B^{+}_r(\over{X})$ and $B_r(\over{X})$) and  will not keep writing the center $\over{X}$ if $\over{X}=0$.
	\item	For any weakly differentiable function $U(x,t)$ on $\Rp$, we denote $\nabla_{x} U=(\pa _{x_1} U, \ldots, \pa _{x_n}U)$  and $\nabla U=(\nabla_{x} U, \pa _{t} U)$.
	\item  For any $\si\in(0,1)$ and $n\geq 2$, we denote  $2^{*}_{\si}=\frac{2n}{n-2\si}$  and $H(x)=(\frac{2}{1+|x|^2})^{(n-2\si)/2}$.
	\item 	$C>0$ is a generic constant which can vary from line to line.
\end{itemize}

\section{\sc Construction of the approximate solutions}\label{sec:3}
We point out that due to the presence of the  Sobolev critical exponent, the  Euler-Lagrange functional corresponding to Eq. \eqref{maineq1} does not satisfy the Palais-Smale condition. One way to overcome such a difficulty is to consider the following subcritical approximation  problem
$$(-\Delta)^{\si} u=K(x)H^{\tau}(x) u^{\frac{n+2\si}{n-2\si}-\tau} ,\quad u>0\quad  \text { in } \,  \Rn,$$
where $\tau> 0$ is small and $H(x)=(\frac{2}{1+|x|^2})^{(n-2\si)/2}$ is defined as before. The aim in this section is to  establish the  existence  and multiplicity results for the above subcritical  equations.

We first introduce some notations which are used throughout the paper.

Let $\{K_{l}(x)\}$ be a sequence of functions satisfying the following conditions.
\begin{itemize}
	\item[(i)]There exists some positive constant $A_{1}>1$ such that for any
	$l= 1, 2, 3,\ldots$,
\be\label{18}
1/A_1\leq  K_{l}(x) \leq A_{1},  \quad \forall\, x \in \Rn.
\ee 
	\item[(ii)]	 For some integer $m \geq 2$, there exist $z_{l}^{(i)} \in \Rn$,  $1 \leq i \leq m$, $R_{l} \leq \frac{1}{2} \min_{i \neq j}|z_{l}^{(i)}-z_{l}^{(j)}|$, such that $K_{l}$ is continuous near $z_{l}^{(i)}$ and 	
	\begin{align}
	\label{19}&\lim _{l \to  \infty} R_{l}=\infty,&\\\label{20}&K_{l}(z_{l}^{(i)})=\max _{x \in B_{R_{l}}(z_{l}^{(i)})} K_{l}(x), & 1 \leq i \leq m,\\&\lim _{l \to  \infty} K_{l}(z_{l}^{(i)})=a^{(i)}>0,& 1 \leq i \leq m,\label{21}\\\label{22}&K_{\infty}^{(i)}(x):=(\text {weak $*$}) \lim _{l \to  \infty} K_{l}(x+z_{l}^{(i)}), & 1 \leq i \leq m.
	\end{align}
	\item[(iii)]There exist some positive constants $A_{2}, A_{3}>1$, $\delta_{0}, \delta_{1}>0$, and
	some bounded open sets $O_{l}^{(1)}, \ldots, O_{l}^{(m)} \subset \Rn$, such that, if we define for  $1 \leq i \leq m$,
	\begin{gather*}
	\widetilde{O}_{l}^{(i)}=\big\{x \in \Rn : \operatorname{dist}(x, O_{l}^{(i)})<\delta_{0}\big\},\\	O_{l}=\bigcup_{i=1}^{m} O_{l}^{(i)},  \quad \widetilde{O}_{l}=\bigcup_{i=1}^{m} \widetilde{O}_{l}^{(i)},
	\end{gather*}
	we have
	\begin{gather}
		z_{l}^{(i)} \in O_{l}^{(i)},\quad \operatorname{diam}(O_{l}^{(i)})< R_{l}/10,\\ \label{23}	K_{l} \in C^{1}(\widetilde{O}_{l},[1/A_{2}, A_{2}]),\\\label{24}	K_{l}(z_{l}^{(i)}) \geq \max _{x \in \pa  O_{l}^{(i)}} K_{l}(x)+\delta_{1},\\\label{25} 	\max _{x \in\widetilde{O}_{l}}|\nabla K_{l}(x)| \leq A_{3}.
	\end{gather}
	\end{itemize}

For $\var>0$ small, we define  $V_{l}(m, \var):=V(m, \var, O_{l}^{(1)}, \ldots, O_{l}^{(m)}, K_{l})$ as in \eqref{bumps}. 
Here and in the following, we are concerned with the case $m=2$, since the more general result   is similar in nature.

If $U$ is a function in $V_l(2,\var)$, one can find an optimal representation, following the ideas introduced in \cite{BC1,BC2}. Namely, we have
\begin{prop}\label{prop:3.1}
	There exists $\var_{0} \in(0,1)$ depending only on $A_{1}, A_{2}, A_{3}, n,\si, \delta_{0}$ and $m$, but independent of $l$, such that, for any $\var\in (0,\var_{0}]$, $U \in V_{l}(2, \var)$, the following 	minimization problem
\be\label{27}
	\min _{(\al, z, \lam ) \in B_{4\var}}\Big\|U-\sum_{i=1}^{2} \al_{i}  \widetilde{\delta}( z_{i}, \lam _{i} )\Big\|
\ee 
has a unique solution $(\al,z,\lam )$ up to a permutation. Moreover, the minimizer  is achieved in $B_{2 \var}$  for large $l$, where
	\begin{align*}
	B_{\var}=\Big\{(\al, z,\lam ):  &\,\al=(\al_{1}, \al_{2}),\,1/(2 A_{2}^{(n-2\si)/ 4\si}) \leq \al_{1}, \al_{2} \leq 2 A_{2}^{(n-2\si)/4\si},\\&\,z=(z_{1}, z_{2})\in {O}_{l}^{(1)}\times {O}_{l}^{(2)},\, \lam =(\lam _{1}, \lam _{2}),\,   \lam _{1}, \lam _{2} \geq \var^{-1}\Big\}.	
	\end{align*}	In particular, we can write  
	$U=\sum_{i=1}^{2} \al_{i} \widetilde{\delta}( z_{i}, \lam _{i} )+v$, where $v\in D$ and satisfis
	$$
	\big\langle\widetilde{\delta}( z_i, \lam _i ), v\big\rangle=\big\langle\frac{\partial \widetilde{\delta}( z_i, \lam _i )}{\partial \lam _i}, v\big\rangle=\big\langle\frac{\partial \widetilde{\delta}( z_i, \lam _i )}{\partial z_i}, v\big\rangle=0\quad \text{ for each }\, i=1,2.
	$$
	Here  $\langle \cdot,\cdot\rangle$ denotes the inner product  \eqref{inner} as before. In addition, the variables $\{\al_{i}\}$ satisfy
\be\label{28''}
|\al_{i}-K_{l}(z_{i})^{(2\si-n)/4\si}|=o_{\var}(1)\quad  \text{ for }\,i=1,2,
\ee where $o_{\var}(1)\to  0$  as $\var\to  0$.  
\end{prop}
\begin{proof}
	The proof  is similar  to  the corresponding statements in \cite{BC1,BC2},  
we omit it here.	
\end{proof}

\begin{rem}
Bahri and Coron	\cite{BC1,BC2} introduced  the  ``critical points at infinity” method and the algebraic-topological tools to study 
	 scalar curvature problems. These  methods provide  so-called Bahri–Coron-type existence criterium, which  also has been applied to deal with the fractional Nirenberg prolem, see, e.g., \cite{AACHM,ACH,CY,CA1,CA2}. 
\end{rem}
In the sequel, we will often spilt $U$, a function in  $ V_{l}(2, \var)$, $\var \in (0,\var_{0}]$, under the form 
$$
U=\al_{1}^{l} \widetilde{\delta}( z_{1}^{l},\lam _{1}^{l} )+\al_{2}^{l}  \widetilde{\delta}( z_{2}^{l},\lam _{2}^{l} )+v^{l}
$$ after making the minimization \eqref{27}. Proposition \ref{prop:3.1}  guarantees the existence and uniqueness of $\al_{i}=\al_{i}(u)=\al_{i}^{l}$, $z_{i}=z_{i}(u)=z_{i}^{l}$ and $\lam _{i}=\lam _{i}(u)=\lam _{i}^{l}$ for $i=1,2$ (we omit the index $l$ for simplicity).

For any $K\in L^{\infty}(\Rn)$ and $U\in D$,  we define
\be I_{K, \tau}(U):= \frac{1}{2} \int_{\Rp}t^{1-2\si}|\nabla U|^2\,\d X-\frac{N_{\si}}{2_{\si}^{*}-\tau}  \int_{\Rn}  K(x)H^{\tau}(x)|U(x,0)|^{2_{\si}^{*}-\tau}\, \d x \label{subcriticalfunctional}
\ee
with $\tau\geq 0$ small. Clearly, $I_{K}= I_{K, 0}$.

To continue our proof, let $\{\over{\tau}_l\}$ be a sequence satisfying
\be\label{26}
\lim _{l \to  \infty} \over{\tau}_{l}=0, \quad \lim _{l \to  \infty}(|z_{l}^{(1)}|+|z_{l}^{(2)}|)^{\over{\tau}_l}=1.
\ee

Now  we give a lower bound energy estimate for some well-spaced \emph{bubbles}.
\begin{lem}\label{lem:3.1}
Let $\var_0$ be the constant in Proposition \ref{prop:3.1}.	Suppose that $\var_1 \in(0,\var_{0})$ small enough and  $l$ large enough, $ \tau \in [0,\over{\tau}_{l}]$. Then there exists some constant $A_{4}=A_{4}(n,\si,\delta_{1}, A_{2})>1$ such that for any
	$U \in V_{l}(2, \var_1 )$ with $z_{1}(U) \in \widetilde{O}_{l}^{(1)}$, $z_{2}(U) \in \widetilde{O}_{l}^{(2)}$, and $\operatorname{dist}(z_{1}(U), \pa  O_{l}^{(1)})<\delta_{1} /(2 A_{3}) $
	or $\operatorname{dist}(z_{2}(U), \pa  O_{l}^{(2)})<\delta_{1} / (2 A_{3})$, we have $$
	I_{K_{l},\tau}(U) \geq c^{(1)}+c^{(2)}+1/A_{4},
$$where
\be\label{28}
	c^{(i)}=\frac{\si N_{\si}}{n}(a^{(i)})^{(2\si-n)/2\si}(S_{n,\si})^{n/\si} .
\ee
\end{lem}
\begin{proof}
We assume  that  	$\operatorname{dist}(z_{1}(U), \pa  O_{l}^{(1)})<\delta_{1} /(2 A_{3})$. It follows from  \eqref{28''} and some direct   computations that, for $\var_1 >0$ small and $l$ large,
	\begin{align*}
	I_{K_{l}, \tau}(U)=&\sum_{i=1}^{2} I_{K_{l}, \tau}(\al_{i}  \widetilde{\delta}( z_{i},\lam _{i} ))+o_{\var_1 }(1)
\\=&\sum_{i=1}^{2}\Big\{\frac{1}{2} K_{l}(z_{i})^{(2\si-n)/2\si} \| \widetilde{\delta}( z_{i},\lam _{i} )\|^{2}\\&-\frac{N_{\si}}{2_{\si}^{*}} K_{l}(z_{i})^{-n/2\si} \int_{\Rn} K_{l}  \delta( z_{i},\lam _{i} )^{2_{\si}^{*}-\tau}  \Big\}+o_{\var_1 }(1)+o(1)
\\\geq&\sum_{i=1}^{2}\Big\{\frac{1}{2} K_{l}(z_{i})^{(2\si-n)/2\si} \| \widetilde{\delta}(0,1)\|^{2}\\&-\frac{N_{\si}}{2_{\si}^{*}} K_{l}(z_{i})^{(2\si-n)/2\si} \|\delta(0,1)\|_E^{2}  \Big\}+o_{\var_1 }(1)+o(1)
	\\=&\sum_{i=1}^{2} \frac{\si N_{\si}}{n} K_{l}(z_{i})^{(2\si-n)/2\si}(S_{n,\si})^{n/\si}+o_{\var_1 }(1)+o(1).
	\end{align*}
Combining this estimate with  the assumption 	$\operatorname{dist}(z_{1}(U), \pa  O_{l}^{(1)})<\delta_{1} /(2 A_{3})$, we obtain
		\begin{align*}
			I_{K_{l}, \tau}(U)	\geq& \frac{\si N_{\si}}{n}(K_{l}(z_{l}^{(1)})-\delta_{1}/2)^{(2\si-n)/2\si}(S_{n,\si})^{n/\si}\\&+\frac{\si N_{\si}}{n} K_{l}(z_{l}^{(2)})^{(2\si-n)/2\si}(S_{n,\si})^{n/\si}+o_{\var_1 }(1)+o(1)\\\geq& \sum_{i=1}^{2} c^{(i)}+1/A_{4},
		\end{align*}
where the  choice of $A_4$ is evident thanks to \eqref{20}-\eqref{21}, \eqref{24}-\eqref{25} and \eqref{28}.
The proof is now complete.
\end{proof}

From now on, the value of $A_{4}$ and $\var_1 $ are fixed.  The main result in this section can be stated as follows:
\begin{thm}\label{thm:3.1}
	Suppose that $\{K_{l}\}$ is a sequence of functions  satisfying (i)-(iii). If there exist some bounded open sets $O^{(1)}, \ldots, O^{(m)} \subset \Rn$ and some  constants $\delta_{2}, \delta_{3}>0$ such that for all $1 \leq i \leq m$,
	\begin{gather}
	\label{1111}
	\widetilde{O}_{l}^{(i)}-z_{l}^{(i)} \subset O^{(i)}\quad \text { for all }\, l,\\\label{29}
	\big\{U \in D^+: I_{K_{\infty}^{(i)}}'(U)=0,\,  c^{(i)} \leq I_{K_{\infty}^{(i)}}(U) \leq c^{(i)}+\delta_{2}\big\}\cap V(1, \delta_{3}, O^{(i)}, K_{\infty}^{(i)})=\emptyset.
	\end{gather}
	Then for any $\var>0$, there exists integer $\over{l}_{\var, m}>0$ such that for any $l \geq \over{l}_{\var, m}$, $\tau\in (0,\over{\tau}_{l})$, there exists  $U_{l}= U_{l, \tau} \in V_{l}(m, \var)\cap D^+$ which solves
\be	\label{30}
	\left\{	\begin{aligned}
	&\operatorname{div}(t^{1-2\si}\nabla U_l)=0&&\text{ in }\,\Rp,\\&\pa_{\nu}^{\si}U_l=K_{l}(x)H^{\tau}(x) U_{l}(x,0)^{\frac{n+2\si}{n-2\si}-\tau}&& \text { on }  \, \Rn.
	\end{aligned}\right.
\ee
	Furthermore, $U_{l}$ satisfies
\be\label{31}
	\sum_{i=1}^{m} c^{(i)}-\var \leq I_{K_{l}, \tau}(U_{l}) \leq \sum_{i=1}^{m} c^{(i)}+\var.
\ee
\end{thm}
\begin{rem}  
	\eqref{31}  is a direct consequence from  the definition of $V_{l}(k, \var)$ once \eqref{30} is proved. 
\end{rem}

We prove Theorem \ref{thm:3.1} by contradiction argument. For simplicity,
we only consider the case   $m = 2$, since  the changes for $m > 2$ are evident. Suppose the contrary of Theorem \ref{thm:3.1}, i.e.,  for some $\var^{*}>0$, there exists a sequence of $l \to  \infty$, $0<\tau_{l}<\over{\tau}_{l}$  such that Eq.  \eqref{30}, for $\tau=\tau_{l}$, has
no solution in $V_{l}(2, \var^{*})$ satisfying \eqref{31} with $\var=\var^{*}$. A  complicated
procedure will be followed in order to  yield a contradiction. It will be outlined now and the details will be given in the next two sections. The proof consists of two parts:
\begin{itemize}
	\item \emph{Part 1.} Under the contrary of Theorem \ref{thm:3.1}, we obtain a uniform lower bounds of the gradient vectors in some certain annular region. It is a standard consequence of the  \emph{Palais-Smale} condition in variational argument.
	\item	\emph{Part 2.} We  use variational method to construct an approximating \emph{minimax} curve. Part 1 can be used to construct a deformation.  In our setting, we  follow  the nonnegative gradient flow to make a deformation, which is an important process  to abtain a counterexample. 
\end{itemize}
Part 1 will be carried out in Section \ref{sec:4} and  Part 2 in Section \ref{sec:5}.
\subsection{First part of the proof of Theorem \ref{thm:3.1}}\label{sec:4}
For $\var_{2}>0$, we denote $\widetilde{V}_{l}(2, \var_{2})$ the set of functions $U$ in $D$ satisfies:  there exist $\al=(\al_{1}, \al_{2}) \in \R^2$, $z=(z_{1}, z_{2}) \in O_{l}^{(1)} \times O_{l}^{(2)}$ and $\lam =(\lam _{1}, \lam _{2})\in \R^2$ such that
\begin{gather*}
\lam _{1},\lam _{2}>\var_{2}^{-1},\\|\lam _{i}^{\tau_{l}}-1|<\var_{2}, ~ i=1,2,\\|\al_{i}-K_{l}(z_{i})^{(2\si-n)/4\si}|<\var_{2}, ~ i=1,2,\\\Big\|U-\sum_{i=1}^{2} \al_{i}\widetilde{\delta}(z_i,\lam _i)^{1+O(\tau_l)}\Big\|<\var_{2}.
\end{gather*}
Throughout the paper, we denote  $p_{l}=\frac{n+2\si}{n-2\si}-\tau_{l}$.
\begin{lem}\label{lem:3.2}
	For $\var_{2}=\var_{2}(\var_1 , \var^{*}, n,\si)>0$ small enough, 
	 we have, for $l$ large enough,
\be\label{32}
	\widetilde{V}_{l}(2, \var_{2}) \subset V_{l}(2, o_{\var_{2}}(1)) \subset V_{l}(2 , \var_1 ) \cap V_{l}(2, \var^{*}),
\ee
	where $o_{\var_{2}}(1)\to  0$ as $\var_{2}\to  0$.  
\end{lem}
\begin{proof}
It  is easy to check \eqref{32}	by using  the definition of $\widetilde{V}_{l}(2, \var_{2})$. Hence we omit it.
\end{proof}
 The following result is the crucial step in the proof of
Theorem \ref{thm:3.1}.
\begin{prop}\label{prop:3.2}
	Under the hypotheses of Theorem \ref{thm:3.1}  and the
	contrary of the conclusion of Theorem \ref{thm:3.1}, there exist some constants $\var_{2} \in(0, \min \{\var_{0},\var_1 , \var^{*}, \delta_{3}\})$ and 
	$\var_{3} \in(0, \min\{\var_{0}, \var_1 , \var_{2}, \var^{*}, \delta_{3}\})$  which are independent of $l$ such that \eqref{32}
	holds for such $\var_{2}$, and there exist $\delta_{4}=\delta_{4}(\var_{2}, \var_{3})>0$ and  $l_{\var_{2}, \var_{3}}'>1$ such that for any
	$l \geq l_{\var_{2}, \var_{3}}'$,  $U \in \widetilde{V}_{l}(2, \var_{2}) \setminus  \widetilde{V}_{l}(2, \var_{2} / 2)$   with $|I_{K_{l}, \tau_{l}}(U)-(c^{(1)}+c^{(2)})|<\var_{3}$,  we have$$
	\|I_{K_{l}, \tau_{l}}'(U)\| \geq \delta_{4},
	$$where $I_{K_{l}, \tau_{l}}'$ denotes the Fr\'echet derivative.
\end{prop}
\begin{rem}
	Proposition \ref{prop:3.2}  will be used to  construct an approximating minimaxing curve in Section \ref{sec:5}.  Evidently, we have, under the contrary of Theorem \ref{thm:3.1}, that for each $l$,    
$$
	\inf\big\{\|I_{K_{l}, \tau_{l}}'(U)\|:U\in  \widetilde{V}_{l}(2, \var_{2}) \setminus  \widetilde{V}_{l}(2, \var_{2} / 2),\, |I_{K_{l}, \tau_{l}}(U)-(c^{(1)}+c^{(2)})|<\var_{3}\big\}>0.
$$
\end{rem}

We prove Proposition \ref{prop:3.2} by contradiction argument. Suppose the statement in the Proposition \ref{prop:3.2} is not true, then no matter how small $\var_{2}, \var_{3}>0$ are, there exists a subsequence (which is still denoted as $\{U_l\}$) such that
\begin{gather}\label{33}
\{U_{l}\} \subset \widetilde{V}_{l}(2, \var_{2}) \setminus  \widetilde{V}_{l}(2, \var_{2} / 2),\\\label{34}
|I_{K_{l}, \tau_l}(U_{l})-(c^{(1)}+c^{(2)})|<\var_{3},\\\label{35}
\lim _{l \to  \infty}\|I_{K_{l}, \tau_{l}}'(U_{l})\|=0.
\end{gather}
However, under the above assumptions, we can prove that there exists another subsequence, still denotes by  $\{U_{l}\}$, such that $U_{l}\in \widetilde{V}_{l}(2, \var_{2} / 2)$, which leads to a contradition. The existence of such sequence needs some lengthy indirect analysis to the  interaction of two  \emph{bubbles}. 
We break the proof of Proposition \ref{prop:3.2} into several claims.

First we write
\be\label{36}
U_{l}=\al_{1}^{l} \widetilde{\delta}( z_{1}^{l}, \lam _{1}^{l} ) +\al_{2}^{l}  \widetilde{\delta}( z_{2}^{l}, \lam _{2}^{l} )+v_{l}
\ee
after making the minimization \eqref{27}. By Proposition \ref{prop:3.1} and some standard arguments in  \cite{BC1,BC2,Li93},  if $\var_{2}>0$ small enough, we have
\begin{gather}
\label{37}(\lam _{1}^{l})^{-1},~ (\lam _{2}^{l})^{-1} = o_{\var_{2}}(1),\\\label{38}
|\al_{i}^{l}-K_{l}(z_{i}^{l})^{(2\si-n)/4\si}| = o_{\var_{2}}(1),  \quad i=1,2,\\\label{39}
\|v_{l}\| = o_{\var_{2}}(1),\\\label{40}
\operatorname{dist}(z_{1}^{l}, O_{l}^{(1)}),~ \operatorname{dist}(z_{2}^{l}, O_{l}^{(2)}) = o_{\var_{2}}(1).
\end{gather}

Next we will derive some elementary estimates of the interaction of two  \emph{bubbles} in \eqref{36}. More precisely,  another  representation of $U_l$ in \eqref{36} will be found and therefore we can  deduce its  location and concentrate rate easily.	Let us  introduce a linear isometry operator first.  	

For $z \in \Rn$,  we define  $T_{z}: D\to  D$  by
$$
T_{z} U(x,t):=U(x+z,t).
$$It is easy to see that $\|T_{z} U\|=\|U\|$. 

 Now we   give \emph{bubble's} profile of \eqref{36}.
\begin{claim}\label{claim:1}
For $\var_{2}>0$  small enough, we have	$$
\lim _{l \to  \infty} \lam _{1}^{l}=\lim _{l \to \infty} \lam _{2}^{l}=\infty.
$$
\end{claim}	
\begin{proof}
Assume to the contrary that $	\lam _{1}^{l}=\lam _{1}+o(1)<\infty$ up to a subsequence. Here and in the following, let $o(1)$ denote any sequence tending to 0 as $l\to  \infty$. Now the proof consists of three steps.

\textbf{ Step 1} (Construct a positive solution). One observes from  \eqref{36} that$$
	T_{z_{1}^{l}} U_{l}=\al_{1}^{l} \widetilde{\delta}( 0,\lam _{1}^{l} )+\al_{2}^{l}  \widetilde{\delta}(z_{2}^{l}-z_{1}^{l}, \lam _{2}^{l})+T_{z_{1}^{l}} v_{l}.
	$$
Then by Proposition \ref{prop:3.1},	after passing to a subsequence, we have
\begin{gather}
	\lim _{l \to  \infty} \al_{1}^{l}=\al_{1} \in\Big[\frac{1}{2}(A_{2})^{(2\si-n)/4\si}-o_{\var_{2}}(1), 2(A_{2})^{(2\si-n)/4\si}+o_{\var_{2}}(1)\Big], 
\label{alpha1}	\\	\lim _{l \to  \infty} \al_{2}^{l}=\al_{2} \in\Big[\frac{1}{2}(A_{2})^{(2\si-n)/4\si}-o_{\var_{2}}(1), 2(A_{2})^{(2\si-n)/4\si}+o_{\var_{2}}(1)\Big], \label{alpha2}
\end{gather}
and $$	T_{z_{1}^{l}} v_{l} \rightharpoonup w_0 \quad \text { weakly  in }\, D$$
for some $w_0 \in D$.  From the lower semi-continuity of the norm and \eqref{39}, we infer that 
\be\label{41}
	\|w_0\| \leq \varliminf_{l \to  \infty}\|T_{z_{1}^{l}} v_{l}\|= o_{\var_{2}}(1).
\ee
Using the assumption  (ii) (stated in the beginning of Section \ref{sec:3}), we get \be\label{42}
	\lim _{l \to  \infty}|z_{1}^{l}-z_{2}^{l}| \geq \lim _{l \to  \infty} R_{l}=\infty.
\ee
		Therefore,  
\be\label{43}
	T_{z_{1}^{l}} U_{l} \rightharpoonup W:=\al_{1}  \widetilde{\delta}( 0,\lam _{1} )+w_0\quad \text { weakly  in }\, D.
\ee
Obviously, $W \neq 0$ if $\var_{2}$ is small enough.
	
	Next we prove that $W$ is a weak solution of the following equation
\be\label{44}
	\left\{\begin{aligned}
	&\operatorname{div}(t^{1-2\si}\nabla W)=0&&\text{ in }\,\Rp,\\&\pa_{\nu}^{\si}W=T_{\zeta}K_{\infty}^{(1)}|W(x,0)|^{\frac{4\si}{n-2\si}}W(x,0) && \text { on }\, \Rn,
	\end{aligned}\right.
\ee
	where $\zeta \in O^{(1)}$ with $\operatorname{dist}(\zeta , \pa  O^{(1)})>\delta_{0} / 2$. Note that we have abused notation a bit, only in this proof	we write $T_{\zeta}K_{\infty}^{(1)}=K_{\infty}^{(1)}(\cdot+\zeta)$.
	
	For any  $\phi \in C_{c}^{\infty}(\over{\R}^{n+1}_{+})$, it follows from \eqref{35} that
	$$
	I_{K_{l},\tau_{l}}'(U_{l})(T_{-z_{1}^{l}} \phi) =o(1)\|T_{-z_{1}^{l}} \phi\| =o(1)\|\phi\| =o(1).
	$$
Summing up \eqref{22}, \eqref{25}, \eqref{26}  and \eqref{43}, we find that
	\begin{align*}
	o(1)=&\big \langle U_{l},T_{-z_{1}^{l}}\phi\big\rangle-N_{\si}\int_{\Rn} K_{l}H^{\tau_l}|U_{l}|^{p_l-1}U_{l}T_{-z_{1}^{l}} \phi
	\\=&\big \langle T_{z_{1}^{l}} U_{l},\phi\big\rangle-N_{\si}\int_{\Rn} K_{l}(\cdot+z_{1}^{l})H^{\tau_l}(\cdot+z_{1}^{l})|T_{z_{1}^{l}} U_{l}|^{p_l-1}(T_{z_{1}^{l}}U_{l}) \phi
	\\=&\big\langle W ,\phi\big\rangle -N_{\si}\int_{\Rn}  T_{\zeta}K_{\infty}^{(1)}|W|^{\frac{4\si}{n-2\si}}W\phi+o(1),
	\end{align*}
	where $\zeta=\lim _{l \to  \infty}(z_{1}^{l}-z_{l}^{(1)})$  along a subsequence. This means $W$ is a weak solution of  \eqref{44}.	 
	
	The positivity of $W$ can be verified from the following argument.
	
	Let us write $W=W^{+}-W^{-}$, where $W^{+}=\max (W, 0)$, $W^{-}=\min (W, 0)$. It follows from  \eqref{41}, \eqref{43} and \eqref{trace} that
	 \be\label{claim1-1}
	\int_{\Rn}(W^{-})^{2_{\si}^{*}}\,\d x = o_{\var_{2}}(1).
\ee Multiplying \eqref{44} 
	by $W^{-}$ and integrating by parts, we have
\begin{align}
\int_{\Rp}t^{1-2\si}|\nabla W^-|^2\, \d X=&N_{\si}\int_{\Rn}T_{\zeta}K_{\infty}^{(1)}(W^{-})^{2^{*}_{\si}}\,\d x\notag\\\leq & o_{\var_{2}}(1)\Big(\int_{\Rn}(W^{-})^{2^{*}_{\si}}\,\d x\Big)^{2/2^{*}_{\si}}\notag\\\leq & o_{\var_{2}}(1)\int_{\Rp}t^{1-2\si}|\nabla W^-|^2\, \d X,
\end{align}	
where we used \eqref{claim1-1} in the first inequality and  \eqref{trace} in the second inequality.	 Hence, if $\var_{2}>0$ is small enough, we immediately obtain $W^{-} \equiv 0$, namely, $W \geq 0$. It folllows from 	\eqref{44} and the strong maximum principle in \cite{XaYa,JLX}  that $W > 0$.
	
	\textbf{ Step 2} (Energy bound estimates).  Now we begin to estimate  the value of $I_{T_{\zeta}K_{\infty}^{(1)}} (W)$ in order to obtain a  contradiction. The estimate we are going to establish is
\be\label{45}
	c^{(1)} \leq I_{T_{\zeta}K_{\infty}^{(1)}}(W) \leq c^{(1)}+o_{\var_{2}}(1),
\ee
	where $o_{\var_{2}}(1)\to  0$ as $\var_{2}\to  0$.
	
 Firstly, multiplying  \eqref{44} by $W$ and  integrating by parts, we have
$$
\int_{\Rp}t^{1-2\si}|\nabla W|^2\, \d X=N_{\si}\int_{\Rn}  T_{\zeta}K_{\infty}^{(1)}W^{2^{*}_{\si}}\,\d x.
$$
	This implies that
$$
	I_{T_{\zeta} K_{\infty}^{(1)}}(W)=\frac{1}{2} \int_{\Rp}t^{1-2\si}|\nabla W|^2\, \d X-\frac{N_{\si}}{2^{*}_{\si}}\int_{\Rn}  T_{\zeta}K_{\infty}^{(1)}W^{2^{*}_{\si}}\,\d x=\frac{\si N_{\si}}{n} \int_{\Rn}  T_{\zeta}K_{\infty}^{(1)}W^{2^{*}_{\si}}\,\d x.
$$
We thus conclude from  \eqref{trace} and the upper bound  $T_{\zeta}K_{\infty}^{(1)} \leq a^{(1)}$ that
	\begin{align*}
	\mathcal{S}_{n,\si} \leq& \frac{\big(\int_{\Rp}t^{1-2\si}|\nabla W|^{2}\,\d X\big)^{1 / 2} }{\big(\int_{\Rn}W^{2^{*}_{\si}}\,\d x\big)^{1/2^{*}_{\si}}}  \leq \frac{\big(\int_{\Rp}t^{1-2\si}|\nabla W|^{2}\,\d X\big)^{1 / 2}}{\big(\int_{\Rn}  T_{\zeta}K_{\infty}^{(1)}W^{2^{*}_{\si}}\,\d x\big)^{1/2^{*}_{\si}}}(a^{(1)})^{1/2^{*}_{\si}} \\
	=&\Big(\int_{\Rn}  T_{\zeta}K_{\infty}^{(1)}W^{2^{*}_{\si}}\,\d x\Big)^{\si/n}N_{\si}^{1/2}(a^{(1)})^{1/2^{*}_{\si}},
	\end{align*}
namely,
$$	S_{n,\si} \leq \Big(\int_{\Rn}  T_{\zeta}K_{\infty}^{(1)}W^{2^{*}_{\si}}\,\d x\Big)^{\si/n}(a^{(1)})^{1/2^{*}_{\si}}.
$$	
	Therefore, 	we complete the  proof of the first inequality in  \eqref{45}.

	On the other hand, we deduce from \eqref{18} that  $1/A_1\leq  K_{\infty}^{(1)}(x) \leq A_{1}$, $\forall\, x\in \Rn$.
Owing to  \eqref{36}-\eqref{37}, \eqref{39} and \eqref{42}, we have
	\begin{align*}
	I_{K_{l}, \tau_{l}}(U_{l})&=I_{K_{l}, \tau_{l}}(\al_{1}^{l}  \widetilde{\delta}( z_{1}^{l},\lam _{1}^{l} ))+I_{K_{l}, \tau_{l}}(\al_{2}^{l}  \widetilde{\delta}( z_{2}^{l},\lam _{2}^{l} ))+o_{\var_{2}}(1)\\&
	=I_{T_{z_{1}^l}K_{l}, \tau_{l}}(\al_{1}^{l} \widetilde{\delta}( 0,\lam _{1}^{l} ))+I_{K_{l},\tau_{l}}(\al_{2}^{l}  \widetilde{\delta}( z_{2}^{l},\lam _{2}^{l} ))+o_{\var_{2}}(1)\\&
	=I_{T_{z_{1}^l}K_{l}, \tau_{l}}(\al_{1} \widetilde{\delta}( 0,\lam _{1} ))+I_{K_{l}, \tau_{l}}(\al_{2}^{l}  \widetilde{\delta}( z_{2}^{l},\lam _{2}^{l} ))+o_{\var_{2}}(1)+o(1)\\&=I_{ T_{\zeta}K_{\infty}^{(1)}}(\al_{1}  \widetilde{\delta}( 0,\lam _{1} ))+I_{K_{l}, \tau_{l}}(\al_{2}^{l} \widetilde{\delta}( z_{2}^{l},\lam _{2}^{l} ))+o_{\var_{2}}(1)+o(1)
	\\&=I_{T_{\zeta}K_{\infty}^{(1)}}(W)+I_{K_{l}, \tau_{l}}(\al_{2}^{l}  \widetilde{\delta}( z_{2}^{l},\lam _{2}^{l} ))+o_{\var_{2}}(1)+o(1).
	\end{align*} 
Consequently,
\be\label{47}
	I_{T_{\zeta}K_{\infty}^{(1)}}(W)=I_{K_{l}, \tau_{l}}(U_{l})-I_{K_{l}, \tau_{l}}(\al_{2}^{l}  \widetilde{\delta}( z_{2}^{l},\lam _{2}^{l} ))+o_{\var_{2}}(1)+o(1).
	\ee
Combining \eqref{35} and \eqref{36}, we find$$
	o(1)=
	I_{K_{l}, \tau_{l}}'(U_{l})(\al_{2}^{l}  \widetilde{\delta}( z_{2}^{l},\lam _{2}^{l} ))=I_{K_{l},\tau_{l} }'(\al_{2}^{l}  \widetilde{\delta}( z_{2}^{l},\lam _{2}^{l} ))(\al_{2}^{l}  \widetilde{\delta}( z_{2}^{l},\lam _{2}^{l} ))+o_{\var_{2}}(1)+o(1),
$$
	namely,\begin{gather}
		\label{48}
	\|\al_{2}^{l}  \widetilde{\delta}( z_{2}^{l},\lam _{2}^{l} ) \|^{2}=N_{\si}\int_{\Rn} K_{l}H^{\tau_l}(\al_{2}^{l} \delta( z_{2}^{l},\lam _{2}^{l} ))^{2_{\si}^{*}-\tau_{l}}+o_{\var_{2}}(1)+o(1),\\	\label{49}I_{K_{l}, \tau_{l}}(\al_{2}^{l}   \widetilde{\delta}( z_{2}^{l},\lam _{2}^{l} ))=\frac{\si}{n}  \|\al_{2}^{l}  \widetilde{\delta}( z_{2}^{l},\lam _{2}^{l} )\|^{2}+o_{\var_{2}}(1)+o(1).
\end{gather}
	From \eqref{alpha2}, we obtain$$
\|\al_{2}^{l}  \widetilde{\delta}(z_{2}^{l},\lam _{2}^{l} )\|^{2} \geq\Big(\frac{1}{2}(A_{2})^{(2\si-n)/4\si}-o_{\var_{2}}(1)\Big)N_{\si}(S_{n,\si})^{n/\si} >\frac{1}{4}(A_{2})^{(2\si-n)/4\si}N_{\si}(S_{n,\si})^{n/\si}>0.
$$
Then,		by  \eqref{trace},  \eqref{19}-\eqref{21},  \eqref{37}, \eqref{42}, and  H\"{o}lder inequality,	we have
	\begin{align*}
	\mathcal{S}_{n,\si}&\leq \frac{\big(\int_{\Rp}t^{1-2\si}|\nabla(\al_{2}^{l}  \widetilde{\delta}( z_{2}^{l},\lam _{2}^{l} ))|^{2}\,\d X\big)^{1/2}}{\big(\int_{\Rn}(\al_{2}^{l} \delta( z_{2}^{l},\lam _{2}^{l} ))^{2_{\si}^{*}}\big)^{1/2_{\si}^{*}}}
\\&	\leq \frac{\big(\int_{\Rp}t^{1-2\si}|\nabla(\al_{2}^{l}  \widetilde{\delta}( z_{2}^{l},\lam _{2}^{l} ))|^{2}\,\d X\big)^{1/2}}{\big(\int_{B_{R_{l}}(z_{l}^{(2)})}(\al_{2}^{l} \delta( z_{2}^{l},\lam _{2}^{l} ))^{2_{\si}^{*}}\big)^{1/2^{*}_{\si}}+o(1)}\\	&\leq \frac{\big(\int_{\Rp}t^{1-2\si}|\nabla(\al_{2}^{l}  \widetilde{\delta}( z_{2}^{l},\lam _{2}^{l} ))|^{2}\,\d X\big)^{1/2}\cdot K_{l}(z_{l}^{(2)})^{1/2^{*}_{\si}}}{\big(\int_{B_{R_{l}}(z_{l}^{(2)})} K_{l}H^{\tau_l}(\al_{2}^{l}  \delta( z_{2}^{l}, \lam _{2}^{l} ))^{2_{\si}^{*}-\tau_{l}}\big)^{1/2^{*}_{\si}}+o(1)}
\\&=\frac{\big(\int_{\Rp}t^{1-2\si}|\nabla (\al_{2}^{l}  \widetilde{\delta}( z_{2}^{l},\lam _{2}^{l} ))|^{2}\,\d X\big)^{1/2} \cdot(a^{(2)})^{1/2^{*}_{\si}}+o(1)}{\big(\int_{\Rn} K_{l}H^{\tau_l}(\al_{2}^{l}  \delta( z_{2}^{l},\lam _{2}^{l} ))^{2_{\si}^{*}-\tau_{l}}\big)^{1/2^{*}_{\si}}+o(1)}.
	\end{align*}
Thus,	using  \eqref{48}, we establish that
$$
	\mathcal{S}_{n,\si} \leq\|\al_{2}^{l}  \widetilde{\delta}( z_{2}^{l},\lam _{2}^{l} )\|^{2\si / n}(a^{(2)})^{1/2^{*}_{\si}} N_{\si}^{1/2^*_{\si}} +o(1),
$$namely,
\be\label{claimestimate}
S_{n,\si} \leq\|\al_{2}^{l}  \widetilde{\delta}( z_{2}^{l},\lam _{2}^{l} )\|^{2\si / n}(a^{(2)})^{1/2^{*}_{\si}} N_{\si}^{-\si/n} +o(1).
\ee
	This 	together with  \eqref{49} gives
		\begin{align}
		I_{K_{l}, \tau_{l}}(\al_{2}^{l} \widetilde{\delta}( z_{2}^{l},\lam _{2}^{l} )) &\geq \frac{\si N_{\si}}{n}(a^{(2)})^{(2\si-n)/2\si}(S_{n,\si})^{n/\si}+o_{\var_{2}}(1)+o(1)\notag\\&=c^{(2)}+o_{\var_{2}}(1)+o(1).\label{a}
		\end{align}
		Putting  \eqref{34}, \eqref{47} and \eqref{a} together, we obtain the
 right hand side of \eqref{45}.
	
\textbf{Step 3} (Completion of the proof). Finally,  a contradiction arises from \eqref{29}, \eqref{43}-\eqref{45}, and the positivity of $W$ for $\var_{2}>0$ small enough. This proves that $\lim_{l\to  \infty } \lam _{1}^{l}=\infty$. Similarly we  have  $\lim_{l\to  \infty } \lam _{2}^{l}=\infty$. Claim \ref{claim:1} has been established.
\end{proof}

For any $\lam >0$ and $z \in \Rn$, we define $\mathscr{T}_{l,\lam , z}: D\to  D$ by $$\mathscr{T}_{l, \lam , z} U(x,t):=\lam ^{2\si / (1-p_{l})} U(\lam ^{-1}x+z,\lam ^{-1}t).$$It is clear that  $$\mathscr{T}_{l, \lam , z}^{-1} U(x,t)=\lam ^{2 \si/ (p_l-1)}  U(\lam (x-z),\lam  t).$$
\begin{lem}\label{lem:3.3}
	There exists some constant $C=C(n,\si,A_{2})>0$ such that for
	small $\var_{2}$ and large $l$, we have 
$$
	(\lam _{1}^{l})^{\tau_{l}} ,~ (\lam _{2}^{l})^{\tau_{l}} \leq C.
	$$
\end{lem}
\begin{proof}
Applying \eqref{35}, we deduce that
\be\label{51}
	I_{K_{l}, \tau_{l}}'(U_{l})( \widetilde{\delta}( z_{1}^{l},\lam _{1}^{l} ))=o(1).
\ee
Now an explicit calculation from  \eqref{39},  \eqref{42},  Claim \ref{claim:1}, and Proposition \ref{prop:3.1} yields that
\begin{gather*}\big\langle  \widetilde{\delta}( z_{1}^{l},\lam _{1}^{l} ) , v_{l}\big\rangle=0,\\
\big\langle   \widetilde{\delta}( z_{1}^{l},\lam _{1}^{l} ),  \widetilde{\delta}( z_{2}^{l},\lam _{2}^{l} )\big\rangle =o(1), \\\int_{\Rn}  K_{l}H^{\tau_l} \delta( z_{2}^{l},\lam _{2}^{l} )^{p_{l}}  \delta( z_{1}^{l},\lam _{1}^{l} )=o(1),\\\int_{\Rn}  K_{l}H^{\tau_l} v_{l}^{p_{l}}  \delta( z_{1}^{l},\lam _{1}^{l} )=o_{\var_{2}}(1).
\end{gather*}
	Putting together the above estimates, we have
\be\label{52}
	(\al_{1}^{l})^{p_{l}}  \int_{\Rn} K_{l}H^{\tau_l} \delta( z_{1}^{l},\lam _{1}^{l} )^{p_{l}+1} =\al_{1}^{l} \| \widetilde{\delta}( z_{1}^{l},\lam _{1}^{l} )\|^2
	+o_{\var_{2}}(1)+o(1).
\ee
Then the proof of the first term completed from \eqref{23}, \eqref{26},  \eqref{38}, \eqref{40}, \eqref{52},    and Claim \ref{claim:1}. Similarly we have $(\lam _{2}^{l})^{\tau_{l}} \leq C$.	
\end{proof}
Without loss of generality, we can always assume that
\be\label{53}
\lam _{1}^{l} \leq \lam _{2}^{l}.
\ee

A direct computation using \eqref{36} shows that 
\be\label{54}
\mathscr{T}_{l, \lam _{1}^{l}, z_{1}^{l}} U_{l}=\widetilde{\al}_{1}^{l}  \widetilde{\delta}(0,1)+\widetilde{\al}_{2}^{l}\widetilde{\delta}(\lam _{1}^{l}(z_{2}^{l}-z_{1}^{l}),\lam _{2}^{l} / \lam _{1}^{l}) +\mathscr{T}_{l, \lam _{1}^{l}, z_{1}^{l}} v_{l},
\ee where
$$
\widetilde{\al}_{1}^{l}=\al_{1}^{l}(\lam _{1}^{l})^{(n-2\si) / 2-2\si /(p_{l}-1)},\quad \widetilde{\al}_{2}^{l}=\al_{2}^{l}(\lam _{1}^{l})^{(n-2\si) / 2-2\si /(p_{l}-1)}.
$$
Then we can verify the existence  of $U_{1} \in D$ and $\xi_{1} \in O^{(1)}$ such that
\begin{gather}
\label{55}
\mathscr{T}_{l, \lam _{1}^{l}, z_{1}^{l}} U_{l} \rightharpoonup U_{1} \quad \text { weakly in }\, D,\\\label{56}\lim _{l \to  \infty}(z_{1}^{l}-z_{l}^{(1)})=\xi_{1},
\end{gather}up to a subsequence.

Accordingly, by making use of \eqref{22}, \eqref{25}, \eqref{56} and \eqref{40}, we have \be\label{57}
\lim _{l \to  \infty} K_{l}(z_{1}^{l})^{(2\si-n)/4\si}=K_{\infty}^{(1)}(\xi_{1})^{(2\si-n)/4\si}.
\ee

Now, one observes from \eqref{35} that for any $\phi \in C_{c}^{\infty}(\over{\R}^{n+1}_+)$, we get
\begin{align*}
o(1)=&I_{K_{l}, \tau_{l}}'(U_{l})(\mathscr{T}^{-1}_{l, \lam _{1}^{l}, z_{1}^{l}} \phi)\\=&
\big\langle U_l,\mathscr{T}_{l, \lam _{1}^{l}, z_{1}^{l}}^{-1} \phi\big\rangle-N_{\si}\int_{\Rn} K_{l}H^{\tau_l}|U_{l}|^{p_{l}-1}U_{l} \mathscr{T}_{l, \lam _{1}^{l}, z_{1}^{l}}^{-1} \phi\\=&(\lam _{1}^{l})^{4\si /(p_{l}-1)+2\si-n}\Big\{\big\langle \mathscr{T}_{l, \lam _{1}^{l}, z_{1}^{l}} U_{l}, \phi\big\rangle-N_{\si}\int_{\Rn}  K_{l}(\cdot/\lam _{1}^{l}+z_{1}^{l})\\ &\times  H^{\tau_l}(\cdot/ \lam _{1}^{l}+z_{1}^{l}) |\mathscr{T}_{l, \lam _{1}^{l}, z_{1}^{l}} U_{l}|^{p_{l}-1}(\mathscr{T}_{l, \lam _{1}^{l}, z_{1}^{l}} U_{l}) \phi\Big\}.
\end{align*}Taking the limit $l \to  \infty$, and then using  \eqref{22}, \eqref{26},  \eqref{40},  \eqref{55}-\eqref{56},  and  Lemma \ref{lem:3.3}, we obtain
$$
\big\langle U_{1} ,\phi\big\rangle-N_{\si}\int_{\Rn} K_{\infty}^{(1)}(\xi_{1})|U_{1}(x,0)|^{\frac{4\si}{n-2\si}} U_{1}(x,0)\,\d x=0.
$$
Namely, $U_1$ satisfies
\be\label{58}
\left\{\begin{aligned}
&\operatorname{div}(t^{1-2\si}\nabla U_1)=0&&\text{ in }\,\Rp,\\&\pa_{\nu}^{\si}U_1=K_{\infty}^{(1)}(\xi_{1})|U_{1}(x,0)|^{\frac{4\si}{n-2\si}} U_{1}(x,0) && \text { on }\, \Rn.
\end{aligned}\right.
\ee
Moreover, we see from \eqref{54} that $U_{1}$  is not identically zero if $\var_{2}$ is small enough.  Then we can argue as before to obtain $U_1>0$.
By the classification theorem in \cite[Proposition 1.3]{MQ} or \cite[Theorem 1.8]{JLX}, there exists $\lam ^{*}>0$ and  $z^{*} \in \Rn$ such that
\be\label{59}
U_{1}=K_{\infty}^{(1)}(\xi_{1})^{(2\si-n)/4\si} \widetilde{\delta}(z^{*},\lam ^{*}).
\ee

\begin{claim}\label{claim:2}
	For $l$ large enough, we have 
$$
	|z^{*}| = o_{\var_{2}}(1), \quad |\lam ^{*}-1| = o_{\var_{2}}(1),\quad (\lam _{1}^{l})^{\tau_{l}}=1+o_{\var_{2}}(1).
$$
\end{claim}
\begin{proof} 
First of all,  using Lemma \ref{lem:3.3},  we find  $$\lim_{l \to  \infty}(\lam _{1}^{l})^{\tau_{l}}=A_{\var_{2}, \var_{3}}$$
along a subsequence,   where $A_{\var_{2}, \var_{3}}$ is a positive constant  independent of $l$ for fixed $\var_{2}$ and $\var_{3}$.  Thanks to	 \eqref{38} and \eqref{57}, we obtain
$$
\al_{1}^{l}=K_{\infty}^{(1)}(\xi_{1})^{(2\si-n)/4\si}+o_{\var_{2}}(1)+o(1).
$$
Note  that	\be\label{60}
	\widetilde{\al}_{1}^{l}=\al_{1}^{l}(\lam _{1}^{l})^{(n-2\si) / 2-2\si /(p_{l}-1)}=\al_{1}^{l}(\lam _{1}^{l})^{-(n-2\si)^{2}\tau_{l} / 8\si +O(\tau_{l}^{2})}.
\ee
Then it can be computed that
\be\label{61}
	\widetilde{\al}_{1}^{l}=K_{\infty}^{(1)}(\xi_{1})^{(2\si-n)/4\si}(A_{\var_{2}, \var_{3}})^{-(n-2\si)^{2} / 8\si}+o_{\var_{2}}(1)+o(1).
\ee
Moreover,	by \eqref{37},  \eqref{42} and  \eqref{53}-\eqref{55}, we have
\be\label{62}
	\widetilde{\al}_{1}^{l}  \widetilde{\delta}(0,1)+\mathscr{T}_{l, \lam _{1}^{l}, z_{1}^{l}} v_{l} \rightharpoonup U_{1} \quad \text { weakly in } \, D.
\ee
	Consequently,	it follows from \eqref{39}, \eqref{61}-\eqref{62},  and Lemma \ref{lem:3.3} that$$
	\big\| K_{\infty}^{(1)}(\xi_{1})^{(2\si-n)/4\si} (A_{\var_{2}, \var_{3}})^{-(n-2\si)^{2} / 8\si} \widetilde{\delta}(0,1)-K_{\infty}^{(1)}(\xi_{1})^{(2\si-n)/4\si}  \widetilde{\delta}( z^{*}, \lam ^{*} ) \big\|=o_{\var_{2}}(1)+o(1).
	$$
Finally, taking the limit $l\to  \infty$, we get
$$|z^{*}|=o_{\var_{2}}(1), \quad  \lam ^{*}=1+o_{\var_{2}}(1), \quad  A_{{\var_{2}, \var_{3}}}=1+o_{\var_{2}}(1).	$$
	\end{proof}	
We define $\xi_{l} \in D$ by
\be\label{63}
\mathscr{T}_{l, \lam _{1}^{l}, z_{1}^{l}} U_{l}=U_{1}+\mathscr{T}_{l, \lam _{1}^{l}, z_{1}^{l}} \xi_{l}.
\ee
It follows from \eqref{55} that
\be\label{63'}
\mathscr{T}_{l, \lam _{1}^{l}, z_{1}^{l}} \xi_{l} \rightharpoonup 0  \quad  \text { weakly in }\, D.
\ee
\begin{claim}\label{claim:3}
For $\var_{2}$  small enough, we have	$\|I_{K_{l}, \tau_{l}}'(\xi_{l})\|=o(1)$.
\end{claim}
\begin{proof}
	For any $\phi \in C_{c}^{\infty}(\over{\R}^{n+1}_+)$, it follows from \eqref{35}, \eqref{63}, \eqref{58}, and Lemma \ref{lem:3.3} that
	\begin{align}
o(1)\|\phi\|=&I_{K_{l}, \tau_{l}}'(U_{l})(\mathscr{T}^{-1}_{l, \lam _{1}^{l}, z_{1}^{l}} \phi)\notag\\=&
	\big\langle U_{l},\mathscr{T}_{l, \lam _{1}^{l}, z_{1}^{l}}^{-1} \phi\big\rangle
	-N_{\si}\int_{\Rn} K_{l}H^{\tau_l}|U_{l}|^{p_{l}-1}U_{l} \mathscr{T}_{l, \lam _{1}^{l}, z_{1}^{l}}^{-1} \phi\notag\\
	=&(\lam _{1}^{l})^{4 \si /(p_{l}-1)+2 \si-n}\Big\{\big\langle \mathscr{T}_{l, \lam _{1}^{l}, z_{1}^{l}} U_{l}, \phi\big\rangle 
	-N_{\si}\int_{\Rn} K_{l}(\cdot/\lam _{1}^{l}+z_{1}^{l})\notag\\&\times  H^{\tau_l}(\cdot/ \lam _{1}^{l}+z_{1}^{l}) |\mathscr{T}_{l, \lam _{1}^{l}, z_{1}^{l}} U_{l}|^{p_{l}-1}(\mathscr{T}_{l, \lam _{1}^{l}, z_{1}^{l}} U_{l}) \phi\Big\}\notag\\=&(\lam _{1}^{l})^{4 \si /(p_{l}-1)+2\si-n}\Big\{\big\langle  U_{1} , \phi \rangle +\big\langle \mathscr{T}_{l, \lam _{1}^{l}, z_{1}^{l}} \xi_{l},\phi\big\rangle-N_{\si}\int_{\Rn} K_{l}(\cdot/\lam _{1}^{l}+z_{1}^{l})\notag\\&\times H^{\tau_l}(\cdot/\lam _{1}^{l}+z_{1}^{l}) |\mathscr{T}_{l, \lam _{1}^{l},z_{1}^{l}} U_{l}|^{p_{l}-1}(\mathscr{T}_{l, \lam _{1}^{l}, z_{1}^{l}} U_{l}) \phi\Big\}\notag	\\=&I_{K_{l},\tau_{l} }'(\xi_{l})(\mathscr{T}_{l, \lam _{1}^{l}, z_{1}^{l}}^{-1} \phi)+(\lam _{1}^{l})^{4 \si /(p_{l}-1)+2 \si-n}N_{\si}\Big\{\int_{\Rn} K_{\infty}^{(1)}(\xi_{1})U_{1}^{2^{*}_{\si}-1} \phi\notag\\&+\int_{\Rn} K_{l}(\cdot/\lam _{1}^{l}+z_{1}^{l})H^{\tau_l}(\cdot / \lam _{1}^{l}+z_{1}^{l})|\mathscr{T}_{l, \lam _{1}^{l}, z_{1}^{l}} \xi_{l}|^{p_{l}-1}(\mathscr{T}_{l, \lam _{1}^{l}, z_{1}^{l}} \xi_{l}) \phi\notag\\&-\int_{\Rn} K_{l}(\cdot/\lam _{1}^{l}+z_{1}^{l}) H^{\tau_l}(\cdot / \lam _{1}^{l}+z_{1}^{l}) |\mathscr{T}_{l, \lam _{1}^{l}, z_{1}^{l}} U_{l}|^{p_{l}-1}(\mathscr{T}_{l, \lam _{1}^{l}, z_{1}^{l}} U_{l}) \phi\Big\}.\label{64}
\end{align}
Then a direct calculation exploiting \eqref{26}, \eqref{57}, \eqref{59},   Claim \ref{claim:2}, H\"{o}lder inequality, and the  Sobolev embedding theorem shows
\be\label{65}
	\Big| \int_{\Rn} K_{l}(\cdot/\lam _{1}^{l}+z_{1}^{l})H^{\tau_l}(\cdot / \lam _{1}^{l}+z_{1}^{l})U_{1}^{p_{l}} \phi -\int_{\Rn} K_{\infty}^{(1)}(\xi_{1})U_{1}^{2^{*}_{\si}-1} \phi \Big|=o(1)\|\phi\|.
\ee
Finally, by \eqref{64}-\eqref{65}, Lemma \ref{lem:3.3},  and some elementary inequalities,  we deduce that
	\begin{align*}
	&|I_{K_{l}, \tau_{l}}'(\xi_{l})(\mathscr{T}_{l, \lam _{1}^{l}, z_{1}^{l}} \phi)|\\=& o(1)\|\phi\|+O(1) \int_{\Rn}(|\mathscr{T}_{l, \lam _{1}^{l}, z_{1}^{l}} \xi_{l}|^{p_{l}-1} U_{1}+|\mathscr{T}_{l, \lam _{1}^{l}, z_{1}^{l}} \xi_{l}|U_{1}^{p_{l}-1})|\phi|\\=& o(1)\|\phi\|,
	\end{align*}
where	the last inequality follows from \eqref{63'}, Claim \ref{claim:2}, H\"older inequality, and the  Sobolev embedding theorem.
	Claim \ref{claim:3} has been established.
\end{proof}
\begin{claim}\label{claim:4}
	$I_{K_{l}, \tau_{l}}\left(\xi_{l}\right) \leq c^{(2)}+\var_{3}+o(1)$.
\end{claim}
\begin{proof} By  a change of variable and using Claim \ref{claim:2} and \eqref{63}, some calculations lead to
	\begin{align}
	I_{K_{l}, \tau_{l}}(U_{l})=&\frac{1}{2} \| U_{l}\|^{2}-\frac{N_{\si}}{p_{l}+1} \int_{\Rn} K_{l}H^{\tau_l}|U_{l}|^{p_{l}+1}\notag
	\\=&(\lam _{1}^{l})^{4 \si /(p_{l}-1)+2 \si-n}\Big\{\frac{1}{2} \big\| \mathscr{T}_{l, \lam _{1}^{l}, z_{1}^{l}} U_{l}\big\|^{2}-\frac{N_{\si}}{2^{*}_{\si}} \int_{\Rn} K_{l}(\cdot/\lam _{1}^{l}+z_{1}^{l})\notag\\&\times H^{\tau_l}(\cdot / \lam _{1}^{l}+z_{1}^{l})|\mathscr{T}_{l, \lam _{1}^{l}, z_{1}^{l}} U_{l}|^{p_{l}+1}\Big\}+o(1)\notag\\=&(\lam _{1}^{l})^{4 \si /(p_{l}-1)+2 \si-n}\Big\{\frac{1}{2} \|U_{1}\|^{2}+\big\langle U_{1},\mathscr{T}_{l, \lam _{1}^{l}, z_{1}^{l}} \xi_{l}\big\rangle+\frac{1}{2} \big\|\mathscr{T}_{l, \lam _{1}^{l}, z_{1}^{l}} \xi_{l} \big\|^{2}\notag\\
	&-\frac{N_{\si}}{2^{*}_{\si}} \int_{\Rn} K_{l}(\cdot/\lam _{1}^{l}+z_{1}^{l})H^{\tau_l}(\cdot / \lam _{1}^{l}+z_{1}^{l})U_{1}^{p_{l}+1}\notag\\
	&-\frac{N_{\si}}{2^{*}_{\si}} \int_{\Rn} K_{l}(\cdot/\lam _{1}^{l}+z_{1}^{l})H^{\tau_l}(\cdot / \lam _{1}^{l}+z_{1}^{l})|\mathscr{T}_{l, \lam _{1}^{l}, z_{1}^{l}} \xi_{l}|^{p_{l}+1}\notag\\
	&-O(1) \int_{\Rn}(|\mathscr{T}_{l, \lam _{1}^{l}, z_{1}^{l}} \xi_{l}|^{p_{l}} U_{1}+|\mathscr{T}_{l, \lam _{1}^{l}, z_{1}^{l}} \xi_{l}| U_{1}^{p_{l}})\Big\}+o(1)\notag\\
=&I_{K_{l}, \tau_{l}}(\xi_{1})+(\lam _{1}^{l})^{2(p_{l}+1)/(p_{l}-1)-n}\Big\{\frac{1}{2} \| U_{1}\|^{2}-\frac{N_{\si}}{2^{*}_{\si}} \int_{\Rn} K_{l}(\cdot/\lam _{1}^{l}+z_{1}^{l})\notag\\&\times H^{\tau_l}(\cdot / \lam _{1}^{l}+z_{1}^{l})U_{1}^{p_{l}+1} \Big\}+o(1),\label{66}
	\end{align}
where we used  	\eqref{59} and \eqref{63'}  in the  last  equality.
	
		We dirive from	\eqref{22},	\eqref{56} and  \eqref{trace} that
	\begin{align}
	&\frac{1}{2} \| U_{1}\|^{2}-\frac{N_{\si}}{2^{*}_{\si}} \int_{\Rn} K_{l}(\cdot/\lam _{1}^{l}+z_{1}^{l})H^{\tau_l}(\cdot / \lam _{1}^{l}+z_{1}^{l})U_{1}^{p_{l}+1}\notag\\
	=&I_{K_{\infty}^{(1)}(\xi_{1})}(U_{1})+o(1) \notag\\
	\geq& \frac{\si N_{\si}}{n} K_{\infty}^{(1)}(\xi_{1})^{(2\si-n)/2\si }(S_{n,\si})^{n/\si}+o(1) \notag\\
	\geq& c^{(1)}+o(1).\label{67}
	\end{align}
	Claim	\ref{claim:4} follows from \eqref{34}, \eqref{66}-\eqref{67}, and the fact $(\lam_{1}^{l})^{4 \si /(p_{l}-1)+2 \si-n} \geq 1$.
\end{proof}

From \eqref{36}, \eqref{59}  and \eqref{63}, we have
\be\label{69}
\xi_{l}=U_{l}-\mathscr{T}_{l, \lam _{1}^{l}, z_{1}^{l}}^{-1} U_{1}=\al_{2}^{l}  \widetilde{\delta}( z_{2}^{l}, \lam _{2}^{l} )+w_{l},
\ee
where
$$
w_{l}=\al_{1}^{l}  \widetilde{\delta}( z_{1}^{l}, \lam _{1}^{l} )-K_{\infty}^{(1)}(\xi_{1})^{(2\si-n)/4\si}(\lam _{1}^{l})^{2 \si/(p_{l}-1)-(n-2\si) / 2}   \widetilde{\delta}( z^{*}/\lam _{1}^{l}+z_{1}^{l}, \lam ^{*} \lam _{1}^{l} )+v_{l}.
$$
Now using  \eqref{60} and Claim \ref{claim:2}, we have,  for large $l$, that
\be
\label{70}\|w_{l}\| = o_{\var_{2}}(1).
\ee

We can simply repeat the previous arguments on $\xi_{l}$ instead of on $U_{l}$. For the reader's convenience, we carry out some crucial steps.

A direct computation using  \eqref{69} shows that 
\be\label{72}
\mathscr{T}_{l, \lam _{2}^{l}, z_{2}^{l}} \xi_{l}=\over{\al}_{2}^{l}  \widetilde{\delta}(0,1)+\mathscr{T}_{l, \lam _{2}^{l}, z_{2}^{l}} w_{l},
\ee
where
\be\label{73}
\over{\al}_{2}^{l}=\al_{2}^{l}(\lam _{2}^{l})^{(n-2\si) / 2-2 \si/(p_{l}-1)}.
\ee

Then we can verify the existence  of $U_{2} \in D$ and $\xi_{2} \in O^{(2)}$ such that\begin{gather}
\label{74}\mathscr{T}_{l, \lam _{2}^{l}, z_{2}^{l}} \xi_{l} \rightharpoonup U_{2}\quad 
\text { weakly in }\, D,
\\\label{75}
\lim _{l \to  \infty}(z_{2}^{l}-z_{l}^{(2)})=\xi_{2},
\end{gather}
up to a subsequence.

Accordingly, by making use of \eqref{22}, \eqref{25} and \eqref{75}, we have
\be\label{76}
\lim _{l \to  \infty} K_{l}(z_{2}^{l})^{(2\si-n)/4\si}=K_{\infty}^{(2)}(\xi_{2})^{(2\si-n)/4\si}.
\ee

For any  $\phi \in C_{c}^{\infty}(\over{\R}^{n+1}_+)$, it follows from Claim \ref{claim:3} and Lemma \ref{lem:3.3} that\begin{align*}
o(1)=&I_{K_{l}, \tau_{l}}'(\xi_{l})(\mathscr{T}_{l, \lam _{2}^{l}, z_{2}^{l}}^{-1} \phi)\\=&(\lam _{2}^{l})^{4 \si /(p_{l}-1)+2 \si-n}\Big\{\big\langle \mathscr{T}_{l, \lam _{2}^{l}, z_{2}^{l}} \xi_{l} , \phi\big\rangle -N_{\si}\int_{\Rn}  K_{l}(\cdot/\lam _{2}^{l}+z_{2}^{l})\\&\times H^{\tau_l}(\cdot / \lam _{2}^{l}+z_{2}^{l})|\mathscr{T}_{l, \lam _{2}^{l}, z_{2}^{l}} \xi_{l}|^{p_l-1}(\mathscr{T}_{l, \lam _{2}^{l}, z_{2}^{l}} \xi_{l}) \phi\Big\}.
\end{align*}
Taking the limit $l \to  \infty$ and arguing as before, we have
$$
\big\langle U_{2},\phi\big\rangle -N_{\si}\int_{\Rn} K_{\infty}^{(2)}(\xi_{2})|U_{2}|^{\frac{4\si}{n-2\si}} U_{2} \phi=0.
$$
Namely, $U_2$ satisfies
\be\label{77}
\left\{\begin{aligned}
&\operatorname{div}(t^{1-2\si}\nabla U_2)=0&&\text{ in }\,\Rp,\\&\pa_{\nu}^{\si}U_2=K_{\infty}^{(2)}(\xi_{2})|U_{2}(x,0)|^{\frac{4\si}{n-2\si}} U_{2}(x,0) && \text { on }\, \Rn.
\end{aligned}\right.
\ee
Arguing as before, one can prove  that, for $\var_{2}$ small enough,  $U_{2}>0$ and for some $z^{**} \in \Rn$ and  $\lam ^{* *}>0$, there holds
\be\label{78}
U_{2}=K_{\infty}^{(2)}(\xi_{2})^{(2\si-n)/4\si} \widetilde{\delta}(z^{**},\lam ^{* *}).
\ee
\begin{claim}\label{claim:5}
	For $l$ large enough, we have 
$$
	|z^{**}| = o_{\var_{2}}(1), \quad |\lam ^{* *}-1| = o_{\var_{2}}(1),\quad (\lam _{2}^{l})^{\tau_{l}}=1+o_{\var_{2}}(1).$$
\end{claim}
\begin{proof} The proof is same as  the proof of  Claim \ref{claim:2}, we omit it here. 
\end{proof}	
We define $\eta_{l} \in D$ by
\be\label{82}
\mathscr{T}_{l, \lam _{2}^{l}, z_{2}^{l}} \xi_{l}=U_{2}+\mathscr{T}_{l, \lam _{2}^{l}, z_{2}^{l}} \eta_{l}.
\ee
Clearly,\be\label{821}
\mathscr{T}_{l, \lam _{2}^{l}, z_{2}^{l}} \eta_{l} \rightharpoonup 0  \quad \text { weakly in }\, D.
\ee
\begin{claim}\label{claim:6}
For	$\var_{2}$  small enough, we have  $\|I_{K_{l},\tau_{l}}'(\eta_{l})\|=o(1)$.
\end{claim}

\begin{proof}The proof 	is same as  the proof of  Claim \ref{claim:3}, we omit it here.
\end{proof}

\begin{claim}\label{claim:7}
For $\var_{2}$  small enough, we have 	$I_{K_{l},\tau_{l}}(\eta_{l}) \leq \var_{3}+o(1)$.
\end{claim}
\begin{proof}The proof 	is same as  the proof of  Claim \ref{claim:4}, we omit it here.
\end{proof}

\begin{claim}\label{claim:8}
	For $\var_{2}>0$   small enough, we have $\eta_{l} \to  0$ strongly in $D$.
\end{claim}
\begin{proof}%
	It follows from Claim \ref{claim:6} and Claim \ref{claim:7} that
\be\label{89}
	\| \eta_{l}\|^{2} \leq \frac{n}{\si} \var_{3}+o(1).
\ee
	Suppose that Claim \ref{claim:8} does not hold, then along a subsequence we have
\be\label{90}
	\|\eta_{l}\|^{\tau_{l}}=1+o(1).
\ee
	We derive from \eqref{90},   H\"{o}lder inequality, and Claim \ref{claim:6} that$$
\|\eta_{l}\|^{2} \leq C(n,\si, A_{1})\Big(\int_{\Rn}|\eta_{l}(x,0)|^{2_{\si}^{*}}\,\d x\Big)^{(p_{l}+1)/2_{\si}^{*}}+o(1),
	$$and then\be\label{91}
\| \eta_{l}\|^{2} \leq C(n,\si, A_{1}) \int_{\Rn}|\eta_{l}(x,0)|^{2_{\si}^{*}}\,\d x+o(1).
\ee
		It follows from \eqref{91} and  \eqref{trace} that
$$
	\mathcal{S}_{n,\si} \leq \frac{\big(\int_{\Rp}t^{1-2\si}|\nabla \eta_{l}|^{2}\,\d X\big)^{1 / 2}}{\big(\int_{\Rn}|\eta_{l}(x,0)|^{2_{\si}^{*}}\,\d x\big)^{1/2_{\si}^{*}}}\leq \frac{\big(\int_{\Rp}t^{1-2\si}|\nabla\eta_{l}|^{2}\,\d X\big)^{1 / 2}C(n,\si, A_{1})^{1/2_{\si}^{*}}}{\big(\int_{\Rp} t^{1-2\si}|\nabla \eta_{l}|^{2}\,\d X+o(1)\big)^{1/2^{*}_{\si}}}.
$$
		Thus,
		\be\label{92}
	\mathcal{S}_{n,\si} \leq C(n,\si, A_{1})^{1/2^{*}_{\si}}\| \eta_{l}\|^{2\si/ n}.
\ee
However,  \eqref{89} and  \eqref{92} cannot hold  at the same time if $\var_{2}>0$ is small enough. Claim \ref{claim:8} has been established.
\end{proof}
Rewriting \eqref{63} and \eqref{82}, we have
\be\label{93}
U_{l}=\mathscr{T}_{l, \lam _{2}^{l}, z_{2}^{l}}^{-1} U_{1}+\mathscr{T}_{l, \lam _{2}^{l}, z_{2}^{l}}^{-1} U_{2} +\eta_{l}.
\ee

\begin{claim}\label{claim:9}
	For $\var_{2}>0$ small enough, we have 
$$	(\lam _{1}^{l})^{\tau_{l}}=1+o_{\var_{3}}(1)+o(1),\quad  (\lam _{2}^{l})^{\tau_{l}}=1+o_{\var_{3}}(1)+o(1).
	$$
\end{claim}
\begin{proof}
 We deduce from \eqref{66}-\eqref{67} and Lemma \ref{lem:3.3} that
\be\label{94}
	I_{K_{l}, \tau_l}(U_{l}) \geq I_{K_{l}, \tau_{l}}(\xi_{l})+(\lam _{1}^{l})^{4 \si /(p_{l}-1)+2 \si-n} c^{(1)}+o(1).
\ee In view of   Claim \ref{claim:5}, \eqref{78}, \eqref{82} and \eqref{821}, some calculations  similar to the proof of Claim \ref{claim:4} led to
\be \label{95}
	I_{K_{l}, \tau_{l}}(\xi_{l}) \geq I_{K_{l}, \tau_{l}}(\eta_{l})+(\lam _{2}^{l})^{4 \si /(p_{l}-1)+2 \si-n} c^{(2)}+o(1).
\ee	Then we use  Claim \ref{claim:8} to deduce that
\be\label{96}
	I_{K_{l}, \tau_l}(\eta_{l})=o(1).
\ee
	Finally, we put together \eqref{34}, \eqref{94}-\eqref{96} to obtain$$
	\sum_{i=1}^{2}\big\{(\lam _{i}^{l})^{4 \si /(p_{l}-1)+2 \si-n}-1\big\} c^{(i)} \leq \var_{3}+o(1).
	$$ 
	This completes the proof of Claim \ref{claim:9}.
\end{proof}

\begin{claim}\label{claim:10}
	Let $\delta_{5}=\delta_{1} /(2 A_{3})$. Then  if $\var_{2}>0$ is chosen to be small enough,   we have, for large $l$, that	$$
	\operatorname{dist}(z_{1}^{l}, \pa  O_{l}^{(1)})\geq \delta_{5}, \quad 
\operatorname{dist}(z_{2}^{l}, \pa  O_{l}^{(2)})\geq \delta_{5}.
$$
\end{claim}
\begin{proof}By using  contradiction argument, we can follow the same proof  in Lemma \ref{lem:3.1} to reach  a contradiction. Hence we omit it here.
		\end{proof}
We are now in the position  to prove  Proposition \ref{prop:3.2}.
\begin{proof}[Proof of Proposition \ref{prop:3.2}]
Applying  \eqref{22}, \eqref{25},   \eqref{56}, \eqref{59}, and Claim \ref{claim:9}, we deduce that
	\begin{align*}
	\mathscr{T}_{l, \lam _{2}^{l}, z_{2}^{l}}^{-1}U_{1} &=(\lam _{1}^{l})^{2\si /(p_{l}-1)} U_{1}(\lam _{1}^{l}(x-z_{1}^{l}),\lam _{1}^{l}) \\
	&=(\lam _{1}^{l})^{2 \si/(p_{l}-1)-(n-2\si)/2\si} K_{\infty}^{(1)}(\xi_{1})^{(2\si-n)/4\si}  \widetilde{\delta}(z_{1}^{l}+z^{*}/\lam _{1}^{l},\lam ^{*} \lam _{1}^{l} )\\&=K_{\infty}^{(1)}(\xi_{1})^{(2\si-n)/4\si}  \widetilde{\delta}(z_{1}^{l}+z^{*}/\lam _{1}^{l},\lam ^{*} \lam _{1}^{l} )+o_{\var_{3}}(1)\\&=K_{l}(z_{1}^{l}+z^{*}/\lam _{1}^{l})^{(2\si-n)/4\si}  \widetilde{\delta}(z_{1}^{l}+z^{*}/\lam _{1}^{l},\lam ^{*} \lam _{1}^{l} )+o_{\var_{3}}(1)+o(1).
	\end{align*}
		Similarly, we have
	$$
	\mathscr{T}_{l, \lam _{2}^{l}, z_{2}^{l}}^{-1} U_{2}=K_{l}(z_{2}^{l}+z^{**}/\lam _{2}^{l})^{(2\si-n)/4\si}  \widetilde{\delta}(z_{2}^{l}+z^{**}/\lam _{2}^{l}, \lam ^{**}  \lam _{2}^{l} )+o_{\var_{3}}(1)+o(1).
	$$
Therefore, we can rewrite \eqref{93} as (see Claim \ref{claim:8} and the above)
	\begin{align}
	U_{l}= &K_{l}(z_{1}^{l}+z^{*}/\lam _{1}^{l})^{(2\si-n)/4\si} \widetilde{\delta}(z_{1}^{l}+z^{*}/\lam _{1}^{l},\lam ^{*}\lam _{1}^{l} )\notag\\& +K_{l}(z_{2}^{l}+z^{**}/\lam _{2}^{l})^{(2\si-n)/4\si}  \widetilde{\delta}( z_{2}^{l}+z^{**}/\lam _{2}^{l},\lam ^{**}\lam _{2}^{l} )+o_{\var_{3}}(1)+o(1) .\label{101}
	\end{align}
		We now fix the value of $\var_{2}$ to be small to make all the previous
	arguments hold and then make $\var_{3}>0$ small (depending on $\var_{2}>0$) to make the following hold (using Claim \ref{claim:9}):
\be\label{102}
	|(\lam ^{*} \lam _{1}^{l})^{\tau_{l}}-1| \leq o_{\var_{3}}(1)+o(1)<\var_{2}/2 ,\quad  |(\lam ^{* *} \lam _{2}^{l})^{\tau_{l}}-1| \leq o_{\var_{3}}(1)+o(1)<\var_{2}/2.
\ee
		From \eqref{101}-\eqref{102}, Claim \ref{claim:1} and  Claim \ref{claim:9}, we see that for $\var_{3}>0$ small, we
	have, for large $l$, $$
	U_{l} \in \widetilde{V}_{l}(2, \var_{2} / 2),
$$
which contradicts to \eqref{33}. This concludes the proof of Proposition	\ref{prop:3.2}.	
\end{proof}
\subsection{Complete  the proof of Theorem \ref{thm:3.1}}\label{sec:5}
In this section we will complete the proof of Theorem \ref{thm:3.1}.  Precisely, under the contrary of Theorem \ref{thm:3.1} and combining  with the Proposition \ref{prop:3.2} established in Section \ref{sec:4}, we will reach a contradiction after a lengthy indirect argument.  The method we shall use is similar to that in \cite{Li93}, see also \cite{CR1,CR2,CES,Se}, but we have to set up a framework to fit the fractional situation.   To reduce overlaps, we will omit the proofs of several intermediate results which closely follow standard arguments, giving appropriate references.  Let us  start   with defining a certain family  of sets and minimax values and giving some notation. 

For any $z\in\Rn$,  denote $\Sigma_{R}^{+}(z)=\B^+_{R}(z)\cup B_{R}(z)$. We  define the space $\H_{0}^{1}(t^{1-2\si},\Sigma_{R}^{+}(z))$ as the closure in $ H^{1}(t^{1-2\si}, \B_{R}^{+}(z))$ of $C_{c}^{\infty}(\over{\B^+_{R}(z)})$ under the norm \be\label{Wballnorm}
\|U\|_{\H_{0}^{1}(t^{1-2\si},\Sigma_{R}^{+}(z))}:=\Big(\int_{\B^+_{R}(z)} t^{1-2\si}|\nabla U|^{2}+|U|^{2}) \,\d X\Big)^{1 / 2}.
\ee   It follows from the Hardy–Sobolev inequality in \cite[Lemma 2.4]{FV} that    $\H_{0}^{1}(t^{1-2\si},\Sigma_{R}^{+}(z))$ can be
endowed with the equivalent norms
$$
\|U\|_{\H_{0}^{1}(t^{1-2\si},\Sigma_{R}^{+}(z))}:=\Big(\int_{\B_{R}^{+}(z)} t^{1-2\si}|\nabla U|^{2} \,\d X\Big)^{1/2}.
$$

In this section,  we write  $\tau=\tau_{l}$, $p=p_l$, and $\H_{0}^{1}(t^{1-2\si},\Sigma_{R}^{+}(z))=\H_{0}^{1}(\Sigma_{R}^{+}(z))$.

Now, we define
\begin{align*}
&\gamma_{l, \tau}^{(1)}=\big\{g^{(1)} \in C([0,1],\H_{0}^{1}(\Sigma_{R_l}^{+}(z_{l}^{(1)}))) : g^{(1)}(0)=0,\, I_{K_{l}, \tau}(g^{(1)}(1))<0\big\},\\&\gamma_{l, \tau}^{(2)}=\big\{g^{(2)} \in C([0,1], \H_0^{1}(\Sigma_{R_l}^{+}(z_{l}^{(2)}))) : g^{(2)}(0)=0,\, I_{K_{l}, \tau}(g^{(2)}(1))<0\big\},\\&c_{l, \tau}^{(1)}=\inf _{g^{(1)}\in \gamma_{l, \tau}^{(1)}} \max _{0 \leq \theta_{1} \leq 1} I_{K_{l },\tau}(g^{(1)}(\theta_{1})),\\&c_{l, \tau}^{(2)}=\inf _{g^{(2)}\in \gamma_{l, \tau}^{(2)}} \max _{0 \leq \theta_{2} \leq 1} I_{K_{l},\tau}(g^{(2)}(\theta_{2})).
\end{align*}
Here we have abused the notation a little by writing $I_{K_{l},\tau}$ as 
$I_{K_{l},\tau}:\H_0^{1}(\Sigma_{R_l}^{+}(z_{l}^{(1)})) \to  \R$ and also $I_{K_{l},\tau}:\H_0^{1}(\Sigma_{R_l}^{+}(z_{l}^{(2)})) \to  \R$.

\begin{prop}\label{prop:4.1}
	Let $\{K_{l}\}$ be a sequence of  functions
	satisfying \eqref{18}, \eqref{20} and \eqref{21}. Then  there holds
\be\label{103}
c_{l, \tau}^{(1)}=c^{(1)}+o(1),\quad c_{l, \tau}^{(2)}=c^{(2)}+o(1),
\ee
	where $o(1)\to  0$ as $l\to  \infty$.
\end{prop}
\begin{proof}
 We will only prove the first term in \eqref{103}, because the other one can be justified in a similar manner. From the definition of $c_{l, \tau}^{(1)}$, we  deduce  that 
\be\label{105}
C^{-1}( n,\si,A_{1})\leq c_{l, \tau}^{(1)} \leq C( n,\si,A_{1})
\ee
 for some constant $C( n,\si,A_{1})>0$.	Moreover, for any $U \in \H^{1}_{0}(\Sigma_{R_l}^{+}(z_{l}^{(1)})) \setminus \{0\}$, one has
\begin{align*}
 c_{l, \tau}^{(1)} &\leq \max _{0 \leq s<\infty} I_{K_{l}, \tau}(s U)\\&=\Big(\frac{1}{2}-\frac{1}{p+1}\Big) \frac{\big(\int_{\B^+_{R_{l}}(z_{l}^{(1)})}t^{1-2\si}|\nabla U|^{2}\,\d X\big)^{(p+1) /(p-1)}}{\big(N_{\si}\int_{B_{R_{l}}(z_{l}^{(1)})} K_{l}(x)H^{\tau}(x)|U(x,0)|^{p+1}\,\d x\big)^{2 /(p-1)}}.
 \end{align*}
Let $U=\eta(x+z_{l}^{(1)},t ) \widetilde{\delta}(z_{l}^{(1)}, \lam _{l})$, where  $\eta$ is a nonnegative smooth cut-off function supported in $\B_{1}$ and	equal to 1 in $\B_{1/2}$, and $\{\lam _{l}\}$ is a sequence satisfying
\be
	\label{106}
	\lim _{l \to  \infty} \lam _{l}=\infty,\quad  (\lam _{l})^{\tau}=1+o(1).
\ee Then we obtain
$$
		c_{l, \tau}^{(1)} \leq \frac{\si N_{\si}}{n}(a^{(1)})^{(2\si-n)/2\si}(S_{n,\si})^{n/\si}+o(1)=c^{(1)}+o(1).
$$
	The other side of the inequality can be proved as the following.

	For any $l, \tau$ fixed, it is well-known  that there exists $\{U_{k}\} \subset \H_{0}^{1}(\Sigma_{R_l}^{+}(z_{l}^{(1)}))$ such that
	\begin{gather*}
	\lim _{k \to  \infty} I_{K_{l}, \tau}(U_{k})=c_{l, \tau}^{(1)},\\\label{107'}
I_{K_{l}, \tau}'(U_{k}) \to  0 \quad \text{ in }\,  H^{-\si}(\Sigma_{R_l}^{+}(z_{l}^{(1)}))\quad  \text{ as }\, k \to  \infty,
	\end{gather*}
	where $H^{-\si}(\Sigma_{R_l}^{+}(z_{l}^{(1)}))$ denotes  the dual space of $\H_{0}^{1}(\Sigma_{R_l}^{+}(z_{l}^{(1)}))$.	 Namely, we have
	\begin{gather}
\frac{1}{2} \int_{\B^+_{R_{l}}(z_{l}^{(1)})}t^{1-2\si}|\nabla U_{k}|^{2}\,\d X-\frac{N_{\si}}{p+1} \int_{B_{R_{l}}(z_{l}^{(1)})} K_{l}(x)H^{\tau}(x)|U_{k}(x,0)|^{p+1}\,\d x=c_{l, \tau}^{(1)}+o_{k}(1),\label{107}\\	\int_{\B^+_{R_{l}}(z_{l}^{(1)})}t^{1-2\si}|\nabla U_{k}|^{2}\,\d X=N_{\si}\int_{B_{R_{l}}(z_{l}^{(1)})} K_{l}(x)H^{\tau}(x)|U_{k}(x,0)|^{p+1}\,\d x+o_{k}(1),\label{107''}
	\end{gather}
	where $o_{k}(1)\to  0$ as $k\to  \infty$.
		It follows that
\be\label{108}
	c_{l, \tau}^{(1)}=\frac{\si }{n} 	\int_{\B^+_{R_{l}}(z_{l}^{(1)})}t^{1-2\si}|\nabla U_{k}|^{2}\,\d X+o_{k}(1)+o(1).
\ee
	By   \eqref{26}, \eqref{105}, \eqref{107} and \eqref{107''},   we compute similarly to the proof of \eqref{claimestimate} to obtain
\be\label{109}		\liminf_{k \to  \infty} \int_{\B^+_{R_{l}}(z_{l}^{(1)})}t^{1-2\si}|\nabla U_{k}|^{2}\,\d X \geq K_{l}(z_{l}^{(1)})^{(2\si-n)/2\si}N_{\si}(S_{n,\si})^{n/\si}+o(1).
\ee
	We can deduce from \eqref{21}, \eqref{108} and  \eqref{109} that $$c_{l, \tau}^{(1)} \geq\frac{\si N_{\si}}{n}(a^{(1)})^{(2\si-n)/2\si}(S_{n,\si})^{n/\si}+o(1)=c^{(1)}+o(1).$$
	This gives the proof of \eqref{103}. 
\end{proof} 

We define
\begin{gather}
\Gamma_{l}=\big\{G=\over{g}^{(1)}+\over{g}^{(2)} : \over{g}^{(1)}, \over{g}^{(2)} \text { satisfy }\eqref{110}-\eqref{114}\big\},\\\over{g}^{(1)}, ~\over{g}^{(2)} \in C([0,1]^{2}, D),\label{110}\\\over{g}^{(1)}(0, \theta_{2})=\over{g}^{(2)}(\theta_{1}, 0)=0, \quad 0 \leq \theta_{1}, \theta_{2} \leq 1,\label{111}\\I_{K_{l}, \tau}(\over{g}^{(1)}(1, \theta_{2}))<0,~ I_{K_{l}, \tau}(\over{g}^{(2)}(\theta_{1}, 1))<0, \quad 0 \leq \theta_{1}, \theta_{2} \leq 1,\label{112}\\\operatorname{supp} \over{g}^{(1)} \subset \B^+_{R_{l}}(z_{l}^{(1)}), \quad \theta=(\theta_1,\theta_2) \in[0,1]^{2},\label{113}\\\operatorname{supp} \over{g}^{(2)} \subset \B^+_{R_{l}}(z_{l}^{(2)}),  \quad\theta=(\theta_1,\theta_2) \in[0,1]^{2},\label{114}\\b_{l, \tau}=\inf _{G \in \Gamma_{l}} \max _{\theta \in[0,1]^{2}} I_{K_{l}, \tau}(G(\theta)).
\end{gather}
\begin{rem}
	Observe that if $G=g^{(1)}+g^{(2)}$ with $g^{(1)}\in \gamma_{l, \tau}^{(1)}$, $g^{(2)}\in \gamma_{l, \tau}^{(2)}$, $\operatorname{supp}g^{(1)}\cap \operatorname{supp}g^{(2)}=\emptyset$,
	then $I_{K_{l}, \tau}(G)=I_{K_{l}, \tau}(g^{(1)})+I_{K_{l}, \tau}
	(g^{(2)})$.
\end{rem}
\begin{prop}\label{prop:4.2}
Suppose that $\{K_{l}\}$ is a sequence of  functions satisfying \eqref{18}, \eqref{20} and \eqref{21}, then there holds	$b_{l, \tau}=c_{l, \tau}^{(1)}+c_{l, \tau}^{(2)}+o(1)$.
\end{prop}
\begin{proof}The first inequality $b_{l, \tau} \geq c_{l, \tau}^{(1)}+c_{l, \tau}^{(2)}$ can be achieved from  the definition  of  $c_{l, \tau}^{(1)}$ and $c_{l, \tau}^{(2)}$ with additional compactness argument on  $[0,1]^2$, we omit it here and refer to \cite[Proposition 3.4]{CR1} for details.

	On the other hand,	for $0 \leq \theta_{1}, \theta_{2} \leq 1$, let
 	\begin{align*}
	g_{l}^{(1)}(\theta_{1})&=\theta_{1} C_{1} K_{l}(z_{l}^{(1)})^{(2\si-n)/4\si} \eta(x+z_{l}^{(1)},t)  \widetilde{\delta}(z_{l}^{(1)},\lam _l),\\	g_{l}^{(2)}(\theta_{2})&=\theta_{2} C_{1} K_{l}(z_{l}^{(2)})^{(2\si-n)/4\si} \eta(x+z_{l}^{(2)},t)  \widetilde{\delta}(z_{l}^{(2)},\lam _l), 
	\end{align*}
where $\{\lam _{l}\}$ is defined in \eqref{106} and  $C_{1}=C_{1}(n,\si,A_{1}, A_{2})>1$ is a constant such that 
$$
I_{K_l,\tau}(g_{l}^{(1)}(\theta_{1}))<0\quad \text{ and }\quad I_{K_l,\tau}(g_{l}^{(2)}(\theta_{2}))<0
$$
for large $l$. We fix the value of $C_{1}$ from now on.	
	
	For $\theta=(\theta_{1}, \theta_{2}) \in[0,1]^{2}$, let
	$G_{l}(\theta)=g_{l}^{(1)}(\theta_{1})+g_{l}^{(2)}(\theta_{2})$.
Observe that $\operatorname{supp}g_{l}^{(1)}(\theta_{1})\cap \operatorname{supp} g_{l}^{(2)}(\theta_{2})=\emptyset$, a direct calculation shows that
	\begin{align*}
	\max _{\theta \in[0,1]^{2}} I_{K_{l}, \tau}(G_{l}(\theta)) 
	= &\max _{\theta_{1} \in[0,1]} I_{K_{l}, \tau}(g_{l}^{(1)}(\theta_{1}))+\max _{\theta_{2} \in[0,1]} I_{K_{l}, \tau}(g_{l}^{(2)}(\theta_{2}))+o(1)\\\leq &\max _{0 \leq s<\infty} I_{K_{l}, \tau}(s \eta(x+z_{l}^{(1)},t)  \widetilde{\delta}(z_{l}^{(1)}, \lam _{l})) \\& +\max _{0 \leq s<\infty} I_{K_{l}, \tau}(s \eta(x+z_{l}^{(2)},t)  \widetilde{\delta}(z_{l}^{(2)}, \lam _l))+o(1)
	\\=&c_{l, \tau}^{(1)}+c_{l, \tau}^{(2)}+o(1),
	\end{align*}
where   the last equality is due to Proposition \ref{prop:4.1}. 
Therefore, 	$
b_{l, \tau} \leq c_{l, \tau}^{(1)}+c_{l, \tau}^{(2)}+o(1)
$. This ends the proof. 
\end{proof} 

In the following, under the contrary of 
Theorem \ref{thm:3.1}, we can construct $H_{l} \in \Gamma_{l}$ for large $l$, such that $$
\max _{\theta \in[0,1]^{2}} I_{K_{l}, \tau}(H_{l}(\theta))<b_{l, \tau},
$$ which contradicts to the definition of $b_{l, \tau}$. A lengthy construction is required to establish this fact and  a brief sketch of it  will be given now by the details. 

\emph{Step 1:}  Choosing some suitably small number $\var_{4}>0$, we can construct
$G_{l} \in \Gamma_{l}$ such that $$
\max _{\theta \in[0,1]^{2}} I_{K_{l}, \tau}(G_{l}(\theta)) \leq b_{l, \tau}+\var_{4},$$ and satisfies some further properties.

\emph{Step 2:}  By  the negative gradient flow of $I_{K_{l}, \tau}$,  $G_{l}$ is deformed  to $U_{l}$ such that $$\max _{\theta\in[0,1]^{2}} I_{K_{l}, \tau}(U_{l}(\theta)) \leq b_{l, \tau}-\var_{4}.$$	If $U_{l}\in \Gamma_{l}$, we will reach a contradiction to the definition of $b_{l, \tau}$. However, $U_{l}$ is not necessarily in $\Gamma_{l}$ any more since the deformation may not preserve properties \eqref{113}-\eqref{114}.

\emph{Step 3:}  Applying  Propositions  \ref{prop:1.4}, \ref{prop:2.1} and \ref{prop:3.2}, we modify $U_{l}$ to
obtain $H_{l} \in \Gamma_{l}$  such that $$
\max _{\theta \in[0,1]^{2}} I_{K_{l}, \tau}(H_{l}(\theta)) \leq b_{l, \tau}-\var_{4}/2.
$$

All the three steps are completed for large $l$ only. Now we give the details to establish these  steps.

\emph{Step 1: Construction of $G_l$.}  Let $G_{l}$ be the one we have just defined. We establish some properties of it which are needed.
\begin{lem}\label{lem:4.1}
	For any $\var\in (0,1)$, if  $I_{K_{l}, \tau}(g_{l}^{(i)}(\theta_{i})) \geq c_{l, \tau}^{(i)}-\var$ for  $i = 1,2$, then there exists some constants $\Lam _{1}=\Lam _{1}(n,\si,\var ,A_{1}, A_{3})>1$ and $C_{0}(n,\si)>0$ such	that for any $l \geq \Lam _{1}$ and  $0 \leq \theta_{1}, \theta_{2} \leq 1$,   we have 	$|C_{1} \theta_{i}-1| \leq C_{0}(n,\si) \sqrt{\var},\,i=1,2.$
\end{lem}
\begin{proof}
We only take into account the case  $i=1$ since the other case can be covered in the same way.
	Let $s_1=C_{1} \theta_{1}$, a direct calculation shows that
	\begin{align}
		I_{K_{l}, \tau}(g_{l}^{(1)}(\theta_{1}))=&\frac{1}{2} s_1^{2} K_{l}(z_{l}^{(1)})^{(2\si-n)/2\si} \|\eta(x+z_{l}^{(1)},t)  \widetilde{\delta}(z_{l}^{(1)}, \lam _{l})\|^{2} \notag\\&-\frac{N_{\si}}{p+1} s_1^{p+1} K_{l}(z_{l}^{(1)})^{(p+1)(2\si-n)/4\si}\notag\\&\times \int_{\Rn} K_{l}(x)H^{\tau}(x)|\eta(x+z_{l}^{(1)},0)  \delta(z_{l}^{(1)}, \lam _{l})|^{p+1}\, \d x\notag\\=&\Big(\frac{1}{2}+o(1)\Big) s_1^{2} K_{l}(z_{l}^{(1)})^{(2\si-n)/2\si} \|  \widetilde{\delta}(0,1)\|^{2}\notag \\& -\Big(\frac{N_{\si}}{p+1}+o(1)\Big) s_1^{p+1} K_{l}(z_{l}^{(1)})^{(2\si-n)/2\si} \|  \delta(0,1)\|_E^{2}\notag
	\\=&\Big[\Big(\frac{n}{2\si}+o(1)\Big) s_1^{2}-\Big(\frac{n-2\si}{2\si}+o(1)\Big) s_1^{p+1}\Big]c_{l,\tau}^{(1)},\label{lem:6.1}
	\end{align}
		where Proposition \ref{prop:4.2} is used in the last step.
Hence, using \eqref{lem:6.1} and the hypothesis $I_{K_{l}, \tau}(\over{g}_{l}^{(1)}(\theta_{1})) \geq c_{l, \tau}^{(1)}-\var$, we  complete the proof.
\end{proof}

\begin{lem}\label{lem:4.2}
	For any $\var\in (0,1)$, there exists a constant  $\Lam _{2}=\Lam _{2}(n,\si,\var, A_{1}, A_{3})>\Lam _{1}$ such
	that for any $l \geq \Lam _{2}$, $0 \leq \theta_{1}, \theta_{2} \leq 1$, we have
	$I_{K_{l}, \tau}(g_{l}^{(i)}(\theta_{i})) \leq c_{l, \tau}^{(i)}+\var/10,\,i=1,2.$
\end{lem}
\begin{proof}
The proof is similar to Proposition \ref{prop:4.2}.
\end{proof}

\begin{lem}\label{lem:4.3}
	For any $\var\in (0,1)$, there exists a constant $\Lam _{3}=\Lam _{3}(n,\si,\var, A_{1}, A_{3})>\Lam _{2}$ such
	that$$
	I_{K_{l}, \tau}(G_{l}(\theta))\big|_{\theta \in \pa [0,1]^{2}} \leq \max 
	\{c^{(1)}+\var, c^{(2)}+\var\}\quad \text{ for all }\, l\geq \Lam _{3}.
	$$
\end{lem}
\begin{proof} Lemma \ref{lem:4.3} follows immediately  from Lemma \ref{lem:4.2}.
\end{proof}

\begin{lem}\label{lem:4.4}
For any $\var\in (0,1/2)$,	 if  $
I_{K_{l}, \tau}(G_{l}(\theta)) \geq c_{l, \tau}^{(1)}+c_{l, \tau}^{(2)}-\var $, then there exists a  constant $C_{0}=C_{0}(n,\si)>1$ such	that  for any $l \geq \Lam _{3}$, $\theta \in[0,1]^{2}$,  we have $ |C_{1} \theta_{i}-1| \leq C_{0} \sqrt{\var},\,i=1,2$.
\end{lem}
\begin{proof}
	Since $g_{l}^{(1)}(\theta_{1})$ and $g_{l}^{(2)}(\theta_{2})$ have disjoint supports, after a direct calculation, we see that	$$I_{K_{l}, \tau}(G_{l}(\theta))=I_{K_{l}, \tau}(g_{l}^{(1)}(\theta_{1}))+I_{K_{l}, \tau}(g_{l}^{(2)}(\theta_{2}))+o(1).$$
Then 	it follows from Lemma \ref{lem:4.2} and the 
 hypothesis  $I_{K_{l}, \tau}(G_{l}(\theta)) \geq c_{l, \tau}^{(1)}+c_{l, \tau}^{(2)}-\var$ that 
\be\label{lem:4.4.1}
 	I_{K_{l}, \tau}(g_{l}^{(1)}(\theta_{1})) \geq c_{l, \tau}^{(1)}-2 \var\quad\text{ and }\quad  I_{K_{l}, \tau}(g_{l}^{(2)}(\theta_{2})) \geq c_{l, \tau}^{(2)}-2 \var.
\ee
	Lemma \ref{lem:4.4} follows immediately from  \eqref{lem:4.4.1} and Lemma \ref{lem:4.1}.
\end{proof}

\emph{Step 2: The deformation  of $G_l$.} Let
$$
M_{l} =\sup \big\{\|I_{K_{l}, \tau}'(U)\| : U \in V_{l}(2, \var_1 )\big\},\quad \beta_{l} =\operatorname{dist}\big\{\pa  \widetilde{V}_{l}(2, \var_{2}), \pa  \widetilde{V}_{l}(2,\var_{2}/2)\big\}.
$$
One  can see from the definition of $M_{l}$ that there exists some constant $C_{2}(n,\si,\var_{2},A_{1})>1$ such that $
M_{l} \leq C_{2}(n,\si,\var_{2},A_{1})$.
It is also clear from the definition of $\widetilde{V}_{l}(2, \var_{2})$ that $
\beta_{l} \geq \var_{2}/4$.
By Lemma \ref{lem:4.4}, we choose $\var_{4}$ to satisfy, for $l$ large, that
$$
 \var_{4}<\min\Big\{\var_{3},\frac{1}{2 A_{4}} ,\frac{\var_{2} \delta_{4}(\var_{2}, \var_{3})^{2}}{8 C_{2}(n,\si,\var_{2},A_{1})}\Big\},$$
\be\label{123}\begin{split}
	I_{K_{l}, \tau}(G_{l}(\theta)) \geq& c_{l, \tau}^{(1)}+c_{l, \tau}^{(2)}-\var_{4}\,\text{ implies that}\\&  G_{l}(\theta) \in \widetilde{V}_{l}(2, \var_{2} / 2),\, z_{1}(G_{l}(\theta)) \in O_{l}^{(1)},\, z_{2}(G_{l}(\theta)) \in O_{l}^{(2)}.
\end{split}
\ee
$G_{l}(\theta)$ has been defined by now.

We know from Lemma \ref{lem:4.2} that for $l$ large enough, 
$$
\max _{\theta \in[0,1]^{2}} I_{K_{l},\tau}(G_{l}(\theta)) \leq c_{l, \tau}^{(1)}+c_{l, \tau}^{(2)}+\var_{4}.
$$

For any $U_{0} \in \widetilde{V}_{l}(2, \var_{2} / 2)$, we consider the negative gradient flow of $I_{K_{l}, \tau}$,
\be\label{124}
\left\{\begin{aligned}
&\frac{\d }{\d s} \xi(s, U_{0}) =-I_{K_{l}, \tau}'(\xi(s, U_{0})), \quad  s \geq 0, \\
&\xi(0, U_{0}) =U_{0}.
\end{aligned}\right.
\ee
Under the contrary of Theorem \ref{thm:3.1}, we know $I_{K_{l}, \tau}$ satisfies the Palais-Smale condition. Furthermore,  the flow defined above never stops before exiting $V_{l}(2, \var^{*})$.

We define $U_{l} \in C([0,1]^{2}, D)$ by the following.  
\begin{itemize}
	\item  If $I_{K_{l}, \tau}(G_{l}(\theta)) \leq c_{l, \tau}^{(1)}+c_{l, \tau}^{(2)}-\var_{4}$, we define $s_{l}^{*}(\theta)=0$.	
	\item If $I_{K_{l}, \tau}(G_{l}(\theta)) > c_{l, \tau}^{(1)}+c_{l, \tau}^{(2)}-\var_{4}$, then according
	to \eqref{123}, $G_{l}(\theta) \in \widetilde{V}_{l}(2, \var_{2} / 2)$, $z_{1}(G_{l}(\theta)) \in O_{l}^{(1)}$ and $z_{2}(G_{l}(\theta)) \in O_{l}^{(2)}$. We
	define $$s_{l}^{*}(\theta)=\min\big\{s>0 : I_{K_{l}, \tau}(\xi(s, G_{l}(\theta)))=c_{l, \tau}^{(1)}+c_{l, \tau}^{(2)}-\var_{4}\big\}.$$ 	
\end{itemize}
Now we set  $$U_{l}(\theta)=\xi(s_{l}^{*}(\theta), G_{l}(\theta)).$$

The existence and continuity 
of $s_{l}^{*}(\theta)$ is guaranteed by the following lemma. 
\begin{lem}\label{lem:4.5}
	For any $U_{0} \in \widetilde{V}_{l}(2, \var_{2} / 2)$ with $z_{1}(U_{0}) \in O_{l}^{(1)}$, $z_{2}(U_{0}) \in O_{l}^{(2)}$, and
	$I_{K_{l}, \tau}(U_{0})\in (c_{l, \tau}^{(1)}+c_{l, \tau}^{(2)}-\var_{4},c_{l, \tau}^{(1)}+c_{l, \tau}^{(2)}+\var_{4}]$,  the flow line $\xi(s, U_{0})$ $(s \geq 0)$
	cannot leave $\widetilde{V}_{l}(2, \var_{2})$ before reaching $I_{K_{l}, \tau}^{-1}(c_{l, \tau}^{(1)}+c_{l, \tau}^{(2)}-\var_{4})$.	
\end{lem}
\begin{proof}
	The proof can be done exactly in the same way as in \cite[Lemma 5]{BC1},  so we omit it.
\end{proof}

 Lemma \ref{lem:4.5} is a local version of Lemma 5 in \cite{BC1} since we only need the compactness property of the flow line in certain region and Proposition \ref{prop:3.2} has provided control of $\|I_{K_{l}, \tau}'\|$ in the region.

We can see from Lemma \ref{lem:4.5} that $s_{l}^{*}(\theta)$ is well defined. Since $I_{K_{l}, \tau}$ has no critical point in $
 \widetilde{V}_{l}(2, \var_{2}) \cap\{ U \in D: | I_{K_{l}, \tau}(U)-c^{(1)}-c^{(2)}|\leq \var_{4}\} \subset V_{l}(2, \var^{*}) \cap\{U  \in D:|I_{K_{l}, \tau}(U)-c^{(1)}-c^{(2)}|\leq \var^{*}\}$
 under the contradiction hypothesis, $s_{l}^{*}(\theta)$ is continuous in $\theta$ (see also  \cite[Proposition 5.11]{Li93b} and \cite[Lemma 5]{BC1}), hence $U_{l} \in C([0,1]^{2}, D)$.

\emph{Step 3: The construction of $H_l$.} 
It follows from the construction of $U_l$ that 
$$
\max_{\theta\in [0,1]^2}I_{K_l,\tau}(U_{l}(\theta))\leq c_{l, \tau}^{(1)}+c_{l, \tau}^{(2)}-\var_{4}.
$$
Since the gradient flow does not keep properties \eqref{113}-\eqref{114}, $U_{l}(\theta)$ is not necessarily in $\Gamma_{l}$ any more. It follows from Lemma \ref{lem:4.5} that if $I_{K_{l}, \tau}(G_{l}(\theta))>c_{l, \tau}^{(1)}+c_{l,\tau}^{(2)}-\var_{4}$, then the  gradient flow $\xi(s,G_l(\theta))$ $(s\geq 0)$ cannot leave  $\widetilde{V}_{l}(2, \var_{2})$ before reaching $I_{K_{l}, \tau}^{-1}(c_{l, \tau}^{(1)}+c_{l, \tau}^{(2)}-\var_{4})$.

Using  \eqref{123} and the above information we know that if $I_{K_{l}, \tau}(G_{l}(\theta))>c_{l, \tau}^{(1)}+c_{l,\tau}^{(2)}-\var_{4}$, then $U_{l}(\theta) \in \widetilde{V}_{l}(2, \var_{2}) \subset V_{l}(2, o_{\var_{2}}(1))$ with $z_{1}(U_{l}(\theta)) \in O_{l}^{(1)}$, $z_{2}(U_{l}(\theta)) \in O_{l}^{(2)}$. This implies that
\begin{gather}
\label{125}
\int_{\om_{l}}t^{1-2\si}|\nabla U_{l}(\theta)|^{2}\,\d X+N_{\si}\int_{\pa'\om_{l}}	|U_{l}(\theta)|^{2_{\si}^{*}}\,\d x = o_{\var_{2}}(1),\\\label{126}
\|U_{l}(\theta)\|_{ W^{\si,2}(\pa'' \om_{l})} = o_{\var_{2}}(1),
\end{gather}
where
\begin{align*}
\om_{l}&=\Rp\setminus \{\B^+_{r}(z_{l}^{(1)}) \cup \B^+_{r}(z_{l}^{(2)})\},\\r&=4(\operatorname{diam} O^{(1)}+\operatorname{diam} O^{(2)}),	\\\operatorname{diam} O^{(1)}&=\sup \{|x-y| : x, y \in O^{(1)}\},\\\operatorname{diam} O^{(2)}&=\sup \{|x-y| : x, y \in O^{(2)}\},
\end{align*}
and $ O^{(1)}, O^{(2)}$ are defined in  \eqref{1111}. By Proposition \ref{prop:2.1} ($\var_{2}>0$ sufficiently small), 
 we can  modify $U_{l}(\theta)$ in $\om_{l}$ after making the following minimization.

Let
$$
\varphi_{l}(\theta)=U_{l}(\theta)|_{\pa'' \om_{l}}.
$$
Thanks to \eqref{125}-\eqref{126}, we can apply Proposition \ref{prop:2.1} to obtain the minimizer $U_{\varphi_{l}}(\theta)$ to
$$
\min \Big\{I_{K_{l}, \om_{l}}(U): U \in D_{\om_l}, \,U|_{\pa'' \om_{l}}=\varphi_{l}(\theta),\,\int_{\om_l}t^{1-2\si}|\nabla U|^2\,\d X\leq C_{1} r_{0}^{2}\Big\},
$$
where $D_{\om_{l}}$ is the closure of $ C^{\infty}_{c}(\over{\om}_{l})$ under the norm $$\|U\|_{D_{\om_{l}}}:=\Big(\int_{\om_{l}}t^{1-2\si}|\nabla U|^{2}\,\d X\Big)^{1 / 2}+\Big(\int_{\pa'' \om_{l}}|U(x,0)|^{2_{\si}^{*}}\,\d x\Big)^{1/2_{\si}^{*}},
$$and 
$C_1$ and $r_{0}$ are the constants given by Proposition \ref{prop:2.1}.

For $\theta \in[0,1]^{2}$, we define  
\begin{align*}
W_{l}(\theta)(X)=\left\{\begin{aligned}
&U_{l}(\theta)(X), & & X \in \B^+_{r}(z_{l}^{(1)}) \cup \B^+_{r}(z_{l}^{(2)}),\\
&U_{\varphi_{l}}(\theta)(X), & & X \in \Rp\setminus \{\B^+_{r}(z_{l}^{(1)}) \cup \B^+_{r}(z_{l}^{(2)})\}=\om_{l}.
\end{aligned}\right.
\end{align*}
It follows from Proposition \ref{prop:2.1} that $U_{l} \in C([0,1]^{2}, D)$ and satisfies
\begin{gather}
\max _{\theta \in[0,1]^{2}} I_{K_{l}, \tau }(W_{l}(\theta)) \leq \max _{\theta \in[0,1]^{2}} I_{K_{l}, \tau}(U_{l}(\theta)) \leq c_{l, \tau}^{(1)}+c_{l, \tau}^{(2)}-\var_{4},\label{127}\\\int_{\om_{l}}t^{1-2\si}|\nabla W_{l}(\theta)|^{2}\, \d X+\int_{\pa'\om_{l}} |W_{l}(\theta)|^{2_{\si}^{*}}\,\d x = o_{\var_{2}}(1),\\\label{128}
\left\{\begin{aligned}
&\operatorname{div}(t^{1-2\si}\nabla  W_{l}(\theta))=0&&\text{ in }\,\om_{l},\\&\pa_{\nu}^{\si}W_{l}(\theta)=K_{l}(x)H^{\tau}(x)W_{l}(\theta)^{p} &&\text { on } \, \pa'\om_{l}.
\end{aligned}\right.
\end{gather}

In order to construct the required $H_l$, we  introduce the following notation. 

First we write
\begin{align*}
\om_l^1&:=(\B_{l_{1}}^+(z_{l}^{(1)}) \setminus  \B^+_{r}(z_{l}^{(1)})) \cup(\B^+_{l_{1}}(z_{l}^{(2)})\setminus  \B^{+}_{r}(z_{l}^{(2)})),\\ \om_l^2&:=(\B_{l_{2}}^+(z_{l}^{(1)}) \setminus  \B^+_{l_{1}}(z_{l}^{(1)})) \cup(\B^+_{l_{2}}(z_{l}^{(2)})\setminus  \B^{+}_{l_{1}}(z_{l}^{(2)})),\\\om_l^3&:=(\Rp \setminus  \B^+_{l_{2}}(z_{l}^{(1)})) \cap(\Rp  \setminus  \B^+_{l_{2}}(z_{l}^{(2)})).
\end{align*}
Clearly,  $\om_l=\om_l^1\cup\om_l^2  \cup\om_l^3$ for large $l$.

For $l_{2}>100 l_{1}>1000 r$ (we determine the values of $l_{1}, l_{2}$ at the end), we introduce cut-off functions $\eta_{l} \in C_{c}^{\infty}(\R^{n+1})$ for large $l$,
\begin{gather*}
\eta_{l}(x,t)=\left\{\begin{aligned}
&1, & & (x,t)\in \B_{l_1}(z_l^{(1)}) \cup \B_{l_1}(z_l^{(2)}), \\
&	0, &  &(x,t)\in (\R^{n+1}\setminus  \B_{l_2}(z_l^{(1)})) \cap (\R^{n+1}\setminus  \B_{l_2}(z_l^{(2)})), \\
&\geq 0 ,&  &\text {elsewhere,}
\end{aligned}\right.\\
|\nabla \eta_{l}| \leq \frac{10}{l_{2}-l_{1}},  \quad (x,t) \in \R^{n+1},
\end{gather*}
and set 
$$
H_{l}(\theta)=\eta_{l}W_{l}(\theta).
$$

Next, we  will prove that $H_{l}(\theta)\in \Gamma_{l}$, but the energy of  $H_{l}(\theta)$ contradicts  to  $b_{l,\tau}$.

Multiplying $(1-\eta_{l}) W_{l}(\theta)$ on both sides of \eqref{128}  and integrating by parts, we have$$\int_{\om_l}t^{1-2\si} \nabla((1-\eta_{l}) W_{l}(\theta)) \nabla W_{l}(\theta)\,\d X=N_{\si}\int_{\pa'\om_l} K_{l}(x)H^{\tau}(x)(1-\eta_{l}(x,0))|U_{l}(\theta)|^{p+1}\,\d x.
$$
A direct computation shows that\begin{align*}
&\int_{\om_l^3}t^{1-2\si}|\nabla  W_{l}(\theta)|^{2}\,\d X-N_{\si}\int_{\pa'\om_l^3}K_{l}(x)H^{\tau}(x)|U_{l}(\theta)|^{p+1}\,\d x\\=&\int_{\om_2^l}t^{1-2\si}W_{l}(\theta) \nabla\eta_{l}  \nabla W_{l}(\theta)\,\d X-\int_{\om_2^l}t^{1-2\si}(1-\eta_{l})  |\nabla W_{l}(\theta)|^2\,\d X\\& +N_{\si}\int_{\pa'\om_2^l}K_{l}(x)(1-\eta_{l}(x,0))H^{\tau}(x)|U_{l}(\theta)|^{p+1}\,\d x\\\geq&-\int_{\om_2^l}\Big[\frac{10}{l_{2}-l_{1}}t^{1-2\si}|  W_{l}(\theta)||\nabla   W_{l}(\theta)| +t^{1-2\si}|\nabla   W_{l}(\theta)|^{2}\Big]\,\d X\\&-2A_{1}N_{\si}\int_{\pa'\om_2^l}|U_{l}(\theta)|^{p+1}\,\d x\\\geq&-\int_{\om_2}\frac{10}{l_{2}-l_{1}}t^{1-2\si}|  W_{l}(\theta)||\nabla  W_{l}(\theta)|\,\d X -4A_{1}N_{\si}\int_{\pa'\om_2}|U_{l}(\theta)|^{p+1}\,\d x.
\end{align*}
 Then by Proposition  \ref{prop:1.4},  for all $(x,t)\in \Sigma^+_{l_2}(z_l^{(i)})\setminus  \Sigma^+_{l_1}(z_l^{(i)})$, $i=1,2$, we have \begin{align}\label{129} | W_{l}(\theta)| &\leq \frac{C(n,\si,A_{1})}{|(x-z_{l}^{(i)},t)|^{n-2\si}},
\\
\label{131} |\nabla_x  W_{l}(\theta)| &\leq \frac{C(n,\si,A_{1},A_3)}{|(x-z_{l}^{(i)},t)|^{n+1-2\si}},  \\
\label{131'}t^{1-2\si} |\pa_{t}  W_{l}(\theta)| &\leq \frac{C(n,\si,A_{1},A_3)}{|(x-z_{l}^{(i)},t)|^{n+1-2\si}},  
\end{align}
Consequently, 
\begin{align}
&\int_{\om_l^3}t^{1-2\si}|\nabla  W_{l}(\theta)|^{2}\,\d X-N_{\si}\int_{\pa'\om_l^3}K_{l}(x)H^{\tau}(x)|U_{l}(\theta)|^{p+1}\,\d x\notag\\\displaystyle \geq& \begin{cases}
-C_{0}(n,\si) C(n,\si,A_{1},A_3)\Big[\frac{1}{(l_2-l_1)l_1^{n-2+2\si}}+\frac{1}{l_1^{n-2+2\si}}+\frac{1}{l_1}\Big]\quad &\text{ if }\, 0<\si\leq \frac{1}{2},\\-C_{0}(n,\si) C(n,\si,A_{1},A_3)\Big[\frac{l_2^{2-2\si}}{(l_2-l_1)l_1^{n+2-4\si}}+\frac{l_2^{2-2\si}}{l_1^{n+2-4\si}}+\frac{1}{l_1}\Big]\quad &\text{ if } \,\frac{1}{2} <\si<1,
\end{cases}\notag\\ \displaystyle \geq& \begin{cases}
-C_{0}(n,\si) C(n,\si,A_{1},A_3)\frac{1}{l_1}\quad &\text{ if } \,0<\si\leq \frac{1}{2},\\-C_{0}(n,\si) C(n,\si,A_{1},A_3)\Big[\frac{l_2^{2-2\si}}{l_1^{n-2}}+\frac{1}{l_1}\Big]\quad &\text{ if } \,\frac{1}{2} <\si<1.
\end{cases}\label{133}
\end{align}
Thanks to \eqref{129}-\eqref{131'}, \cite[Lemma A.4]{JLX} and a density argument, we see from \eqref{127} and \eqref{133} that
\begin{align*}
I_{K_{l}, \tau}(H_{l}(\theta))=& \frac{1}{2} \|\eta_{l} W_{l}(\theta)\|^{2} -\frac{N_{\si}}{p+1} \int_{\Rn} K_{l}(x)H^{\tau}(x)|\eta_{l}(x,0) W_{l}(\theta)|^{p+1}\,\d x\\=&\frac{1}{2} \int_{\Rp}t^{1-2\si} |\nabla \eta_{l}|^{2}| W_{l}(\theta)|^{2}\,\d X\\&+\int_{\Rp}t^{1-2\si} \eta_{l} W_{l}(\theta) \nabla \eta_{l}\nabla  W_{l}(\theta)\,\d X\\& +\frac{1}{2} \int_{\Rp}t^{1-2\si}|\eta_{l}|^{2}|\nabla  W_{l}(\theta)|^{2}\,\d X\\&-\frac{N_{\si}}{p+1} \int_{\Rn} K_{l}(x)H^{\tau}(x)|\eta_{l}(x,0) W_{l}(\theta)|^{p+1}\,\d x
\\\leq& I_{K_{l}, \tau}(W_{l}(\theta)) -\frac{1}{2} \int_{\Rp}t^{1-2\si}(1-|\eta_{l}|^{2})|\nabla  W_{l}(\theta)|^{2}\,\d X\\& \displaystyle+\frac{N_{\si}}{p+1} \int_{\Rn} K_{l}(x)H^{\tau}(x)(1-|\eta_{l}(x,0)|^{p+1})|W_{l}(\theta)|^{p+1}\,\d x
\\& +\begin{cases}
C_{0}(n,\si) C(n,\si,A_{1},A_3)\Big[\frac{1}{(l_2-l_1)^2l_1^{n-2+2\si}}+\frac{1}{(l_2-l_1)l_1}\Big]\quad &\text{ if } \,0<\si\leq \frac{1}{2},\\C_{0}(n,\si) C(n,\si,A_{1},A_3)\Big[\frac{l_2^{2-2\si}}{(l_2-l_1)^2l_1^{n-4\si}}+\frac{l_2^{1-2\si}}{(l_2-l_1)l_1^{2n+1-4\si}}\Big]\quad &\text{ if }\, \frac{1}{2} <\si<1,
\end{cases}\\\leq &c_{l, \tau}^{(1)}+c_{l, \tau}^{(2)}-\var_{4}-\frac{1}{2}\int_{ \om_1^l}t^{1-2\si}|\nabla  W_{l}(\theta)|^{2}\,\d X\\&+\frac{N_{\si}}{p+1}\int_{\pa'  \om_1^l} K_{l}(x)H^{\tau}(x)|U_{l}(\theta)|^{p+1} \,\d x\\&\displaystyle +\begin{cases}
C_{0}(n,\si) C(n,\si,A_{1},A_3)\Big[\frac{1}{(l_2-l_1)^2l_1^{n-2+2\si}}+\frac{1}{(l_2-l_1)l_1}\Big]\quad &\text{ if }\, 0<\si\leq \frac{1}{2},\\C_{0}(n,\si) C(n,\si,A_{1},A_3)\Big[\frac{l_2^{2-2\si}}{(l_2-l_1)^2l_1^{n-4\si}}+\frac{l_2^{1-2\si}}{(l_2-l_1)l_1^{2n+1-4\si}}\Big]\quad &\text{ if }\, \frac{1}{2} <\si<1,
\end{cases}\\\leq &c_{l, \tau}^{(1)}+c_{l, \tau}^{(2)}-\var_{4}\\& +\begin{cases}
C_{0}(n,\si) C(n,\si,A_{1},A_3)\frac{1}{l_1}\quad &\text{ if } \,0<\si\leq \frac{1}{2},\\C_{0}(n,\si) C(n,\si,A_{1},A_3)\Big(\frac{l_2^{2-2\si}}{l_1^{n-2}}+\frac{1}{l_1}\Big)\quad &\text{ if }\, \frac{1}{2} <\si<1,\end{cases}\\& \displaystyle +\begin{cases}
C_{0}(n,\si) C(n,\si,A_{1},A_3)\Big[\frac{1}{(l_2-l_1)^2l_1^{n-2+2\si}}+\frac{1}{(l_2-l_1)l_1}\Big]\quad &\text{ if }\, 0<\si\leq \frac{1}{2},\\C_{0}(n,\si) C(n,\si,A_{1},A_3)\Big[\frac{l_2^{2-2\si}}{(l_2-l_1)^2l_1^{n-4\si}}+\frac{l_2^{1-2\si}}{(l_2-l_1)l_1^{2n+1-4\si}}\Big]\quad &\text{ if }\, \frac{1}{2} <\si<1.
\end{cases}
\end{align*}
Now we choose $l_{1}>10 r$,  $l_{2}>200 l_{1}$ to be large enough such that
\be\label{135}
I_{K_{l}, \tau}(H_{l}(\theta))\leq c_{l, \tau}^{(1)}+c_{l, \tau}^{(2)}-\frac{\var_{4}}{2}.
\ee
Then for $l$ large enough (depending on $l_{1}, l_{2}, \var's, C's$), we have
\be\label{136}
H_{l} \in \Gamma_{l}.
\ee
Therefore, for $l$ sufficiently large, we have
\be\label{137}
\max _{\theta \in[0,1]^{2}} I_{K_{l}, \tau}(H_{l}(\theta)) \leq c_{l, \tau}^{(1)}+c_{l, \tau}^{(2)}-\frac{\var_{4}}{2}<b_{l, \tau}.
\ee
However, \eqref{137} cannot hold by  \eqref{136} and the definition of $b_{l, \tau}$. This completes the proof of Theorem \ref{thm:3.1}. 	
\section{\sc Blow up analysis and  proof of main theorems}\label{sec:6}
In this section we present our main result Propostion \ref{prop:5.1}, from which we deduce Theorems \ref{thm:0.1}-\ref{thm:0.3} and Corollary \ref{cor:1}. The crucial ingredients of our proofs are the understanding of the blow up profiles, see the  work  in Jin-Li-Xiong \cite{JLX}.
\subsection{Blow up analysis}
We  present  the result as following:
\begin{prop}\label{prop:5.1}
Suppose that $\{K_{l}\}$ is a sequence of functions  satisfying conditions (i)-(iii) (see Section \ref{sec:3}) and $(H_3)$.  Assume also that there exist some bounded open sets $O^{(1)}, \ldots, O^{(m)} \subset \Rn$
	and some positive constants $\delta_{2}$, $\delta_{3}>0$ such that for any $1 \leq i \leq m$,	
	\begin{gather*}
	\widetilde{O}_{l}^{(i)}-z_{l}^{(i)} \subset O^{(i)} \quad  \text { for all }\, l,\\ 
	\Big\{U \in D^+: I_{K_{\infty}^{(i)}}'(U)=0,\, c^{(i)} \leq I_{K_{\infty}^{(i)}}(U) \leq c^{(i)}+\delta_{2}\Big\}\cap V(1, \delta_{3}, O^{(i)}, K_{\infty}^{(i)})=\emptyset.
	\end{gather*}
	Then for any $\var>0$, there exists integer $\over{l}_{\var, m}>0$ such that for any $l \geq \over{l}_{\var, m}$, there exists  $ U_{l} \in V_{l}(m, \var)\cap D^+$ which solves
\be\label{139}
\left\{\begin{aligned}
&\operatorname{div}(t^{1-2\si}\nabla U_{l})=0&&\text{ in }\,\Rp,\\
&\pa_{\nu}^{\si} U_l=K_{l}(x) U_{l}(x,0)^{\frac{n+2\si}{n-2\si}}&&\text{ on }\,\Rn.
\\
\end{aligned}\right.
\ee
	Furthermore, $U_{l}$ satisfies 
$$
	\sum_{i=1}^{m} c^{(i)}-\var \leq I_{K_{l}}(U_{l}) \leq \sum_{i=1}^{m} c^{(i)}+\var.
$$		
\end{prop}
The proof of Proposition \ref{prop:5.1} is by contradiction arguments, 
 depending on  blow up analysis for a family of equations \eqref{30} approximating Eq. \eqref{139}. More precisely, if the sequence
of  subcritical solutions $U_{l, \tau}$ $(0<\tau<\over{\tau}_{l})$ obtained in Theorem \ref{thm:3.1} is uniformly bounded
as $\tau\to  0$, some local estimates in \cite{JLX}  imply that there exists a subsequence converging to a
positive solution of Eq. \eqref{139}. However, a prior $\{U_{l,\tau}\}$ might blow up, 
 we have to rule out this possibility. Note that $U_{l, \tau}\in V_{l}(m, o_{\var_{2}}(1))$, which consists of functions with $m$ $(m \geq 2)$  bump,     we  apply some  results of blow up analysis developed in \cite{JLX} to conclude that, as $\tau\to  0$, there is no blow up occurring under the hypotheses of Proposition \ref{prop:5.1}.   

In the  following  we  show the boundedness of $\{U_{l, \tau}\}$ (as $\tau \to  0$) by contradiction argument. More precisely, we reach a contradiction by checking balance via a Pohozaev type identity in some proper region. We start by recalling the notion of blow up points, isolated blow up points and isolated simple blow up points.

Let $\om \subset \Rn$ be a domain,  $\tau_{i} \geq 0$ satisfy $\lim _{i \to  \infty} \tau_{i}=0$,  $p_{i}=\frac{n+2\si}{n-2\si}-\tau_{i}$, and let $K_{i} \in C^{1,1}(\om)$ satisfy, for some constants $A_{1}, A_{2}>0$,
\be\label{bdd}
1 / A_{1} \leq K_{i}(x) \leq A_{1} \quad \text { for all }\, x \in \om, \quad\|K_{i}\|_{C^{1,1}(\om)} \leq A_{2}.
\ee
Let $u_{i} \in L^{\infty}(\om) \cap E$ with $u_{i} \geq 0$ in $\Rn$ satisfy
\be\label{subcriticaleq}
(-\Delta)^{\si} u_{i}= K_{i} u_{i}^{p_{i}} \quad \text { in } \,\om.
\ee

We say that $\{u_i\}$ blows up if $\|u_i\|_{L^\infty(\om)}\to \infty$ as $i\to \infty$.

\begin{defn}\label{def:isolatedblowup}
	Suppose that $\{K_i\}$ satisfies \eqref{bdd} and $\{u_i\}$ satisfies \eqref{subcriticaleq}.
	We say a point $\overline y\in \om$ is an isolated blow up point of $\{u_i\}$ if there exist
	$0<\overline r<\mbox{dist}(\overline y,\om)$, $\overline C>0$, and a sequence $y_i$ tending to $\overline y$, such that,
	$y_i$ is a local maximum of $u_i$, $u_i(y_i)\to  \infty$ and
$$
	u_i(y)\leq \overline C | y-y_i|^{-2\sigma/(p_i-1)} \quad \mbox{ for all }\, y\in B_{\overline r}(y_i).
$$
\end{defn}

Let $y_i\to  \overline y$ be an isolated blow up point of $u_i$, define
\be\label{def:average}
\overline u_i(r)=\frac{1}{|\pa B_r(y_i)|} \int_{\pa B_r(y_i)}u_i,\quad r>0,
\ee
and
$$
\overline w_i(r)=r^{2\sigma/(p_i-1)}\overline u_i(r), \quad r>0.
$$

\begin{defn}\label{isolatedsimpleblowup}
	We say $y_i \to \overline y\in \om$ is an isolated simple blow up point, if $y_i \to \overline y$ is an isolated blow up point, such that, for some
	$\rho>0$ (independent of $i$), $\overline w_i$ has precisely one critical point in $(0,\rho)$ for large $i$.
\end{defn}

Utilizing these notions, we present some facts. By some standard blow up arguments, the blow up points  cannot occur in $\Rn \setminus (\bigcup_{i=1}^{m}{\widetilde{O}}_{l}^{(i)})$ since the energy of $\{U_{l, \tau}\}$ in the region is small using the fact that $U_{l, \tau} \in V_{l}(m, o_{\var_{2}}(1))$ and the definition of $V_{l}(m, o_{\var_{2}}(1))$.
Hence the blow up points can occur only in $\bigcup_{i=1}^{m}{\widetilde{O}}_{l}^{(i)}$. By the structure of functions in $V_{l}(m,o_{\var_{2}}(1))$ and some blow up arguments  obtained in \cite[Proposition 5.1]{JLX},  there are at most $m$ isolated blow up points, namely, the blow up occurs in  $\{(\over{y}_{1},0), \ldots, (\over{y}_{m},0)\}$  for some $\over{y}_{i} \in \widetilde{O}_{l}^{(i)}$ $(1\leq i\leq m)$. Futhermore, we conclude from \cite[Proposition 4.16]{JLX}  that an isolated blow up point has to be an isolated simple blow up point. From the structure of functions in $V_{l}(m,o_{\var_{2}}(1))$ we know that if the blow up does occur, there have to be exactly $m$ isolated simple blow up points,  see \cite[Section 4]{JLX} for more details. 

Let us consider this situation only, namely, $\{\over{Y}_{i}=(\over{y}_{i},0), 1\leq i\leq m\}$  is the blow up set and they are all isolated simple blow up points. Moreover,  in our situation, $K_{i}(x)=K(x)H^{\tau_i}(x)$ is the sequence of functions in Eq. \eqref{subcriticaleq}.  We may assume that  the blow up occurs at $U_{i}=U_{l, \tau_{i}}$ and we can apply some blow up analysis   results in \cite{JLX}  to $U_{i}$. Here and in the following we suppress the
dependence of $l$ in the notation since $l$ is fixed in the blow up analysis.

Now we  complete the proof of Proposition \ref{prop:5.1} by  checking balance
via a Pohozaev type indentity.
\begin{proof}[Proof of Proposition \ref{prop:5.1}]
Let $U_i$ be the extension of $u_i$ (see \eqref{extension})  to the solution of Eq. \eqref{subcriticaleq} with $K_{i}=KH^{\tau_i}$. 	We may assume without loss of generality that $\over{y}_1=0$ and $y_{i}\to 0$ be the sequence as in Definition   \ref{isolatedsimpleblowup}.  Applying the Pohozaev identity \cite[Proposition 4.7]{JLX} to $U_{i}$, we obtain 
\be\label{141}
\int_{\pa ' \B_{R}^{+}(y_{i})} B'(Y, U_{i}, \nabla U_{i}, R, \si)+\int_{\pa '' \B_{R}^{+}(y_{i})} s^{1-2 \si} B''(Y, U_{i}, \nabla U_{i}, R, \si)=0,
\ee
where $$
B'(Y, U_{i}, \nabla U_{i}, R, \si)=\frac{n-2 \si}{2} K_{i} U_{i}^{p_{i}+1}+\langle Y, \nabla U_{i}\rangle K_{i} U_{i}^{p_{i}}
$$
and $$
B''(Y, U_{i}, \nabla U_{i}, R, \si)=\frac{n-2 \si}{2} U_{i} \frac{\pa  U_{i}}{\pa  \nu}-\frac{R}{2}|\nabla U_{i}|^{2}+R\Big|\frac{\pa  U_{i}}{\pa  \nu}\Big|^{2}.
$$
	
	We are going to derive a contradiction to \eqref{141}, by	showing that for small  $R>0$,
\be\label{Pohozaeve1}
	\limsup _{i \to  \infty} U_{i}(Y_{i})^{2} \int_{\pa ' \B_{R}^{+}(y_{i})} B'(Y, U_{i}, \nabla U_{i}, R, \si) \leq 0
\ee
	and
\be\label{Pohozaeve2}
	\limsup _{i \to  \infty} U_{i}(Y_{i})^{2} \int_{\pa '' \B_{R}^{+}(y_{i})} s^{1-2 \si} B''(Y, U_{i}, \nabla U_{i}, R, \si)<0.
\ee
Hence Proposition \ref{prop:5.1} will be established.

Let $\mathscr{S}=\{\over{Y}_{1},\over{Y}_{2},\ldots,\over{Y}_{m}\}$, applying B\^ocher Lemma in \cite[Lemma 4.10]{JLX} and  maximum principle, we deduce that
$$
U_{i}(Y_{i}) U_{i}(Y)\to  G(Y)=\sum_{k=1}^{m} b_{k}|Y-\over{Y}_{k}|^{2\si-n}+h(Y)\quad \text { in }\, C_{loc}^{\al}(\over{\R}_{+}^{n+1} \setminus  \mathscr{S}) \cap C_{loc}^{2}(\Rp)
$$
	and   
$$
 u_{i}(y_{i}) u_{i}(y)\to  G(y,0)=\sum_{k=1}^{m} b_{k}|y-\over{y}_{k}|^{2\si-n}+h(y,0)\quad \text { in }\, C_{loc}^{2}(\Rn\setminus  \{\over{y}_1,\over{y}_2,\ldots,\over{y}_m\})
$$
	as $i\to  \infty$, where $b_{k}>0$ $(1 \leq k \leq m)$ and $h(Y)$ satisfies
	$$
	\left\{\begin{aligned}
	&\operatorname{div}(s^{1-2 \si} \nabla h)=0 && \text { in }\, \Rp, \\
	& \pa _{\nu}^{\si} h=0 && \text { on }\, \Rn.
	\end{aligned}\right.
	$$	 
In particular, in a small half punctured disc at $Y_1$, we have 
	$$
	\lim _{i \to  \infty} U_{i}(Y_{i}) U_{i}(Y)=b_{1}|Y-Y_1|^{2\si-n}+b+w(Y),
	$$
	where $b>0$ is a positive constant and $w(Y)$ is a smooth function near $Y_1$ with $w(Y_1)=0$.

Now if we	choose $R>0$ small  enough, it is easy to verify \eqref{Pohozaeve2} by
	\begin{align*}
	&	\limsup _{i \to  \infty} U_{i}(Y_{i})^{2} \int_{\pa '' \B_{R}^{+}(y_{i})} s^{1-2 \si} B''(Y, U_{i}, \nabla U_{i}, R, \si) \\
	=&\int_{\pa ''\B_{R}^{+}(y_{i})} s^{1-2 \si} B''(Y, G, \nabla G, R, \si)\\=&-\frac{(n-2 \si)^{2}}{2} b_1^{2} \int_{\pa ''\B_{1}^{+}} s^{1-2 \si}\,\d y\,\d s+o_{R}(1)<0.
	\end{align*}

	On the other hand, via integration by parts, we have
	\begin{align*}
	&\int_{\pa ' \B_{R}^{+}(Y_{i})} B'(Y, U_{i}, \nabla U_{i}, R, \si) \\
	=&\frac{n-2 \si}{2} \int_{B_{R}(y_{i})} K_i u_{i}^{p_{i}+1}+\int_{B_{R}(y_{i})}\langle y-y_i, \nabla u_{i}\rangle K_i u_{i}^{p_{i}} \\
	=&\frac{n-2 \si}{2} \int_{B_{R}(y_{i})} K_i u_{i}^{p_{i}+1}-\frac{n}{p_{i}+1} \int_{B_{R}(y_{i})} K_i u_{i}^{p_{i}+1} \\
	& -\frac{1}{p_{i}+1} \int_{B_{R}(y_{i})}\langle y-y_i, \nabla K_i \rangle u_{i}^{p_{i}+1}+\frac{R}{p_{i}+1} \int_{\pa  B_{R}(y_{i})} K_i u_{i}^{p_{i}+1} \\
	\leq&-\frac{1}{p_{i}+1} \int_{B_{R}(y_{i})}\langle y-y_i, \nabla K_i\rangle u_{i}^{p_{i}+1}+C u_{i}(y_i)^{-(p_{i}+1)},
	\end{align*}
	where \cite[Proposition 4.9]{JLX} is used in the last inequality. Hence \eqref{Pohozaeve1} follows from \cite[Corollary 4.15]{JLX}.
	
Finally, from the above argument we know that there will be no blow up occur under the hypotheses of  Proposition \ref{prop:5.1} and hence Proposition \ref{prop:5.1} is established.
\end{proof}
\subsection{Proof of main theorems}
We are now ready to complete the proofs of the main results in our paper.
\begin{proof}[Proof of Theorem \ref{thm:0.2}] 
Suppose Theorem \ref{thm:0.2} is false, then for some $\over{\var}>0$ (we can assume $\over{\var}$ to be very small) and $k \in\{2,3,4, \ldots\}$, there exists a sequence of integers $I_{l}^{(1)}, \ldots, I_{l}^{(k)}$, such that
$$
	\lim _{l \to  \infty}|I_{l}^{(i)}|=\infty ,\quad \lim _{l \to  \infty}|I_{l}^{(i)}-I_{l}^{(j)}|=\infty, ~ i \neq j,
$$
	but Eq. \eqref{maineq2} has no solution in $V(k, \over{\var}, B_{ \over{\var}}(x_{l}^{(1)}), \ldots, B_{\over{\var}}(x_{l}^{(k)}))$ which satisfies
	$k c-\over{\var} \leq I_{K} \leq k c+\over{\var}$, where $$c=\frac{\si N_{\si}}{n}(K(x^{*}))^{(2\si-n) / 2\si}(S_{n,\si})^{n/\si}, \quad x_{l
	}^{(i)}=x^{*}+(I_{l}^{(i)} T, 0,\ldots,0).$$
		
	For $\var>0$ small, define
	\begin{gather*}
	K_{l}(x_{1}, x_{2},\ldots, x_{n})=K(x_{1}, x_{2},\ldots, x_{n}),\\
	O_{l}^{(i)}=B_{\var}(x_{l}^{(i)}),\quad  \widetilde{O}_{l}^{(i)}=B_{2 \var}(x_{l}^{(i)}),\\R_{l}=\min _{i \neq j}\big\{\sqrt{|I_{l}^{(i)}|}, \sqrt{|I_{l}^{(i)}-I_{l}^{(j)}|}\big\},\\K_{\infty}^{(i)}(x_{1}, x_{2},\ldots, x_{n})=K_{\infty}(x_{1}, x_{2}, \ldots,x_{n})=\lim _{l \to  \infty} K(x_{1}+lT, x_{2},\ldots, x_{n}),\\a^{(i)}=K(x^{*}).
	\end{gather*}
	It is easy to see that $K_{\infty}$ is $T$--periodic in $x_{1}$ and satisfies  assumption $(H_3)$ and 
$$
	K_{\infty}(x^{*})=\sup _{x \in \Rn} K_{\infty}(x)>0.
$$
	Under our hypothesis, it follows from \cite[Theorem 5.4]{JLX} that it is impossible to have more than one   blow up point.
	Applying \cite[Theorem 5.6]{JLX}  and Proposition \ref{prop:5.1}, we
	immediately get a contradiction.
\end{proof}
\begin{proof}[Proof of Theorem \ref{thm:0.3}]
	Theorem \ref{thm:0.3} can be proved similarly to Theorem \ref{thm:0.2} before, we omit it here.
\end{proof}
\begin{proof}[Proof of Theorem \ref{thm:0.1}]
	Let $\widetilde{x} \in \Sn$ be the north pole and make a stereographic projection to the equatorial plane of $\Sn$, then  \eqref{maineq} is transformed to \eqref{maineq1}.
Here	$K(x) \in L^{\infty}(\Rn)$ satisfies, for some  constants $A_{1}>1$, $R>1$ and $K_{\infty}>0$,
$$
1/A_1\leq K \leq A_1,\quad 	K \in C^{0}(\Rn \setminus  B_{R}),\quad 
	\lim _{|x| \to  \infty} K(x)=K_{\infty}.
$$	Let $\psi \in C^{\infty}(\Rn)$ satisfies assumption $(H_3)$ and 
$$
	\|\psi\|_{C^{2}(\Rn)}<\infty,\quad 
	\lim _{|x| \to  \infty } \psi(x)=: \psi_{\infty}>0,\quad 
	\langle \nabla \psi,
	x\rangle<0,~  \forall\, x \neq 0.
$$
It follows   that
	$\psi$ violates the Kazdan-Warner type condition \eqref{KW} and  therefore$$
	(-\Delta)^{\si} u=\psi|u|^{\frac{4\si}{n-2\si}} u  \quad  \text { in }\,  \Rn
$$has no nontrivial solution in $E$. 	
	
	For any $\var\in (0,1)$, $k=1,2,3, \ldots$, $m=2,3,4, \ldots$, we choose an integer $\over{k}$ such that for any $2 \leq s \leq m$, there holds $C_{\over{k}}^s\geq k$, where $C_{\over{k}}^s$ is a combination number. Then we  choose $e_{1}, e_{2}, \ldots, e_{\over{k}} \in \pa  B_{1}$ to	be $\over{k}$ distinct points.
		Let
	$$
	\delta_{R}=\max _{|x| \geq R}|K(x)-K_{\infty}|+\max _{|x| \geq R}|\psi(x)-\psi_{\infty}|, \quad R>1,
	$$
	and $\widetilde{\om}_{l}^{(i)}$ be the connected component of$$
	\big\{x : \var(\psi(x-l e_{i})-\psi_{\infty})+K_{\infty}-\delta_{\sqrt{l}}>K(x)\big\},
	$$which contains $x=l e_{i}$.	Define
	\begin{gather*}
	R_{l}=\min _{1 \leq i \leq m} \sup \{|x-l e_{i}| : x \in \widetilde{\om}_{l}^{(i)}\},\\	K_{\var, k, m, l}=\left\{\begin{aligned}
	&\var(\psi(x-l e_{i})-\psi_{\infty})+K_{\infty}-\delta_{\sqrt{l}} && \text { if } x \in \widetilde{\om}_{l}^{(i)}, \\
	&K(x) && \text { otherwise. }
	\end{aligned}\right.
	\end{gather*}
	It is easy to prove that\begin{gather*}
	\operatorname{diam}(\widetilde{\om}_{l}^{(i)}) \leq \sqrt{l}\quad  \text { for large }\, l,\\\lim _{l \to  \infty} R_{l}=\infty.
	\end{gather*}
	
Now	we consider the equation:
\be\label{153}
\left\{\begin{aligned}
&\operatorname{div}(t^{1-2\si}\nabla U)=0&&\text{ in }\,\Rp,\\
&\pa_{\nu}^{\si} U=K_{\var, k, m,l} U(x,0)^{\frac{n+2\si}{n-2\si}}&&\text{ on }\,\Rn.
\\
\end{aligned}\right.
\ee
	For any $2 \leq s \leq m$, we claim that, for $l$ large enough,  Eq. \eqref{153}	has at least $k$ solutions with $s$ bumps.  
	
To verify it,	let $e_{j_{1}}, \ldots, e_{j_{s}}$ be any $s$ distinct points among $e_{1}, \ldots, e_{\over{k}}$. For $i=1,2, \ldots, s$, we define  
	\begin{gather*}
	z_{l}^{(i)}=l e_{j_{i}} ,\\ O_{l}^{(i)}=B_{1}(z_{l}^{(i)}),\quad \widetilde{O}_{l}^{(i)}=B_{2}(z_{l}^{(i)}), \\
	K_{\infty}^{(i)}=\var(\psi-\psi_{\infty})+K_{\infty},\\ a^{(i)}=\var(\psi(0)-\psi_{\infty})+K_{\infty}.
	\end{gather*}
Then using an argument similar to the proof of Theorem \ref{thm:0.2}, we can prove  that there exists at least
	a solution in $V_{l}(s, \var)$ for large $l$. It is also easy to see that if we choose a different set of $s$ points among $\{e_{1}, \ldots, e_{\over{k}}\}$, we get different solutions since
	their mass are distributed in different regions by the definition of $V_{l}(s, \var)$. By the choice of $\over{k}$, there are at least $k$ different sets of such points. Therefore  Eq. \eqref{153} has at least $k$ solutions for large $l$.  This gives the proof of  the claim.
	
Finally, we fix $l$ large enough to make the above argument work for all $2 \leq s \leq m$ and set $K_{\var, k, m}=K_{\var, k, m, l}$. Thus there exists at least $k$ solutions with $s$ ($2 \leq s \leq m$) bumps to the following equations		$$
	(-\Delta)^{\si} u= K_{\var, k, m} u^{\frac{n+2\si}{n-2\si}},\quad u>0 \quad \text { in }\, \Rn.
$$Theorem \ref{thm:0.1} follows from the above after performing a sterographic projection on the original equation, we omit the details here.
\end{proof}
\begin{proof}[Proof of Corollary \ref{cor:1}]We can see from the proof in Theorem \ref{thm:0.1} that if $K\in C^{\infty}(\Sn)$, then $K_{\var,k,m}-K\in C^{\infty}(\Sn)$ can also be achieved. 
\end{proof}
\appendix

\section{\sc Some a priori estimates}\label{sec:1}
In this section we present some a priori estimates of  positive solutions to the  equation
$$
(-\Delta)^{\si} u=K(x)u^{\frac{n+2\si}{n-2\si}-\tau},\quad  |x|\geq 1
$$
with $\tau \geq 0$ small. For this purpose,  we use the  extension formula \eqref{extension} to consider the related degenerate elliptic equations.  Our proofs are in the spirit of those in \cite{XaYa,CJYX,JLX} and some standard rescaling arguments. We omit some proofs of several results by giving appropriate references.
\begin{prop}\label{prop:1.1}
	Suppose that $K \in L^{\infty}(\Rn \setminus  B_{1})$   satisfies $1/A_0\leq K(x)\leq A_0$, $\forall\, x\in \Rn \setminus  B_{1}$ for some positive constant $A_0>1$. Then there exists some positive constants	$\mu_{1}=\mu_{1}(n,\si, A_0)$ and $C(n,\si, A_0)$ such that for any  positive solution  of
	$$
	\left\{\begin{aligned}
	&	\operatorname{div}(t^{1-2 \si} \nabla U)=0& & \text { in }\, \Rp, \\
	&\pa_{\nu}^{\si} U=K(x) U(x, 0)^{\frac{n+2\si}{n-2\si}}& & \text { on }\, \Rn \setminus  B_{1},
	\end{aligned}\right.
	$$
	with $\nabla U \in L^{2}(t^{1-2 \si} , \Rp \setminus  \B_{1}^{+})$, 
	$U(x, 0) \in L^{2^{*}_{\si}}(\Rn \setminus  B_{1})$ and \begin{align}\label{prop:1.1-1}
	\int_{\Rp \setminus  \B_{1}^{+}} t^{1-2 \si}|\nabla U|^{2}\,\d X\leq \mu_{1},\end{align}  we have
	\begin{align*}
	\sup_{X\in \over{\R}^{n+1}_{+}\setminus  \B^+_2}	|X|^{n-2\si}U(X)\leq C(n,\si, A_0).
	\end{align*} 
\end{prop}
\begin{proof}
We perform a Kelvin transformation on $U(X)$. Let 
	\begin{gather*}
	\widetilde{X}=(\widetilde{x}, \widetilde{t})=\frac{X}{|X|^{2}}, \quad|X| \geq 1,\\V(\widetilde{X})=\frac{1}{|\widetilde{X}|^{n-2\si}} U\Big(\frac{\widetilde{x}}{|\widetilde{X}|^2},\frac{\widetilde{t}}{|\widetilde{X}|^2}\Big) .
	\end{gather*} 
Thus $V(\widetilde{X})$ satisfies  
$$
	\left\{\begin{aligned}
	&	\operatorname{div}(\widetilde{t}^{1-2 \si} \nabla V)=0& & \text { in }\, \Rp, \\
	&\pa_{\nu}^{\si}V=K(\frac{\widetilde{x}}{|\widetilde{x}|^{2}}) V(\widetilde{x}, 0)^{\frac{n+2\si}{n-2\si}}& & \text { on } \, B_{1}\setminus \{0\}.
	\end{aligned}
	\right.
$$
	By \eqref{prop:1.1-1} and Lemma \ref{traceexterior}, we have 	$$	\int_{\B_{1}^{+}} \widetilde{t}^{1-2 \si}|\nabla V|^{2}\,\d  \widetilde{X}+\int_{B_{1}} |V(\widetilde{x}, 0)|^{2^{*}_{\si}}\,\d  \widetilde{x} \leq C_{0}(n,\si) \mu_{1},
$$see also \cite[lemma 2.6]{FV} for  more delicate computations. 
It follows from the regularity results in  	\cite[Lemma 2.8]{JLX} that $V(\cdot,0)\in L^{q}(B_{9/10})$ for some $q>\frac{n}{2\si}$. Using the Harnack inequality in \cite[Proposition 2.6(iii)]{JLX} (see also \cite{XaYa}),  we just need to give an a priori bound of $\|V(\cdot,0)\|_{L^{\infty}(B_{0.5})}$  to complete the proof.  Now the rest proof  can be done in the same procedure as in the proof of \cite[Lemma 6.5]{Niu18}, we omit it here. 
\end{proof}
\begin{prop}\label{prop:1.2}
	Let $\mu_{1}$ and  $C(n,\si,A_0)$ be the positive constants in
	Proposition \ref{prop:1.1}. Then for any $2<l_{1}<l_{2}<\infty$, there exists a positive constant $R_{1}=R_{1}(n, \si,A_0, \mu_{1}, l_{1}, l_{2})>l_{2}$ such that for any $K \in L^{\infty}(B_{R_{1}} \setminus  B_{1})$ with
	$1/A_0\leq K(x)\leq A_0$, $\forall\,x\in B_{R_1} \setminus  B_{1}$, and any positive solution of
$$	\left\{\begin{aligned}
	&\operatorname{div}(t^{1-2 \si} \nabla U)=0 && \text { in }\, \Rp, \\
	&\pa_{\nu}^{\si} U=K(x) U(x, 0)^{\frac{n+2\si}{n-2\si}}& & \text { on } \,B_{R_{1}} \setminus  B_{1},
	\end{aligned}\right.
	$$
	with $$
	\int_{\B_{R_{1}}^{+} \setminus  \B_{1}^{+}} t^{1-2 \si}|\nabla U|^{2}\,\d X+\int_{B_{R_{1}} \setminus  B_{1}}|U(x, 0)|^{2^{*}_{\si}}\,\d x\leq \mu_{1},	$$
	 we have$$
	\sup_{X\in \over{\B }^+_{l_2}\setminus  \B^+_{l_1}}	|X|^{n-2\si}U(X) \leq 2 C(n,\si, A_0).	$$
\end{prop}
\begin{proof}
	The proof can be done by using contradition argument and Proposition \ref{prop:1.1}, we omit it here and also refer to \cite[Lemma 6.6]{Niu18} for a similar proof.
\end{proof}
\begin{prop}\label{prop:1.3}
	Suppose that $l_{2}>100 l_{1}>100$, $K \in C^{1}(B_{l_{2}} \setminus  B_{l_{1}})$ and   $\|K\|_{C^{1}(B_{l_{2}} \setminus  B_{l_{1}})}\leq A_1$ for some positive constant $A_1$.  Then for any  positive solution 
	$$
	\left\{\begin{aligned}
	&\operatorname{div}(t^{1-2 \si} \nabla U)=0 && \text { in }\, \Rp, \\
	&\pa_{\nu}^{\si}  U= K(x) U(x, 0)^{\frac{n+2\si}{n-2\si}}& & \text { on }\, B_{l_{2}} \setminus  B_{l_{1}},
	\end{aligned}\right.
	$$
	satisfying 
\be\label{prop1.3-1}
	\sup_{x\in B_{l_{2}} \setminus  B_{l_{1}}}|x|^{n-2\si}U(x, 0) \leq A
\ee for some  constant $A>1$,
	we have$$
	\sup_{X\in\over{\B}_{l_2/4}^+\setminus  \B_{4l_1}^+}|X|^{n+1-2\si}|\nabla_x U| \leq C(n,\si,A_1,A),
	$$
	and $$
	\sup_{X\in\over{\B}_{l_2/4}^+\setminus  \B_{4l_1}^+}|X|^{n}|t^{1-2\si}\pa_t U| \leq C(n,\si,A_1,A),
	$$where $C(n,\si,A_1,A)$ is some positive constant depending only on $n,\si,A_1$ and $A$. 
\end{prop}
\begin{proof}
	For any $r \in(4 l_{1}, l_{2}/4)$,  we have
$$\left\{\begin{aligned}
	&\operatorname{div}(t^{1-2\si} \nabla U)=0&& \text { in }\, \Rp, \\
	&\pa_{\nu}^{\si}  U= K(x) U(x, 0)^{\frac{n+2\si}{n-2\si}}& & \text { on  }\,B_{2r}\setminus  B_{r/2},
	\end{aligned}\right.
$$	and
	$\sup_{x\in B_{2r}\setminus  B_{r/2}}	U(x,0) \leq (r/2)^{n-2\si} A$
	by \eqref{prop1.3-1}. 
	
	Let $V(X)=r^{n-2\si} U(rX)$, then it solves
	$$	\left\{\begin{aligned}
	&\operatorname{div}(t^{1-2 \si} \nabla V)=0 && \text { in }\, \Rp, \\
	&\pa_{\nu}^{\si} V=r^{-2n-1}K(rx)V(x,0)^{\frac{n+2\si}{n-2\si}} & & \text { on }\,B_2\setminus  B_{1/2}.
	\end{aligned}\right.$$
	Note that $V(x,0) \leq C(n,\si,A)$ and $r^{-2n-1}|K(rx)V(x,0)^{\frac{n+2\si}{n-2\si}}| \leq C(n,\si,A_1,A)$
	in the annulus $\{x \in \Rn : 1/2 \leq |x| \leq 2\}$, we deduce from  the regularity results in Section 2.2 of \cite{JLX} and  the proof of Proposition 2.19 in \cite{JLX} (see also \cite{XaYa})  that
	$$
	\sup_{|X|=1}	|\nabla_x V| \leq C(n,\si,A_1,A),
$$and
	$$
	\sup_{|X|=1}	|t^{1-2\si}\pa_{t} V| \leq C(n,\si,A_1,A).
$$
	As a consequence, $$
	\sup_{|X|=r}|X|^{n+1-2\si}	|\nabla_x U| \leq C(n,\si,A_1,A) ,
$$and$$ 
	\sup_{|X|=r}	|X|^{n}|t^{1-2\si}\pa_t U| \leq C(n,\si,A_1,A).
$$This finishes the proof.
\end{proof}
\begin{prop}\label{prop:1.4}
	Let  $\mu_1$, $R_{1}$ and $C(n, \si,A_{0})$ be the constants in Proposition \ref{prop:1.2}. Then	for any $2<l_{1}<l_{2}<\infty$, there exist some positive constants $\mu_{2}=\mu_{2}(n,\si, A_{0}) \leq \mu_{1}$, $\over{\tau}=\over{\tau}(n, \si,l_{1}, l_{2}, A_{0})$, such that for any $0 \leq \tau \leq \over{\tau}$, $K\in L^{\infty}(B_{R_1}\setminus  B_1)$ with $1/A_0\leq K(x)\leq A_0$, $\forall\,x\in B_{R_1}\setminus  B_1$, and any positive solution of
	$$
	\left\{\begin{aligned}
	&\operatorname{div}(t^{1-2 \si} \nabla U)=0 && \text { in }\, \Rp, \\
	&\pa_{\nu}^{\si}  U= K(x) U(x, 0)^{\frac{n+2\si}{n-2\si}-\tau}& &\text { on }\, B_{2R_1}\setminus  B_1,
	\end{aligned}\right.
	$$
	with
	$$
	\int_{\B_{2R_{1}}^{+} \setminus  \B_{1}^{+}} t^{1-2\si}|\nabla U|^{2}\,\d X +\int_{B_{2R_1}\setminus  B_1} |U(x, 0)|^{2^{*}_{\si}}\,\d x \leq \mu_{2},
	$$
	we have
\be\label{prop:2.4-1}
	\sup_{X\in \over{\B}^+_{l_2}\setminus  \B^+_{l_1}}	|X|^{n-2\si} U(X) \leq 3 C(n, \si, A_0).
\ee
	Furthermore, if $K\in C^{1}(B_{R_1}\setminus  B_1)$ with $\|K\|_{C^{1}(B_{R_1}\setminus  B_1)}\leq A_1$ for some positive constant $A_1$, we have$$
	\sup_{X\in \over{\B}^+_{l_2}\setminus  \B^+_{l_1}}	|X|^{n+1-2\si}|\nabla_x U|\leq 2 C(n, \si, A_1,A),
$$
	and
$$
	\sup_{X\in \over{\B}^+_{l_2}\setminus  \B^+_{l_1}} |X|^{n}|\pa_{t} U|\leq  2 C(n, \si, A_1,A) ,
$$
	where $C(n, \si, A_1,A)$ is the constant in  Proposition \ref{prop:1.3} with $A$ replaced by $3C(n,\si,A_0)$.
\end{prop}
\begin{proof}
	Suppose the contrary, then there exists a sequence $0\leq \tau_{j} \to  0$ and  $U_{j}>0$ satisfying
	\begin{gather*}
	\left\{\begin{aligned}
	&\operatorname{div}(t^{1-2 \si} \nabla U_{j})=0 && \text { in }\, \Rp, \\
	&\pa_{\nu}^{\si} U_{j}= K(x) U_{j}(x, 0)^{\frac{n+2\si}{n-2\si}-\tau_j}& & \text { on }\, B_{2R_1}\setminus  B_1,
	\end{aligned}\right.
	\\
	\int_{\B_{2R_{1}}^{+} \setminus  \B_{1}^{+}} t^{1-2\si}|\nabla U_{j}|^{2}\,\d X +\int_{B_{2R_1}\setminus  B_1} |U_{j}(x, 0)|^{2^{*}_{\si}}\, \d x\leq \mu_{2},
	\end{gather*}
	but$$
	\sup _{X\in \over{\B}^+_{l_2}\setminus  \B^+_{l_1}}\{|X|^{n-2\si}U_{j}(X)\} > 3 C(n,\si, A_0).
$$
	
	Choosing $\mu_{2} \in(0, \mu_{1})$ to be small, we use  an argument similar to  the proof in \cite[Lemma 6.5]{Niu18} to obtain$$
	U_{j}(X) \to  	U(X) \quad  \text { in }\, C_{loc}^{ \al/2}(\B^+_{2R_1}\setminus  \B^+_{1}).
	$$Moreover,  $U(X)$ satisfies
\be\label{13}
	\sup _{X\in \over{\B }^+_{l_2}\setminus  \B^+_{l_1}}\{|X|^{n-2\si}U(X)\} \geq  3 C(n,\si, A_0).
\ee
	However,	 by Proposition \ref{prop:1.2} we have
$$
	\sup _{X\in \over{\B}^+_{l_2}\setminus  \B^+_{l_1}}\{|X|^{n-2\si}U(X)\} \leq 2 C(n,\si, A_0),
$$
	which contradicts to \eqref{13}.   The rest of the proof  can be done in a similar way by using Proposition \ref{prop:1.3} and contradiction argument, we omit it here.
\end{proof}
\section{\sc A minimization problem}\label{sec:2}
For any  $z_1,z_{2}  \in\Rn$ satisfy $|z_{1}-z_{2}| \geq 10$, denote $\om:=\Rp \setminus \{\B^+_{1}(z_{1}) \cup \B^+_{1}(z_{2})\}$. We define $D_{\om}$ by taking the closure of $C_{c}^{\infty}(\over{\om})$ under the norm
\begin{align*}
\|U\|_{D_{\om}}:=\Big(\int_{\om}t^{1-2\si}|\nabla U|^2\,\d X\Big)^{1/2}+\Big(\int_{\pa'\om}|U(x,0)|^{2^{*}_{\si}}\,\d x\Big)^{1/2^{*}_{\si}}.
\end{align*}
Clearly, $D_{\om}$ is a Banach space.

By  some appropriate extension  and using \eqref{trace}, we have a   Sobolev trace  inequality on $D_{\om}$:
\begin{lem}\label{traceexterior}
	Let $D_{\om}$ be defined as above. There exists some positive constant
	$C(n,\si)$ such that for all $ U \in D_{\om}$, there holds
$$
	\Big(\int_{\pa'\om}|U(x,0)|^{2_{\si}^{*}}\,\d x\Big)^{1/2^{*}_{\si}} \leq C(n,\si)\Big(\int_{\om } t^{1-2\si} |\nabla U|^2\, \d X\Big)^{1 / 2},
$$
	where	the constant $C(n,\si)$   depends only
	on $n$ and $\si$. In particular, it does not depend on $z_{1}, z_{2}$ provided $|z_{1}-z_{2}| \geq 10$.
\end{lem}
\begin{proof}
	Its proof can be done as in \cite[Lemma A.1]{JX2} or \cite[Lemma 2.6]{FV}  with minor modifications, so we omit it here.
\end{proof}

Let  $K\in L^{\infty}(\pa'\om)$ satisfy $\|K\|_{L^{\infty}(\pa'\om)}\leq A_0$ for some  constant $A_{0}>0$.  We  define a functional on $D_{\om}$ by
$$
I_{K,\om}(U):=\frac{1}{2}\int_{\om}t^{1-2\si}|\nabla U|^2\,\d X -\frac{N_{\si}}{2_{\si}^{*}-\tau} \int_{\pa'\om } K(x)H^{\tau}(x)|U(x,0)|^{2_{\si}^{*}-\tau}\,\d x
$$
with $\tau\geq 0$ small. For any $U \in D_{\om}$, using H\"older inequality and Lemma \ref{traceexterior},   we have
\begin{align}
\Big| I_{K,\om}(U)-\frac{1}{2} \int_{\om}t^{1-2\si}|\nabla U|^2\,\d X\Big|\notag \leq  & A_0 C_0(n,\si)\Big(\int_{\pa'\om }|U(x,0)|^{2_{\si}^{*}}\,\d x\Big)^{(2_{\si}^{*}-\tau)/2_{\si}^{*}}\notag\\\leq &  A_0 C_0(n,\si)\Big(\int_{\om}t^{1-2\si}|\nabla U|^2\,\d X\Big)^{(2_{\si}^{*}-\tau)/2},\label{16}
\end{align}
where $C_0(n,\si)$ denotes some universal constant  which can vary from line to line.

\begin{prop}\label{prop:2.1}
	Let $D_{\om}$ be defined as above. There exist some constants $r_{0}=r_{0}(n,\si, A_{0})\in (0,1)$ and 	$C_{1}=C_{1}(n,\si)>1$ such that for any $z_{1},z_{2}  \in \Rn$ with $|z_{1}-z_{2}| \geq 10$, and $W \in W^{\si,2}(\pa'' \om)$ with $r=\|W\|_{W^{\si,2}(\pa'' \om)}\leq r_{0}$, the following minimum is achieved:$$
	\min \Big\{I_{K,\om}(U):U \in D_{\om},\, U|_{\pa'' \om}=W, \,\int_{\om}t^{1-2\si}|\nabla U|^2\,\d X\leq C_{1} r_{0}^{2}\Big\}.
	$$The minimizer is unique (denoted $U_{W}$) and satisfies $
	\int_{\om}t^{1-2\si}|\nabla U_{W}|^2\,\d X\leq  C_{1} r^{2}/2$. Furthermore, the map $W \mapsto U_{W}$ is continuous from $ W^{\si,2} (\pa'' \om)$ to $D_{\om}$. In particular, 	the constants $r_{0}$ and $C_1$  are independent 	of  $z_{1}, z_{2}$  provided $|z_{1}-z_{2}| \geq 10$.
\end{prop}
\begin{rem} Following the  terminology   in \cite{Adams},	$W^{\si,2}(\pa'' \om)$  stands for fractional order Sobolev space defined on the boundary $\pa'' \om$, obtained by transporting (via a partition of unity and pull-back) the standard scale $W^{\si,2}(\Rn)=H^{\si}(\Rn)$. We refer the readers to  \cite[p.215]{Adams} for more details.
\end{rem}

\begin{proof}[Proof of Proposition \ref{prop:2.1}]
	We claim that there exist some constant $C_{1}=C_{1}(n,\si)>0$ and	$\over{W} \in  D_{\om}$ such that
\be\label{prop3.2-1}
	\int_{\om}t^{1-2\si}|\nabla \over{W}|^2\,\d X	 \leq \frac{C_{1}}{8} r_0^{2}\quad \text{ and }\quad  \over{W}|_{\pa'' \om}=W.	\ee
	To justify this, we first note that $\pa'' \om$ is compact and smooth (mollifying the singularities of $\pa'' \om$ if necessary), thus any  open cover $\{U_j\}$ of $\pa'' \om$ and the associated collection $\{\Psi_j\}$ of smooth maps from $\B_1$ onto the sets $U_j$  are finite collections, say $1 \leq j \leq k$. If $\{\omega_j\}$ is a partition of unity for $\pa'' \om$ subordinate to $\{U_j\}$,
	we  define $\theta_j W$ on $\Rn=\pa \Rp$ by
	$$
	\theta_j W(x)= \begin{cases}(\omega_j W)(\Psi_j(x, 0)) & \text {if }\,|x|<1, \\ 0 & \,\text {otherwise, }\end{cases}
	$$
	where $x=(x_1,\ldots,x_n)$. Then \eqref{prop3.2-1} follows directly from  the proof of  Proposition 2.1  in \cite{JLX}. We fix the value of $C_{1}$ from now on and  the value of $r_{0}=r_{0}(n,\si, A_{0})$ is  determined in the following.

	First  it follows from \eqref{16}-\eqref{prop3.2-1} that if  $r_{0}>0$ is chosen small enough, then
	\begin{align}
	I_{K,\om}(\over{W}) & \leq \frac{1}{2} \int_{\om}t^{1-2\si}|\nabla \over{W}|^2\,\d X+ A_0 C_0(n,\si)\Big(\int_{\om}t^{1-2\si}|\nabla \over{W}|^2\,\d X\Big)^{(2_{\si}^{*}-\tau)/2}\notag\\
	& \leq \frac{4}{5} 	\int_{\om}t^{1-2\si}|\nabla \over{W}|^2\,\d X \leq \frac{C_{1}}{10} r^{2}.\label{prop3.2-2}
	\end{align}
	Observe that for any $U$ satisfying $C_{1} r^{2}/2 \leq \int_{\om}t^{1-2\si}|\nabla U|^2\,\d X \leq 2C_{1} r_{0}^{2}$,
 we  derive from \eqref{16}  that
	\begin{align*}
	I_{K,\om}(U)  \geq &\frac{1}{2} \int_{\om}t^{1-2\si}|\nabla U|^{2}\,\d X-A_{0} C_{0}(n,\si)\Big(\int_{\om}t^{1-2\si}|\nabla U|^{2}\,\d X\Big)^{(2^{*}_{\si}-\tau)/2} \\
	\geq &\frac{1}{2} \int_{\om}t^{1-2\si}|\nabla U|^{2}\, \d X\\&-A_{0} C_{0}(n,\si)(2 C_{1} r_{0}^{2})^{2\si /(n-2\si)-\tau/2} \int_{\om}t^{1-2\si}|\nabla U|^{2}\, \d X.
	\end{align*}
	Thus,  if  we choose $r_{0}>0$ to further satisfy	
	$A_{0} C_0(n,\si)(2 r_{0}^{2} C_{1})^{2\si /(n-2\si)-\tau/2} \leq 1/4$,
	then using \eqref{prop3.2-2} we have 	
$$
	I_{K,\om}(U)  \geq \frac{1}{4}  \int_{\om}t^{1-2\si}|\nabla U|^{2}\, \d X \geq \frac{1}{4}\Big(\frac{C_{1} r^{2}}{2}\Big) >I_{K,\om}(\over{W}).
$$
	Therefore, the minimizer only can be achieved in  $\{U \in D_{\om}:\,  \leq \int_{\om}t^{1-2\si}|\nabla U|^2\,\d X \leq C_{1} r^{2}/2\}$.
	
	Next we prove the existence of the minimizer.
	Write $U=V+\over{W}$, $V|_{\pa''\om}=0$,  $J_{K,\om}(V)=I_{K,\om}(U)=I_{K,\om}(V+\over{W})$.
	We only need to minimize $J_{K,\om}(V)$ for $	\int_{\om}t^{1-2\si}|\nabla V|^2\,\d X \leq 2 C_{1} r_{0}^{2}$ due to the above argument. It is easy to see that if $r_{0}$ is small enough, then $J_{K,\om}$ is strictly convex in the ball $\{V \in D_{\om}:\, V|_{\pa'' \om}=0, 	 \int_{\om}t^{1-2\si}|\nabla V|^2\,\d X\leq 2 C_{1} r_{0}^{2}\}$. Therefore it is standard to conclude the existence of a unique local minimizer $V_{W}$.
	
	Finally, set
	$U_{W}=V_{W}+\over{W}$, then $U_{W}$ is a local minimizer. As discussed above, $U_{W}$ satisfies
	$\int_{\om}t^{1-2\si}|\nabla U_W|^2\,\d X \leq C_{1}r^{2}/2$.	The uniqueness and the continuity of the map
	$W \mapsto  U_{W}$ follows from the strict local convexity of $J_{K,\om}$.  
\end{proof}

\bigskip

\noindent Z. Tang

\noindent School of Mathematical Sciences, Beijing Normal University\\ Beijing 100875, China\\
Email: \textsf{tangzw@bnu.edu.cn}

\medskip
\noindent H. Wang

\noindent School of Mathematical Sciences, Beijing Normal University\\
Beijing 100875, China\\[1mm]
Email: \textsf{hmw@mail.bnu.edu.cn}

\medskip

\noindent N. Zhou

\noindent School of Mathematical Sciences, Beijing Normal University \\Beijing 100875, China\\
Email: \textsf{nzhou@mail.bnu.edu.cn}

\end{document}